\documentclass[]{article}
\usepackage{amsmath}
\usepackage{amsfonts}
\usepackage{amssymb}
\usepackage{graphicx}
\usepackage{authblk}
\usepackage{latexsym,amsthm}
\usepackage{multirow}
\usepackage{hhline}
\usepackage{array}
\usepackage{setspace}
\usepackage{float}
\usepackage{mathtools}
\usepackage{enumitem}
\usepackage{subcaption}
\newtheorem{theorem}{Theorem}[section]
\newtheorem{lemma}[theorem]{Lemma}

\newtheorem{problem}[theorem]{Problem}

\title{Development of a new sixth order accurate compact scheme for two and three dimensional Helmholtz equation}
\author[1]{Neelesh Kumar}
\author[2]{Ritesh Kumar Dubey}
\affil[1]{Research Institute, SRM Institute of Science \& Technology, Chennai, India}
\affil[2]{Research Institute \& Department of Mathematics, SRM Institute of Science \& Technology, Chennai, India}
\affil[1]{neeleshkumar.r@res.srmuniv.ac.in}
\affil[2]{riteshkumar.d@res.srmuniv.ac.in}
\date{}
\begin{document}
	
	\maketitle
	
	\begin{abstract}
		In this work, a new compact sixth order accurate finite difference scheme for the two and three-dimensional Helmholtz equation is presented. The main significance of the proposed scheme is that its sixth order leading truncation error term does not explicitly depend on the associated wave number. This makes the scheme robust to work for the Helmholtz equation even with large wave numbers. The convergence analysis of the new scheme is given.  Numerical results for various benchmark test problems are given to support the theoretical estimates. These numerical results confirm the accuracy and robustness of the proposed scheme.
	\end{abstract}	
	 Helmholtz equations, Wave number.
\section{Introduction}
\label{sec:1}
Consider the boundary value problem governed by the Helmholtz equation
\begin{equation}\label{eq1}
\nabla^2 u(\textbf{r})+K^2u(\textbf{r})=f(\textbf{r}),~~ \textbf{r}\in\Omega
\end{equation}
with some associated boundary conditions, where $K$ is the wave number which is constant. The function $f$ and the solution value $u$ are assumed to be sufficiently smooth and they have required continuous and bounded derivatives. $f$ represents a harmonic source term. In this case, it is also assumed that the source function $f$ and its required derivatives are known explicitly.

The Helmholtz equation is a time-harmonic solution of the wave equation. There are many physical phenomena which are described by the Helmholtz equation. Some of these include water wave propagation, acoustic wave scattering, electromagnetic wave propagation. There has been growing interest in constructing compact highly accurate difference schemes for solving partial differential equations \cite{fang2002finite,peiro2005finite,ames2014numerical, bieniasz2008two,chawla1979sixth}, particularly to solve the Helmholtz equations \cite{nabavi2007new,harari1995accurate,shaw1988integral,mehdizadeh2003investigation,harari1991finite}. Various numerical techniques for solving the boundary value problems modeled by the Helmholtz equation have been developed using different approaches, to name a few of them as the finite-difference methods \cite{nabavi2007new,sutmann2007compact,singer1998high}, the boundary element method \cite{shaw1988integral}, the finite-element methods \cite{ainsworth2004discrete,babuska1997pollution} and the spectral-element methods \cite{mehdizadeh2003investigation,shen2005spectral}.

It is observed that the quality of the numerical solution of the Helmholtz problems depends on the wave number. Thus, the numerical solution for the Helmholtz equation (\ref{eq1}) is highly oscillatory in the case of large wave number $K$. Bayliss et al. \cite{bayliss1985accuracy}, Babuska and Sauter \cite{babuska1997pollution} discussed that for a given accuracy, the number of grid points increases with increasing the wave number. Thus for a given mesh, the numerical errors are developed with the wave number. There are many numerical schemes which can provide accurate numerical results for large wave number. On the other hand, compact high order finite difference techniques for the numerical solution of the Helmholtz equation have been used widely since they provide high accuracy with less computational costs. Nabavi et al. \cite{nabavi2007new}, Singer and Turkel \cite{singer2006sixth} discussed the compact sixth order finite difference schemes for the Helmholtz equations in two dimensions. A compact sixth order accurate scheme for the three dimensional Helmholtz equation is presented by Sutmann \cite{sutmann2006high} for constant $K$ and by Turkel et al. \cite{turkel2013compact} for variable wave number. Fu \cite{fu2008compact} developed fourth order accurate finite difference schemes with high wave number $K$ for sufficiently small $Kh$, where $h$ is a spatial mesh step size. In this work, a new compact sixth order finite difference scheme for (\ref{eq1}) is presented which is also suitable for high wave numbers. The main significance of this work is that the leading truncation error of the proposed scheme does not depend explicitly on the wave number $K$. It is shown that the scheme is uniquely solvable. The convergence analysis for the proposed scheme is also discussed. The proposed scheme is compared with the standard sixth order schemes \cite{nabavi2007new,sutmann2007compact}. The resulting scheme is also compact \cite{settle2013derivation} in the sense that it involves only the patch of cells adjacent to a given node in the mesh. We expect that the proposed scheme will be useful for the Helmholtz equation with large wave numbers.

The paper is organized as follows. In Section \ref{sec:2}, a new compact sixth-order scheme for (\ref{eq1}) is derived. In Section \ref{sec:3}, it is proved that the proposed scheme is uniquely solvable. The error bound is also established for the proposed scheme and the standard sixth order scheme \cite{nabavi2007new}. In Section \ref{sec:4}, the numerical results for various benchmark test problems are given to justify the accuracy of the new scheme. At the last, we concluded our work in Section \ref{sec:5}.

\section{Formulation of a new sixth-order accurate approximation}
\label{sec:2}
In the past, Nabavi \cite{nabavi2007new} and others \cite{sutmann2007compact,singer2006sixth} discussed compact sixth order approximations for the Helmholtz equations. The leading truncation error term in a given difference scheme depends on the solution value $u$, the source function $f$ and the wave number $K$ which shows that the quality of the numerical solution depends significantly on the wave number $K$. 
In this work, our effort is to make the scheme whose truncation error is less relevance to the wave number $K$. 
\subsection{Two-Dimensional case}
Consider the two-dimensional Helmholtz equation
\begin{equation}\label{eq15}
\nabla^2u+K^2u(x,y)=f(x,y), ~~(x,y)\in\Omega
\end{equation}
with some Dirichlet boundary
\begin{equation}\label{bc1}
u|_{\partial\Omega}={\Phi}(x,y),~(x,y)\in\partial\Omega
\end{equation}
where $\Omega=[a,b]\times[c,d]\in\mathbb{R}^2$ with its boundary $\partial\Omega$. 
Assuming $u$ to be sufficiently smooth, consider the discretization of rectangular domain $\Omega$ on a compact stencil with mesh step size $\Delta x$ and $\Delta y$ in $x$- \& $y$- directions respectively. 
The discrete grid points along $x$- \& $y$- directions are defined as $x_i=a+ih, y_j=c+jh, i,j=0(1)N.$ 

For ease of notations, we denote the following notations in further derivation
\begin{equation*}
\begin{split}
\partial_x^{r_1}\partial_y^{r_2}\varphi=&\frac{\partial^{r_1+r_2}\varphi}{\partial x^{r_1}\partial y^{r_2}}, r_i\geq 0, r_i\in\mathbb{Z}, \varphi\in\{u,f\},
\nabla^n=\partial_x^n+\partial_y^n, n=2, 4, 6, 8,\\
\mathcal{D}u_{i,j}=&u_{i+1,j+1}+u_{i-1,j+1}+u_{i+1,j-1}+u_{i-1,j-1},\\ \mathcal{F}u_{i,j}=&u_{i+1,j}+u_{i-1,j}+u_{i,j+1}+u_{i,j-1}.
\end{split}
\end{equation*}
%
%
%

The operators used to approximate derivatives with a minimum stencil width are known as compact. The compact scheme amongst the various finite difference replacements has received more attention due to a minimal width of stencils and easy computations. In contrast, high-order difference schemes formulated with non-compact stencils yield a higher bandwidth of the iteration matrix which involves large arithmetic operations.

Using linear combination of minimum number of grid points, the following compact operators corresponding to the second derivatives of $\varphi\in\{u,f\}$ at the grid point $(x_i,y_j),i,j=0(1)N$ are given by
\begin{equation}\label{eq16}
\delta_x^2\varphi_{i,j}=\frac{\varphi_{i-1,j}-2\varphi_{i,j}+\varphi_{i+1,j}}{h^2},~~ \delta_y^2\varphi_{i,j}=\frac{\varphi_{i,j-1}-2\varphi_{i,j}+\varphi_{i,j+1}}{h^2}.
\end{equation}
Similarly, the composite operator for the mixed derivatives of $\varphi\in\{u,f\}$ is defined as
\begin{equation}\label{eq17}
\begin{split}
\delta_x^2\delta_y^2\varphi_{i,j}=\frac{\mathcal{D}\varphi_{i,j}+4\varphi_{i,j}-2\mathcal{F}\varphi_{i,j}}{h^4}.
\end{split}
\end{equation}
Then for sufficiently smooth $\varphi\in\{u,f\}$, Taylor series expansion to (\ref{eq16})-(\ref{eq17}) gives
\begin{equation}\label{eq18}
\delta_x^2\varphi=\partial_x^2\varphi+\frac{h^2}{12}\partial_x^4\varphi+\frac{h^4}{360}\partial_x^6\varphi+O(h^6),
\delta_y^2\varphi=\partial_y^2\varphi+\frac{h^2}{12}\partial_y^4\varphi+\frac{h^4}{360}\partial_y^6\varphi+O(h^6),
\end{equation}
\begin{equation}\label{eq18b}
\delta_x^2\delta_y^2\varphi=\partial_x^2\partial_y^2\varphi+\frac{h^2}{12}(\partial_x^4\partial_y^2+\partial_x^2\partial_y^4)\varphi+O(h^4).
\end{equation}
Let $u_{i,j}$ be the discrete value of the solution value $u$ which satisfies (\ref{eq15}), then
substituting the $O(h^2)$- approximations of derivatives from (\ref{eq18}) into  (\ref{eq15}), we get
\begin{equation}\label{eq19}
\delta_x^2u_{i,j}+\delta_y^2u_{i,j}+K^2u_{i,j}+T_{i,j}=f_{i,j},
\end{equation}
where $f_{i,j}=f(x_i,y_j)$ and $T_{i,j}$ is the leading truncation error, and is given by
\begin{equation}\label{eq20}
T_{i,j}=-\frac{h^2}{12}(\partial_x^4+\partial_y^4)u_{i,j}-\frac{h^4}{360}(\partial_x^6+\partial_y^6)u_{i,j}+O(h^6).
\end{equation}
For getting sixth-order accurate schemes, we have to find fourth and second order approximation of respectively fourth and sixth derivatives of solution values. 

To achieve this, differentiating  (\ref{eq15}) twice with respect to $x,y$ and solving for fourth derivatives, we get
\begin{equation}\label{eq21}
\begin{split}
\partial_x^4u=\partial_x^2(f-K^2u)-\partial_x^2\partial_y^2u,~
\partial_y^4u=\partial_y^2(f-K^2u)-\partial_x^2\partial_y^2u,
\end{split}
\end{equation}
Further differentiating (\ref{eq21}) twice with respect to $x,y$
\begin{equation}\label{eq22}
\begin{split}
\partial_x^6u=\partial_x^4(f-K^2u)-\partial_x^4\partial_y^2u,~
\partial_y^6u=\partial_y^4(f-K^2u)-\partial_x^2\partial_y^4u.
\end{split}
\end{equation}
Also from (\ref{eq15}), we have
\begin{equation}\label{eq23}
\begin{split}
(\partial_x^4\partial_y^2+\partial_x^2\partial_y^4)u=\partial_x^2\partial_y^2(\partial_x^2u+\partial_y^2u)=\partial_x^2\partial_y^2f-K^2\partial_x^2\partial_y^2u.
\end{split}
\end{equation}
Using (\ref{eq18b}), we get $O(h^2)$-approximations of (\ref{eq23})
\begin{equation}\label{eq24}
\begin{split}
(\partial_x^4\partial_y^2+\partial_x^2\partial_y^4)u=\delta_x^2\delta_y^2f-K^2\delta_x^2\delta_y^2u+O(h^2).
\end{split}
\end{equation}
Substituting (\ref{eq24}) into (\ref{eq18b}), we will get $O(h^4)$-approximation of $\partial_x^2\partial_y^2u$ as
\begin{equation}\label{eq25}
\partial_x^2\partial_y^2u=\delta_x^2\delta_y^2u-\frac{h^2}{12}\delta_x^2\delta_y^2f+\frac{h^2K^2}{12}\delta_x^2\delta_y^2u+O(h^4).
\end{equation}
Now combining equations in (\ref{eq21}) and using (\ref{eq25}), we get $O(h^4)$-approximation of $\nabla^4u$ at the grid $(x_i,y_j)$ as
\begin{equation}\label{eq26D2}
\begin{split}
\nabla^4u_{i,j}=\nabla^2f-K^2(f-K^2u)-\left(2+\frac{h^2K^2}{6}\right)\delta_x^2\delta_y^2u+\frac{h^2}{6}\delta_x^2\delta_y^2f+O(h^4).
\end{split}
\end{equation}
Again combining both equations in (\ref{eq22}) and using (\ref{eq21}) and (\ref{eq23}), we have
\begin{equation}\label{eq27D2}
\nabla^6u=\nabla^4f-K^2\nabla^2f+K^4\nabla^2u+\partial_x^2\partial_y^2(3K^2u-f).
\end{equation}
From (\ref{eq15}), (\ref{eq18}) and (\ref{eq18b}), we get $O(h^2)$-approximations of $\nabla^6u$ at the grid $(x_i,y_j)$
\begin{equation}\label{eq28D2}
\begin{split}
\nabla^6u_{i,j}=&(\delta_x^2\partial_x^2+\delta_y^2\partial_y^2)f_{i,j}-K^2\nabla^2f_{i,j}+K^4f_{i,j}-K^6u_{i,j}\\
&+\delta_x^2\delta_y^2(3K^2u-f)_{i,j}+O(h^2).
\end{split}
\end{equation}
Substituting equations (\ref{eq26D2}) and (\ref{eq28D2}) into (\ref{eq20}), we have
\begin{equation}\label{eq29D2}
\begin{split}
T_{i,j}=&E_1(K^2u-f)_{i,j}+\delta_x^2\delta_y^2(E_2u+E_3f)_{i,j}+E_4\nabla^2f_{i,j}\\
&+E_5(\delta_x^2\partial_x^2+\delta_y^2\partial_y^2)f_{i,j}+O(h^6),
\end{split}
\end{equation}
where,
\begin{equation*}
\begin{split}
E_1=&\left(\frac{K^4h^4}{360}-\frac{K^2h^2}{12}\right),
E_2=\left(\frac{h^2}{6}+\frac{K^2h^4}{180}\right),
E_3=-\frac{h^4}{90},\\
E_4=&\left(\frac{K^2h^4}{360}-\frac{h^2}{12}\right), E_5=-\frac{h^4}{360}.
\end{split}
\end{equation*}
Substituting $T_{i,j}$ from (\ref{eq29D2}) into (\ref{eq19}), we get the standard sixth-order scheme \cite{nabavi2007new} for (\ref{eq15})
\begin{equation}\label{eq26}
\begin{split}
&\left(\delta_x^2+\delta_y^2+\frac{h^2}{6}\left(1+\frac{K^2h^2}{30}\right)\delta_x^2\delta_y^2\right)u_{i,j}+\left(1-\frac{K^2h^2}{12}\left(1-\frac{K^2h^2}{30}\right)\right)K^2u_{i,j}\\
=&\left(1-\frac{K^2h^2}{12}\left(1-\frac{K^2h^2}{30}\right)\right)f_{i,j}+\frac{h^2}{12}\left(1-\frac{K^2h^2}{30}\right)\left(\partial_x^2+\partial_y^2\right)f_{i,j}\\
&+\frac{h^4}{360}\left(\delta_x^2\partial_x^2+\delta_y^2\partial_y^2\right)f_{i,j}+\frac{h^4}{90}\delta_x^2\delta_y^2f_{i,j}+T_1^{(2)},
\end{split}
\end{equation}
where the leading truncation error term is given by
\begin{equation}\label{eq27}
\begin{split}
T_1^{(2)}=&\frac{K^2h^6}{2160}(\partial_x^4\partial_y^2+\partial_x^2\partial_y^4)u+\frac{h^6}{20160}(\partial_x^8+\partial_y^8)u+\frac{h^6}{2160}(\partial_x^6\partial_y^2+\partial_x^2\partial_y^6)u\\
&+\frac{h^6}{864}\partial_x^4\partial_y^4u-\frac{h^6}{4320}\partial_x^6\partial_y^6f-\frac{h^6}{1080}(\partial_x^4\partial_y^2+\partial_x^2\partial_y^4)f+O(h^8).
\end{split}
\end{equation}
From (\ref{eq27}), it is clear that $T_1^{(2)}$ depends upon the solution value $u$, the source function $f$ and it also explicitly depends on the wave number $K$.

For simplicity, we omit subscripts $i,j$ in $T_1^{(2)}$ and in further derivation as well. 

Now from (\ref{eq15}) and (\ref{eq21}), we have
\begin{equation}\label{eq28}
\begin{split}
(\partial_x^6\partial_y^2+\partial_x^2\partial_y^6)u=&\partial_x^2\partial_y^2(\partial_x^4u+\partial_y^4u)=\partial_x^2\partial_y^2(\nabla^2f-K^2\nabla^2u-2\partial_x^2\partial_y^2u)\\
=&(\partial_x^4\partial_y^2+\partial_x^2\partial_y^4)f-K^2\partial_x^2\partial_y^2(f-K^2u)-2\partial_x^4\partial_y^4u.
\end{split}
\end{equation}
Now, differentiating  (\ref{eq22}) twice w.r.t. $x,y$, and combining both, we get
\begin{equation}\label{eq29}
\begin{split}
\nabla^8u=\nabla^6f-K^2\nabla^6u-(\partial_x^6\partial_y^2+\partial_x^2\partial_y^6)u.
\end{split}
\end{equation}
Using (\ref{eq27D2}) and (\ref{eq28}) into (\ref{eq29}) , we have
\begin{equation}\label{eq30}
\begin{split}
\nabla^8u=&\nabla^6f-(\partial_x^4\partial_y^2+\partial_x^2\partial_y^4)f-K^2\nabla^4f+K^4\nabla^2f+2K^2\partial_x^2\partial_y^2(f-2K^2u)\\
&-K^6\nabla^2u+2\partial_x^4\partial_y^4u.
\end{split}
\end{equation}
Substituting (\ref{eq23}), (\ref{eq28}), and (\ref{eq30}) into (\ref{eq27})
\begin{equation}\label{eq312D}
\begin{split}
T_1^{(2)}=&\gamma_1K^2h^6(K^2\nabla^2f+2\partial_x^2\partial_y^2f
-\nabla^4f-K^4\nabla^2u-4K^2\partial_x^2\partial_y^2u)\\
&+\gamma_2h^6\partial_x^4\partial_y^4u-\frac{11\gamma_1h^6}{3}\nabla^6f
-\frac{31\gamma_1h^6}{3}(\partial_x^4\partial_y^2+\partial_x^2\partial_y^4)f+O(h^8),
\end{split}
\end{equation}
where $\gamma_1=1/20160, \gamma_2=1/3024$. Using the $O(h^2)$- approximations of the second order derivatives from  (\ref{eq18})-(\ref{eq18b}) into (\ref{eq312D}), we have
\begin{equation}\label{eq31}
\begin{split}
T_1^{(2)}=&\gamma_1K^2h^6(K^2(\delta_x^2+\delta_y^2)f+2\delta_x^2\delta_y^2f
-(\delta_x^2\partial_x^2f+\delta_y^2\partial_y^2f)-K^4(\delta_x^2+\delta_y^2)u\\
-&4K^2\delta_x^2\delta_y^2u)+\gamma_2h^6\partial_x^4\partial_y^4u-\frac{\gamma_1h^6}{3}(11\nabla^6f
+31(\partial_x^4\partial_y^2+\partial_x^2\partial_y^4)f)+O(h^8).
\end{split}
\end{equation}
Substituting $T_1^{(2)}$ from (\ref{eq31}) into (\ref{eq26}), we get a new sixth order compact scheme for the three dimensional Helmholtz equation (\ref{eq15})
\begin{equation}\label{eq32}
\begin{split}
\alpha_1\left(\delta_x^2+\delta_y^2\right)u_{i,j}+\alpha_2\delta_x^2\delta_y^2u_{i,j}+K^2\alpha_3u_{i,j}=\beta_1f_{i,j}+\beta_2\left(\delta_x^2+\delta_y^2\right)f_{i,j}\\
+\beta_3\delta_x^2\delta_y^2f_{i,j}+\beta_4(\partial_x^2+\partial_y^2)f_{i,j}+\beta_5(\delta_x^2\partial_x^2+\delta_y^2\partial_y^2)f_{i,j}+T_2^{(2)},
\end{split}
\end{equation}
where
\begin{equation}\label{eq33}
\begin{split}
\alpha_1=&\left(1+\frac{K^6h^6}{20160}\right), \alpha_2=\frac{h^2}{6}\left(1+\frac{K^2h^2}{30}+\frac{K^4h^4}{840}\right),\\ \alpha_3=&\left(1-\frac{K^2h^2}{12}+\frac{K^4h^4}{360}\right), \beta_1=1-\frac{K^2h^2}{12}+\frac{K^4h^4}{360},~ \beta_2=\frac{K^4h^6}{20160},\\
\beta_3=&\frac{h^4}{90}\left(1+\frac{K^2h^2}{112}\right), \beta_4=\frac{h^2}{12}\left(1-\frac{K^2h^2}{30}\right), ~\beta_5=\frac{h^4}{360}\left(1-\frac{K^2h^2}{56}\right).
\end{split}
\end{equation}
And the leading truncation error term to new sixth order scheme is given by
\begin{equation}\label{eq34}
\begin{split}
T_2^{(2)}=\frac{h^6}{3024}\partial_x^4\partial_y^4u-\frac{11h^6}{60480}(\partial_x^6+\partial_y^6)f-\frac{31h^6}{60480}(\partial_x^4\partial_y^2+\partial_x^2\partial_y^4)f.
\end{split}
\end{equation}

From (\ref{eq34}), It is clear that the truncation error $T^{(2)}_2$ does not explicitly depend on the wave number $K$. It depends only on the solution value $u$, the source function $f$.

It is noted that the source function taken here is known at each grid point. Therefore right side of equation (\ref{eq32}) is fully known. If the source function $f$ in equation  (\ref{eq15}) is unknown, then we need at most fourth order accurate approximation of $\partial_x^2f$ and $\partial_y^2f$ which can be obtained by enlarging the cell stencil. Then the resultant system of equation can be solved using some direct or iterative methods \cite{Thomas2013,young2014iterative,hackbusch1994iterative}.
\subsection{Three-Dimensional case}
We consider the Helmholtz equation in three dimensions 
\begin{equation}\label{eq3D1}
\nabla^2u+K^2u(x,y,z)=f(x,y,z), (x,y,z)\in \Omega.
\end{equation}
subject to the Dirichlet boundary condition $u(x,y,z)={\Psi}(x,y,z),~(x,y,z)\in\partial\Omega$,
where $u=u(x,y,z)$ is a sufficiently differentiable function and $\Omega=[a_1,a_2]\times[b_1,b_2]\times[c_1,c_2]\in\mathbb{R}^3$ with its boundary $\partial\Omega$. We consider the discretization of $\Omega$ on a compact 27-point cell stencil with uniform mesh step size $h$ in each direction of the coordinate axis. The internal grid points along $x$-, $y$- \& $z$- directions are $x_i=a_1+ih, y_j=b_1+jh, z_k=c_1+kh, i,j,k=0(1)N,$ and $u_{i,j,k}=u(x_i,y_j,z_k)$ be the discrete value of the solution $u$ at this grid.

For ease of notations, let we  denote some index sets $\Lambda_0=\{0,2\}, \Lambda_1=\{0,2,6\}, \Lambda_2=\{0,2,4\}, \Lambda_3=\{0,4\}$, and $\Lambda_4=\{2,4\}$. Also we denote some operators for partial derivatives by
\begin{equation*}
\begin{split}
\partial_x^{r_1}\partial_y^{r_2}\partial_z^{r_3}\varphi=&\frac{\partial^{r_1+r_2+r_3}\varphi}{\partial x^{r_1}\partial y^{r_2}\partial z^{r_3}}, r_i\geq 0, r_i\in \mathbb{Z}, \varphi\in\{u,f\},\\
\nabla^n= &\partial_x^n+\partial_y^n+\partial_z^n, n=2, 4, 6, 8.
\end{split}
\end{equation*}
The following standard finite difference compact operators for the second derivatives of $\varphi\in\{u,f\}$ at the grid point $(x_i,y_j,z_k)$ are given by
\begin{equation}\label{op1}
\begin{split}
\delta_x^2\varphi_{i,j,k}=&\frac{\varphi_{i-1,j,k}-2\varphi_{i,j,k}+\varphi_{i+1,j,k}}{h^2},~ \delta_y^2\varphi_{i,j,k}=\frac{\varphi_{i,j-1,k}-2\varphi_{i,j,k}+\varphi_{i,j+1,k}}{h^2},\\
\delta_z^2\varphi_{i,j,k}=&\frac{\varphi_{i,j,k-1}-2\varphi_{i,j,k}+\varphi_{i,j,k+1}}{h^2}.
\end{split}
\end{equation}
Thus for sufficiently smooth $\varphi\in\{u,f\}$, Taylor series expansion to (\ref{op1}) gives
\begin{equation}\label{op2}
\begin{split}
\partial_x^2\varphi_{i,j,k}=&\delta_x^2\varphi_{i,j,k}+O(h^2),\partial_y^2\varphi_{i,j,k}=\delta_y^2\varphi_{i,j,k}+O(h^2),\\
\partial_z^2\varphi_{i,j,k}=&\delta_z^2\varphi_{i,j,k}+O(h^2).
\end{split}
\end{equation}
From (\ref{op1}), The composite operators for the mixed derivatives of $\varphi\in\{u,f\}$
\begin{equation}\label{op3}
\begin{split}
\partial_x^2\partial_y^2\varphi_{i,j,k}=&\delta_x^2\delta_y^2\varphi_{i,j,k}+O(h^2),~ \partial_y^2\partial_z^2\varphi_{i,j,k}=\delta_y^2\delta_z^2\varphi_{i,j,k}+O(h^2),\\ \partial_x^2\partial_z^2\varphi_{i,j,k}=&\delta_x^2\delta_z^2\varphi_{i,j,k}+O(h^2),~ \partial_x^2\partial_y^2\partial_z^2\varphi_{i,j,k}=\delta_x^2\delta_y^2\delta_z^2\varphi_{i,j,k}+O(h^2).
\end{split}
\end{equation}
Let $u_{i,j,k}$ satisfies (\ref{eq3D1}) at the grid $(x_i,y_j,z_k)$, then from (\ref{op2}), we have
\begin{equation}\label{eq3D2}
(\delta_x^2+\delta_y^2+\delta_z^2)u_{i,j,k}+K^2u_{i,j,k}+T_{i,j,k}=f_{i,j,k},
\end{equation}
where $f_{i,j,k}=f(x_i,y_j,z_k)$ and the local truncation error is given by
\begin{equation}\label{eq3D3}
\begin{split}
T_{i,j,k}&=\nabla^2u_{i,j,k}-(\delta_x^2+\delta_y^2+\delta_z^2)u_{i,j,k}=-\frac{h^2}{12}\nabla^4u_{i,j,k}-\frac{h^4}{360}\nabla^6u_{i,j,k}+O(h^6).
\end{split}
\end{equation}
In order to get $O(h^4)$-approximation of $\nabla^4u$, $O(h^2)$-approximation of $\nabla^6u$, differentiating (\ref{eq3D1}) twice w.r.t. $x, y, z$ and solving for fourth derivatives, we get
\begin{equation}\label{eq3D5}
\begin{split}
\partial_x^4u=&\partial_x^2(f-K^2u)-\partial_x^2(\partial_y^2+\partial_z^2)u,
\partial_y^4u=\partial_y^2(f-K^2u)-\partial_y^2(\partial_x^2+\partial_z^2)u,\\
\partial_z^4u=&\partial_z^2(f-K^2u)-\partial_z^2(\partial_x^2+\partial_y^2)u.
\end{split}
\end{equation}
Again by differentiating forcing function of (\ref{eq3D1}), we have mixed derivatives as
\begin{equation}\label{eq3D10}
\begin{split}
\partial_x^2\partial_y^2f=(\partial_x^4\partial_y^2+\partial_x^2\partial_y^4)u+\partial_x^2\partial_y^2\partial_z^2u+K^2\partial_x^2\partial_y^2u,\\
\partial_y^2\partial_z^2f=(\partial_y^4\partial_z^2+\partial_y^2\partial_z^4)u+\partial_x^2\partial_y^2\partial_z^2u+K^2\partial_y^2\partial_z^2u,\\
\partial_x^2\partial_z^2f=(\partial_x^4\partial_z^2+\partial_x^2\partial_z^4)u+\partial_x^2\partial_y^2\partial_z^2u+K^2\partial_x^2\partial_z^2u.
\end{split}
\end{equation}
Further differentiating (\ref{eq3D5}) twice w.r.t. $x, y, z$ for sixth derivatives
\begin{equation}\label{eq3D14}
\begin{split}
\partial_x^6u=&\partial_x^4(f-K^2u)-\partial_x^4\partial_y^2u-\partial_x^4\partial_z^2u,
\partial_y^6u=\partial_y^4(f-K^2u)-\partial_x^2\partial_y^4u-\partial_y^4\partial_z^2u,\\
\partial_z^6u=&\partial_z^4(f-K^2u)-\partial_x^2\partial_z^4u-\partial_y^2\partial_z^4u.
\end{split}
\end{equation}
Combining the equations in (\ref{eq3D10}), then we have
\begin{equation}\label{eq3D11}
\begin{split}
\sum_{\substack{i,j,k\in\Lambda_2\\i+j+k=6\\i\neq j}}{\partial_x^i\partial_y^j\partial_z^ku}=\sum_{\substack{i,j,k\in\Lambda_0\\i+j+k=4}}{\partial_x^i\partial_y^j\partial_z^k(f-K^2u)}-3\partial_x^2\partial_y^2\partial_z^2u.
\end{split}
\end{equation}
Using $O(h^2)$-approximations from (\ref{op3}) into (\ref{eq3D11}), we get 
\begin{equation}\label{eq3D111}
\begin{split}
\sum_{\substack{i,j,k\in\Lambda_2\\i+j+k=6\\i\neq j}}{\partial_x^i\partial_y^j\partial_z^ku}
=\sum_{\substack{i,j,k\in\Lambda_0\\i+j+k=4}}{\delta_x^i\delta_y^j\delta_z^k(f-K^2u)}-3\delta_x^2\delta_y^2\delta_z^2u+O(h^2).
\end{split}
\end{equation}
For $O(h^4)$- approximations of $\partial_x^2\partial_y^2u,\partial_y^2\partial_z^2u,\partial_x^2\partial_z^2u$, we consider
\begin{equation}\label{eq3D7}
\begin{split}
\partial_x^2\partial_y^2u=\delta_x^2\delta_y^2u+T_1,
\partial_y^2\partial_z^2u=\delta_y^2\delta_z^2u+T_2,
\partial_x^2\partial_z^2u=\delta_x^2\delta_z^2u+T_3,
\end{split}
\end{equation}
where, $T_i's$ can be obtained by using the Taylor series expansion to (\ref{eq3D7})
\begin{equation}\label{eq3D8}
\begin{split}
T_1=&-\frac{h^2}{12}(\partial_x^4\partial_y^2u+\partial_x^2\partial_y^4u)+O(h^4),
T_2=-\frac{h^2}{12}(\partial_y^4\partial_z^2u+\partial_y^2\partial_z^4u)+O(h^4),\\
T_3=&-\frac{h^2}{12}(\partial_x^4\partial_z^2u+\partial_x^2\partial_z^4u)+O(h^4).
\end{split}
\end{equation}
Adding equations in (\ref{eq3D7}) and using (\ref{eq3D8}), we have
\begin{equation}\label{eq3D9}
\sum_{\substack{i,j,k\in\Lambda_0\\i+j+k=4}}{\partial_x^i\partial_y^j\partial_z^ku}=\sum_{\substack{i,j,k\in\Lambda_0\\i+j+k=4}}{\delta_x^i\delta_y^j\delta_z^ku}-\frac{h^2}{12}\sum_{\substack{i,j,k\in\Lambda_2\\i+j+k=6\\i\neq j}}{\partial_x^i\partial_y^j\partial_z^ku}+O(h^4).
\end{equation}
Now using $O(h^2)$-approximation from (\ref{eq3D111}) into (\ref{eq3D9}), we get $O(h^4)$ approximation
\begin{equation}\label{eq3D6}
\begin{split}
\sum_{\substack{i,j,k\in\Lambda_0\\i+j+k=4}}{\partial_x^i\partial_y^j\partial_z^ku}&=\sum_{\substack{i,j,k\in\Lambda_0\\i+j+k=4}}{\delta_x^i\delta_y^j\delta_z^ku}-\frac{h^2}{12}\sum_{\substack{i,j,k\in\Lambda_0\\i+j+k=4}}{\delta_x^i\delta_y^j\delta_z^k(f-K^2u)}\\
&+\frac{h^2}{4}\delta_x^2\delta_y^2\delta_z^2u+O(h^4).
\end{split}
\end{equation}
Adding equations in (\ref{eq3D5}) and using (\ref{eq3D6}), we have $\nabla^4u$ at the grid $(x_i,y_j,z_k)$
\begin{equation}\label{eq3D66}
\begin{split}
\nabla^4u_{i,j,k}=&\nabla^2f_{i,j,k}-K^2f_{i,j,k}+K^4u_{i,j,k}-\frac{h^2}{2}\delta_x^2\delta_y^2\delta_z^2u_{i,j,k}\\
-&2\sum_{\substack{i,j,k\in\Lambda_0\\i+j+k=4}}{\delta_x^i\delta_y^j\delta_z^ku_{i,j,k}}+\frac{h^2}{6}\sum_{\substack{i,j,k\in\Lambda_0\\i+j+k=4}}{\delta_x^i\delta_y^j\delta_z^k(f-K^2u)_{i,j,k}}+O(h^4).
\end{split}
\end{equation}
To find $\nabla^6u$, adding the equations in (\ref{eq3D14}) and using (\ref{eq3D5}) and (\ref{eq3D11})
\begin{equation}\label{eq3D15}
\begin{split}
\nabla^6u&=\nabla^4f-K^2\nabla^2f+K^4\nabla^2u+\sum_{\substack{i,j,k\in\Lambda_0\\i+j+k=4}}{\partial_x^i\partial_y^j\partial_z^k}(3K^2u-f)+3\partial_x^2\partial_y^2\partial_z^2u.
\end{split}
\end{equation}
Substituting approximations from (\ref{op2})-(\ref{op3}) into (\ref{eq3D15}), we have $O(h^2)$ approximations of $\nabla^6u$ at the grid $(x_i,y_j,z_k)$
\begin{equation}\label{eq3D13}
\begin{split}
\nabla^6u_{i,j,k}=&(\delta_x^2\partial_x^2+\delta_y^2\partial_y^2+\delta_z^2\partial_z^2)f_{i,j,k}-K^2\nabla^2f_{i,j,k}+K^4f_{i,j,k}-K^6u_{i,j,k}\\
&+3\delta_x^2\delta_y^2\delta_z^2u_{i,j,k}+\sum_{\substack{i,j,k\in\Lambda_0\\i+j+k=4}}{\delta_x^i\delta_y^j\delta_z^k(3K^2u-f)_{i,j,k}}+O(h^2).
\end{split}
\end{equation}
Finally, substituting $O(h^4)$-approximation of $\nabla^4u_{i,j,k}$ from (\ref{eq3D66}) and $O(h^2)$-approximation of $\nabla^6u_{i,j,k}$ from (\ref{eq3D13}) into (\ref{eq3D3}), we get sixth order leading truncation error
\begin{equation}\label{eq3D16}
\begin{split}
T_{i,j,k}=&E_1(K^2u-f)_{i,j,k}+\left(\delta_x^2\delta_y^2+\delta_y^2\delta_z^2+\delta_x^2\delta_z^2\right)(E_2u+E_3f)_{i,j,k}\\
+&E_6\delta_x^2\delta_y^2\delta_z^2u_{i,j,k}+E_4\nabla^2f_{i,j,k}+E_5(\delta_x^2\partial_x^2+\delta_y^2\partial_y^2+\delta_z^2\partial_z^2)f_{i,j,k}+O(h^6),
\end{split}
\end{equation}
where,
\begin{equation*}
\begin{split}
E_2=&\left(\frac{h^2}{6}+\frac{K^2h^4}{180}\right),
E_6=\frac{h^4}{30},
E_1=\left(\frac{K^4h^4}{360}-\frac{K^2h^2}{12}\right),
E_5=-\frac{h^4}{360},\\
E_4=&\left(\frac{K^2h^4}{360}-\frac{h^2}{12}\right),
E_3=-\frac{h^4}{90}.
\end{split}
\end{equation*}
Substituting $T_{i,j,k}$ from (\ref{eq3D16}) into (\ref{eq3D2}), we get sixth order scheme \cite{sutmann2007compact}
\begin{equation}\label{eq3D17}
\begin{split}
(\delta_x^2+\delta_y^2+\delta_z^2)u_{i,j,k}+&A_1(\delta_x^2\delta_y^2+\delta_y^2\delta_z^2+\delta_x^2\delta_z^2)u_{i,j,k}+A_2\delta_x^2\delta_y^2\delta_z^2u_{i,j,k}+A_3u_{i,j,k}\\
=&B_1f_{i,j,k}+B_2(\delta_x^2\delta_y^2+\delta_y^2\delta_z^2+\delta_x^2\delta_z^2)f_{i,j,k}+B_3\nabla^2f_{i,j,k}\\
&+B_4(\delta_x^2\partial_x^2+\delta_y^2\partial_y^2+\delta_z^2\partial_z^2)f_{i,j,k}+T^{(3)}_1,
\end{split}
\end{equation}
where,
\begin{equation}\label{eq3D18}
\begin{split}
T^{(3)}_1=&\alpha_1^{(1)}\nabla^8u-\alpha_2^{(1)}\nabla^6f+\alpha_3^{(1)}\sum_{\substack{i,j,k\in\Lambda_1\\i+j+k=8}}{\partial_x^i\partial_y^j\partial_z^ku}+K^2\alpha_3^{(1)}\sum_{\substack{i,j,k\in\Lambda_2\\i+j+k=6\\i\neq j}}{\partial_x^i\partial_y^j\partial_z^ku}\\
&+\alpha_4^{(1)}\sum_{\substack{i,j,k\in\Lambda_3\\i+j+k=8}}{\partial_x^i\partial_y^j\partial_z^ku}+\alpha_5^{(1)}\sum_{\substack{i,j,k\in\Lambda_4\\i+j+k=8}}{\partial_x^i\partial_y^j\partial_z^ku}-\alpha_6^{(1)}\sum_{\substack{i,j,k\in\Lambda_2\\i+j+k=6\\i\neq j}}{\partial_x^i\partial_y^j\partial_z^kf}.
\end{split}
\end{equation}
The coefficients of the scheme  and its truncation error term are given by
\begin{equation*}
\begin{split}
A_1=&\frac{h^2}{6}\left(1+\frac{K^2h^2}{30}\right),
A_2=\frac{h^4}{30},
A_3=K^2\left(1-\frac{K^2h^2}{12}+\frac{K^4h^4}{360}\right),\\
B_1=&1-\frac{K^2h^2}{12}+\frac{K^4h^4}{360},
B_2=\frac{h^4}{90},
B_3=\frac{h^2}{12}\left(1-\frac{K^2h^2}{30}\right),
B_4=\frac{h^4}{360},\\
\alpha_1^{(1)}=&\dfrac{h^6}{20160}, \alpha_2^{(1)}=\dfrac{h^6}{4320}, \alpha_3^{(1)}=\dfrac{h^6}{2160}, \alpha_4^{(1)}=\dfrac{h^6}{864}, \alpha_5^{(1)}=\dfrac{h^6}{360}, \alpha_6^{(1)}=\dfrac{h^6}{1080}.
\end{split}
\end{equation*}
For the ease of notations, here and below we omit subscripts $i,j,k$.
Now from (\ref{eq3D1}), we have
\begin{equation}\label{eq3D19}
\sum_{\substack{i,j,k\in\Lambda_4\\i+j+k=8}}{\partial_x^i\partial_y^j\partial_z^ku}=\partial_x^2\partial_y^2\partial_z^2(\partial_x^2u+\partial_y^2u+\partial_z^2u)=\partial_x^2\partial_y^2\partial_z^2f-K^2\partial_x^2\partial_y^2\partial_z^2u.
\end{equation}
Now, differentiating  (\ref{eq3D14}) twice w.r.t. $x,y,z$, and combining them, we get
\begin{equation}\label{eq3D24}
\begin{split}
&\partial_x^8u=\partial_x^6(f-K^2u)-(\partial_x^6\partial_y^2+\partial_x^6\partial_z^2)u, \partial_y^8u=\partial_y^6(f-K^2u)-(\partial_x^2\partial_y^6+\partial_y^6\partial_z^2)u,\\
&\partial_z^8u=\partial_z^6(f-K^2u)-(\partial_x^2\partial_z^6+\partial_y^2\partial_z^6)u.
\end{split}
\end{equation}
Again from (\ref{eq3D1}), we have
\begin{equation}\label{eq3D22}
\begin{split}
&\partial_x^6\partial_y^2u=\partial_x^4\partial_y^2(f-K^2u-\partial_y^2u-\partial_z^2u),\partial_x^2\partial_y^6u=\partial_x^2\partial_y^4(f-K^2u-\partial_x^2u-\partial_z^2u),\\
&\partial_y^6\partial_z^2u=\partial_y^4\partial_z^2(f-K^2u-\partial_x^2u-\partial_z^2u),\partial_y^2\partial_z^6u=\partial_y^2\partial_z^4(f-K^2u-\partial_x^2u-\partial_y^2u),\\
&\partial_x^2\partial_z^6u=\partial_x^2\partial_z^4(f-K^2u-\partial_x^2u-\partial_y^2u),\partial_x^6\partial_z^2u=\partial_x^4\partial_z^2(f-K^2u-\partial_y^2u-\partial_z^2u).
\end{split}
\end{equation}
Combining equations given in (\ref{eq3D22}), we have
\begin{equation}\label{eq3D23}
\begin{split}
\sum_{\substack{i,j,k\in\Lambda_1\\i+j+k=8}}{\partial_x^i\partial_y^j\partial_z^ku}=\sum_{\substack{i,j,k\in\Lambda_2\\i+j+k=6\\i\neq j}}{\partial_x^i\partial_y^j\partial_z^k(f-K^2u)}-2\sum_{\substack{i,j,k\in\Lambda_4\\i+j+k=8}}{\partial_x^i\partial_y^j\partial_z^ku}\\
-2\sum_{\substack{i,j,k\in\Lambda_3\\i+j+k=8}}{\partial_x^i\partial_y^j\partial_z^ku}.
\end{split}
\end{equation}
Adding equations in (\ref{eq3D24}), we get
\begin{equation}\label{eq3D25}
\nabla^8u=\nabla^6(f-K^2u)-\sum_{\substack{i,j,k\in\Lambda_1\\i+j+k=8}}{\partial_x^i\partial_y^j\partial_z^ku}.
\end{equation}
Substituting (\ref{eq3D11}),(\ref{eq3D15}),(\ref{eq3D19}),(\ref{eq3D23}), and (\ref{eq3D25}) into (\ref{eq3D18}), we get
\begin{equation}\label{eq3D26}
\begin{split}
T^{(3)}_1=&\beta_1^{(1)}K^4\nabla^2(K^2u-f)+\partial_x^2\partial_y^2\partial_z^2(\beta_2^{(1)}K^2u+\beta_3^{(1)}f)+\beta_1^{(1)}K^2\nabla^4f\\
+&\beta_6^{(1)}\nabla^6f+2\beta_1^{(1)}K^2\sum_{\substack{i,j,k\in\Lambda_0\\i+j+k=4}}{\partial_x^i\partial_y^j\partial_z^k(2K^2u-f)}\\
+&\beta_4^{(1)}\sum_{\substack{i,j,k\in\Lambda_3\\i+j+k=8}}{\partial_x^i\partial_y^j\partial_z^ku}
+\beta_5^{(1)}\sum_{\substack{i,j,k\in\Lambda_2\\i+j+k=6\\i\neq j}}{\partial_x^i\partial_y^j\partial_z^kf},
\end{split}
\end{equation}
where,
$\beta_1^{(1)}=-\alpha_1^{(1)}, \beta_2^{(1)}=-(8\alpha_1^{(1)}-2\alpha_2^{(1)}+\alpha_4^{(1)}), \beta_3^{(1)}=(2\alpha_1^{(1)}-2\alpha_2^{(1)}+\alpha_4^{(1)}), \beta_4^{(1)}=(2\alpha_1^{(1)}-2\alpha_2^{(1)}+\alpha_3^{(1)}), \beta_5^{(1)}=(\alpha_2^{(1)}-\alpha_1^{(1)}-\alpha_5^{(1)})$, and $\beta_6^{(1)}=\alpha_1^{(1)}-\alpha_6^{(1)}$.\\

Using the $O(h^2)$- approximations from  (\ref{op2})-(\ref{op3}) into (\ref{eq3D26})
\begin{equation}\label{eq3D27}
\begin{split}
T^{(3)}_1=&\beta_1^{(1)}K^6(\delta_x^2+\delta_y^2+\delta_z^2)u-\beta_1^{(1)}K^4(\delta_x^2+\delta_y^2+\delta_z^2)f+\beta_2^{(1)}K^2\delta_x^2\delta_y^2\delta_z^2u\\
+&2\beta_1^{(1)}K^2\sum_{\substack{i,j,k\in\Lambda_0\\i+j+k=4}}{\delta_x^i\delta_y^j\delta_z^k(2K^2u-f)}+\beta_1^{(1)}K^2(\delta_x^2\partial_x^2+\delta_y^2\partial_y^2+\delta_z^2\partial_z^2)f\\
+&\beta_4^{(1)}\sum_{\substack{i,j,k\in\Lambda_3\\i+j+k=8}}{\partial_x^i\partial_y^j\partial_z^ku}+\beta_5^{(1)}\sum_{\substack{i,j,k\in\Lambda_2\\i+j+k=6\\i\neq j}}{\partial_x^i\partial_y^j\partial_z^kf}
+\beta_3^{(1)}\partial_x^2\partial_y^2\partial_z^2f\\
+&\beta_6^{(1)}\nabla^6f+O(h^8).
\end{split}
\end{equation}
Finally, we substitute (\ref{eq3D27}) into (\ref{eq3D17}), to get another sixth order compact scheme for the three dimensional Helmholtz equation (\ref{eq3D1})
\begin{equation}\label{eq3D28}
\begin{split}
&C_1(\delta_x^2+\delta_y^2+\delta_z^2)u_{i,j,k}+C_2(\delta_x^2\delta_y^2+\delta_y^2\delta_z^2+\delta_x^2\delta_z^2)u_{i,j,k}+C_3\delta_x^2\delta_y^2\delta_z^2u_{i,j,k}\\
&+C_4u_{i,j,k}=D_1f_{i,j,k}+D_2(\delta_x^2+\delta_y^2+\delta_z^2)f_{i,j,k}+D_3(\delta_x^2\delta_y^2+\delta_y^2\delta_z^2+\delta_x^2\delta_z^2)f_{i,j,k}\\
&+D_4\nabla^2f_{i,j,k}+D_5(\delta_x^2\partial_x^2+\delta_y^2\partial_y^2+\delta_z^2\partial_z^2)f_{i,j,k}+T^{(3)}_2,
\end{split}
\end{equation}
where $T^{(3)}_2=O(h^6)$ be the sixth order leading truncation error term to new scheme and is given by
\begin{equation}\label{eq3D29}
\begin{split}
T^{(3)}_2=\beta_4^{(1)}\sum_{\substack{i,j,k\in\Lambda_3\\i+j+k=8}}{\partial_x^i\partial_y^j\partial_z^ku}+\beta_5^{(1)}\sum_{\substack{i,j,k\in\Lambda_2\\i+j+k=6\\i\neq j}}{\partial_x^i\partial_y^j\partial_z^kf}
+\beta_3^{(1)}\partial_x^2\partial_y^2\partial_z^2f+\beta_6^{(1)}\nabla^6f,
\end{split}
\end{equation}
where, the coefficients of the new scheme are given by
\begin{equation*}
\begin{split}
C_1=&1+\frac{K^6h^6}{20160},
C_2=\frac{h^2}{6}\left(1+\frac{K^2h^2}{30}+\frac{K^4h^4}{840}\right),
C_3=\frac{h^4}{30}\left(1+\frac{17K^2h^2}{252}\right),\\
C_4=&K^2\left(1-\frac{K^2h^2}{12}+\frac{K^4h^4}{360}\right),
D_1=1-\frac{K^2h^2}{12}+\frac{K^4h^4}{360},
D_2=\frac{K^4h^6}{20160},\\
D_3=&\frac{h^4}{90}\left(1+\frac{K^2h^2}{112}\right),
D_4=\frac{h^2}{12}\left(1-\frac{K^2h^2}{30}\right),
D_5=\frac{h^4}{360}\left(1-\frac{K^2h^2}{56}\right).
\end{split}
\end{equation*}
It is noted that the truncation error $T^{(3)}_2$ does not explicitly depend on the wave number $K$. It depends only on the solution value $u$, the source function $f$.
\section{High order accurate approximation of Neumann boundary}
Dirichlet boundary conditions specify the solution value $u$ at each node of the boundary. Therefore, in case of Dirichlet boundary, the difference schemes can be used for all interior grid points. However, Neumann boundary conditions specify the derivative of the solution value at some part of the boundary. In the later case, the compact schemes can not be used straightforward for all interior grid points since the values at some boundary points are not given.

In this section, a discretization technique is developed for a Neumann boundary condition such that the discretization is consistent with the given sixth-order accurate difference schemes. Here, we consider a sixth-order approximation for $\partial_xu=g(y)$ and $\partial_xu=g(y,z)$ for the two- and three-dimensional Helmholtz equations respectively.
\subsection{Two-dimensional case}
Without loss of generality, we consider a sixth order accurate discretization for a Neumann condition $\partial_xu|_{x=x_0}=g(y)$ in two dimensions, where the function $g$ has required continuous and bounded derivatives. 

Consider the uniform mesh discretization of the domain $\Omega\in\mathbb{R}^2$ as $x_i=x_0+ih, y_j=y_0+jh, i, j=0(1)N$. For the aim, we cosider further discretization of the domain $\Omega$ by adding a row consisting of ghost points $(x_{-1},y_j)=(x_0-h,y_j), j=0(1)N$, outside the domain.

The central difference second order approximation of first derivative gives
\begin{equation*}\label{NBC1}
\partial_xu|_{x=x_0}=\frac{u_{1,j}-u_{-1,j}}{2h}+O(h^2).
\end{equation*}
Using the Taylor series expansion at $(x_0,y_j)$, we get
\begin{equation}\label{NBC2}
\frac{u_{1,j}-u_{-1,j}}{2h}=(\partial_xu)_{0,j}+\frac{h^2}{6}(\partial_x^3u)_{0,j}+\frac{h^4}{120}(\partial_x^5u)_{0,j}+O(h^6).
\end{equation}
Now by successive differentiation of the Helmholtz equation (\ref{eq15}), we get
\begin{equation}\label{NBC3}
\begin{split}
\partial_x^3u=&\partial_xf-K^2\partial_xu-\partial_x\partial_y^2u,\\
\partial_x^5u=&\partial_x^3f-K^2\partial_xf-\partial_x\partial_y^2f+K^4\partial_xu+2K^2\partial_x\partial_y^2u+\partial_x\partial_y^4u.
\end{split}
\end{equation}
Neumann condition $\partial_xu=g(y)$ with its tangential derivatives gives
\begin{equation}\label{NBC4}
\begin{split}
\partial_x^3u=&\partial_xf-K^2g-\partial_y^2g,\\
\partial_x^5u=&\partial_x^3f-K^2\partial_xf-\partial_x\partial_y^2f+
K^4g+2K^2\partial_y^2g+\partial_y^4g.
\end{split}
\end{equation}
Substituting the above derivatives at $(0,j)$ into (\ref{NBC2}), we subsequently get a sixth-order approximation for the Neumann boundary $\partial_xu|_{x=x_0}=g(y)$
\begin{equation}\label{NBC5}
\begin{split}
\frac{u_{1,j}-u_{-1,j}}{2h}=&g_j+\frac{h^2}{6}(\partial_xf)_{0,j}+\frac{h^4}{120}(\partial_x^3f-K^2\partial_xf-\partial_x\partial_y^2f)_{0,j}\\
-&\frac{h^2}{6}(K^2g+\partial_y^2g)_j+\frac{h^4}{120}(K^4g+2K^2\partial_y^2g+\partial_y^4g)_j+O(h^6).
\end{split}
\end{equation}
Similar discretizations hold for the Neumann conditions $\partial_xu|_{x=x_N}=g(y)$, $\partial_yu|_{y=y_0}=g(x)$ and $\partial_yu|_{y=y_N}=g(x)$ in other directions. Using the difference schemes at the boundary point $(x_0,y_j), j=0(1)N$, we can eliminate the values $u_{-1,j}$ at the ghost points.
\subsection{Three-dimensional case}
Here, we now consider a Neumann condition $\partial_xu|_{x=x_0}=g(y,z)$ for the three-dimensional Helmholtz equation, assuming the function $g$ has required continuous and bounded derivatives in $y$ and $z$.

As in case of two dimensions, we also introduce a further discretization of the domain $\Omega\in\mathbb{R}^3$ having ghost points $(x_{-1},y_j,z_k), j,k=0(1)N$, outside the boundary of the domain. Then the central difference approximation of derivative gives
\begin{equation*}
\partial_xu|_{x=x_0}=\frac{u_{1,j,k}-u_{-1,j,k}}{2h}+O(h^2).
\end{equation*}
Taylor series expansion at $(x_0,y_j,z_k)$ gives
\begin{equation}\label{NBC6}
\frac{u_{1,j,k}-u_{-1,j,k}}{2h}=(\partial_xu)_{0,j,k}+\frac{h^2}{6}(\partial_x^3u)_{0,j,k}+\frac{h^4}{120}(\partial_x^5u)_{0,j,k}+O(h^6).
\end{equation}
Successive differentiation of the Helmholtz equation (\ref{eq3D1}) gives
\begin{equation}\label{NBC7}
\begin{split}
\partial_x^3u=&\partial_xf-K^2\partial_xu-\partial_x\partial_y^2u-\partial_x\partial_z^2u,\\
\partial_x^5u=&\partial_x^3f-K^2\partial_xf-\partial_x\partial_y^2f-\partial_x\partial_z^2f+
K^4\partial_xu+2K^2(\partial_x\partial_y^2u+\partial_x\partial_z^2u)\\
&+\partial_x\partial_y^4u+\partial_x\partial_z^4u+2\partial_x\partial_y^2\partial_z^2u.
\end{split}
\end{equation}
Using the Neumann boundary condition $\partial_xu=g(y,z)$ and its tangential derivatives into (\ref{NBC7}), we have
\begin{equation}\label{NBC8}
\begin{split}
\partial_x^3u=&\partial_xf-K^2g-\partial_y^2g-\partial_z^2g,\\
\partial_x^5u=&\partial_x^3f-K^2\partial_xf-\partial_x\partial_y^2f-\partial_x\partial_z^2f+
K^4g+2K^2(\partial_y^2g+\partial_z^2g)+\partial_y^4g+\partial_z^4g\\
&+2\partial_y^2\partial_z^2g.
\end{split}
\end{equation}
Substituting the Neumann boundary $\partial_xu|_{x=x_0}=g(y,z)$ and derivatives $\partial_x^3u$ and $\partial_x^5u$ at $(0,j,k)$ into (\ref{NBC6}), we get
\begin{equation}\label{NBC9}
\begin{split}
\frac{u_{1,j,k}-u_{-1,j,k}}{2h}=&g_{j,k}+\frac{h^2}{6}(\partial_xf)_{0,j,k}\\
+&\frac{h^4}{120}(\partial_x^3f-K^2\partial_xf-\partial_x\partial_y^2f-\partial_x\partial_z^2f)_{0,j,k}\\
-&\frac{h^2}{6}(K^2g+\partial_y^2g+\partial_z^2g)_{j,k}+\frac{h^4}{120}\left(K^4g+2K^2(\partial_y^2g+\partial_z^2g)\right.\\
+&\left.\partial_y^4g+\partial_z^4g+2\partial_y^2\partial_z^2g\right)_{j,k}+O(h^6).
\end{split}
\end{equation}
Similar discretizations hold for the Neumann conditions in other directions. Using the difference schemes at the boundary point $(x_0,y_j,z_k), j,k=0(1)N$, we can eliminate the values $u_{-1,j,k}$ at the ghost points.
\section{Convergence analysis}\label{sec:3}
In this section, the theoretical analysis for the proposed scheme is presented. It is shown that the proposed scheme is uniquely solvable for sufficiently small $Kh$, where $K$ is the wave number and $h$ is the mesh step size. The bound on the error norms to the proposed scheme (\ref{eq32}) is also established. It is proved that as the mesh step size approaches to zero such that $Kh$ is sufficiently small, then the solution of the proposed difference scheme converges to the solution of the corresponding boundary value problem. We consider here a convergence analysis for the proposed scheme (\ref{eq32}) to the Helmholtz equation (\ref{eq15})-(\ref{bc1}) in two dimensions. For the three dimensional case, the treatment is similar to the former.

Let $U_{i,j}$ be the approximate solution to $u_{i,j}$ so that $u_{i,j}=U_{i,j}+O(h^6)$. Consider the difference scheme (\ref{eq32}) in operator form as
\begin{equation}\label{a1}
\mathcal{L}_{i,j}U_{i,j}=F_{i,j},
\end{equation}
where
$\mathcal{L}=-\alpha_1\left(\delta_x^2+\delta_y^2\right)-\alpha_2\delta_x^2\delta_y^2
-K^2\alpha_3,$
\begin{equation*}
F_{i,j}=-\beta_1f_{i,j}-\beta_2\left(\delta_x^2+\delta_y^2\right)f_{i,j}-\beta_3\delta_x^2\delta_y^2f_{i,j}-\beta_4\nabla^2f_{i,j}-\beta_5(\delta_x^2\partial_x^2+\delta_y^2\partial_y^2)f_{i,j},
\end{equation*}
with the corresponding truncation error $T_2^{(2)}$ given by (\ref{eq34}). The coefficients $\alpha_i, \beta_i$'s are given in (\ref{eq33}). Now we will show that the proposed difference scheme is uniquely solvable as $Kh\rightarrow 0$.
\subsection{Solvability of the Difference Scheme}
Rewriting the equation (\ref{a1}) in terms of discrete stencil points
\begin{equation}\label{a2}
\begin{split}
a00U_{i,j}+a10\mathcal{F}U_{i,j}+a20\mathcal{D}U_{i,j}=h^2F_{i,j},
\end{split}
\end{equation}
where the operators $\mathcal{F}$ and $\mathcal{D}$ are same as in (\ref{eq17}) and
\begin{equation}\label{a3}
\begin{split}
a00=&(16800-5152K^2h^2+416K^4h^4-13K^6h^6)/5040,\\
a10=&-(13440-224K^2h^2-8K^4h^4+K^6h^6)/20160,\\
a20=&-(840+28K^2h^2+K^4h^4)/5040.
\end{split}
\end{equation}
The difference equation (\ref{a2}), in the simple matrix form, can be written as
\begin{equation*}
A\textbf{U}=\textbf{r},
\end{equation*}
where $A=[a_{i,j}], ~i,j=1(1)n, n=(N-1)^2$ is a tri-block-diagonal matrix with the following elements\\
\begin{equation}\label{a4}
\begin{split}
\text{For}~~ j=&2(1)N-1:\\
&a_{(j-1)(N-1)+i,(j-2)(N-1)+i-1}=a20,~~ i=2(1)N-1,\\
&a_{(j-1)(N-1)+i,(j-2)(N-1)+i}=a10,~~  i=1(1)N-1,\\
&a_{(j-1)(N-1)+i,(j-2)(N-1)+i+1}=a20,~~  i=1(1)N-2.
\end{split}
\end{equation}
\begin{equation}\label{a5}
\begin{split}
\text{For}~~ j=&1(1)N-1:\\
&a_{(j-1)(N-1)+i,(j-1)(N-1)+i-1}=a10,~~  i=2(1)N-1,\\ 
&a_{(j-1)(N-1)+i,(j-1)(N-1)+i}=a00,~~  i=1(1)N-1,\\ 
&a_{(j-1)(N-1)+i,(j-1)(N-1)+i+1}=a10,~~  i=1(1)N-2. 
\end{split}
\end{equation}
\begin{equation}\label{a6}
\begin{split}
\text{For}~~ j=&1(1)N-2:\\
&a_{(j-1)(N-1)+i,j(N-1)+i-1}=a20,~~ i=2(1)N-1,\\
&a_{(j-1)(N-1)+i,j(N-1)+i}=a10,~~  i=1(1)N-1,\\
&a_{(j-1)(N-1)+i,j(N-1)+i+1}=a20,~~  i=1(1)N-2.
\end{split}
\end{equation}

\begin{figure}
	\centering
	\includegraphics[scale=0.4]{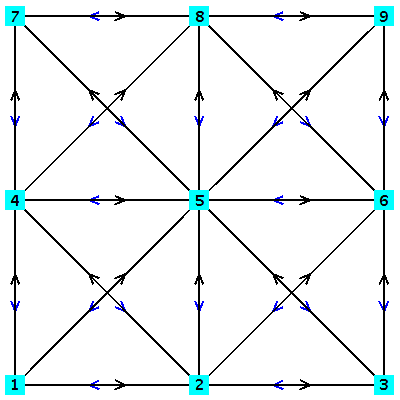}
	\caption{Directed graph for its adjacency matrix $A$ for $N=4$.}
	\label{fig:c1}	
\end{figure}	

Now we claim that the graph of the matrix $A$ is strongly connected and hence we have the following result.
\begin{lemma}\label{lm:1}
	Let $A=[a_{i,j}], i,j=1(1)n, n=(N-1)^2$ be a matrix with its elements $a_{i,j}$ given by (\ref{a4})-(\ref{a6}). When $Kh$ is sufficiently small, where $K$ is the wave number and $h$ is the grid length, then the directed graph $\mathcal{G}(A)$ of the matrix $A$ is strongly connected.
\end{lemma}
\begin{proof}
	Consider the given matrix $A=[a_{i,j}], i,j=1(1)n$ with its elements given by (\ref{a4})-(\ref{a6}).
	For $Kh\rightarrow 0$, the distinct nonzero values of the elements of the matrix are given by
	a00={10}/{3},~a10=-{2}/{3},~a20=-{1}/{6}.
	Therefore, for sufficiently small $Kh$, all the blocks of the matrix contain nonzero elements in its lower, upper and main diagonal. For constructing the directed graph of the matrix $A=[a_{i,j}], i,j=1(1)n$, consider $n$ distinct points in a plane denoted by $P_1, P_2,..,P_n$. Then corresponding to each nonzero element $a_{i,j}$ of the matrix $A$, connect the point $P_i$ to $P_j$ by a directed arc, directed from $P_i$ to $P_j$. The graph thus obtained is called a directed graph of the matrix $A$ \cite{varga2009matrix,saad2003iterative}. Such a directed graph of a matrix is depicted for $N=4$ in Figure \ref{fig:c1} and its adjacency matrix is given as
	\begin{equation*}\label{adj}
	A=
	\begin{bmatrix}
	d1 & d2 & 0 & d2 & d3 & 0 & 0 & 0 & 0\\
	d2 & d1 & d2 & d3 & d2 & d3 & 0 & 0 & 0\\
	0 & d2 & d1 & 0 & d3 & d2 & 0 & 0 & 0\\
	d2 & d3 & 0 & d1 & d2 & 0 & d2 & d3 & 0\\
	d3 & d2 & d3 & d2 & d1 & d2 & d3 & d2 & d3\\
	0 & d3 & d2 & 0 & d2 & d1 & 0 & d3 & d2\\
	0 & 0 & 0 & d2 & d3 & 0 & d1 & d2 & 0\\
	0 & 0 & 0 & d3 & d2 & d3 & d2 & d1 & d2\\
	0 & 0 & 0 & 0 & d3 & d2 & 0 & d2 & d1
	\end{bmatrix},	
	\end{equation*}
	where $d1=a00, d2=a10, d3=a20$. Hence from the Figure, it is clear that for any ordered pair of points $(P_i,P_j), i,j=1(1)n$, there exist a directed path from $P_i$ to $P_j$ and hence the directed graph of the matrix $A$ is strongly connected \cite{varga2009matrix}, which proves the Lemma.
\end{proof}
We next establish the conditions on row sum of the matrix in the following result.
\begin{lemma}\label{lm:2}
	Let $A=[a_{i,j}], i,j=1(1)n$, where $n=(N-1)^2$, be a matrix with its elements $a_{i,j}$ given by (\ref{a4})-(\ref{a6}). When $Kh$ is sufficiently small, where $K$ is the wave number and $h$ is the grid length, then
	\begin{enumerate}[label=(\alph*)]
		\item $a_{i,j}\leq 0, ~~~i\neq j,~~ i,j=1(1)n$,
		\item $\sum_{j=1}^{n}{a_{i,j}}\geq 0$,~$i=1(1)n$ and
		$\sum_{j=1}^{n}{a_{i,j}}> 0$,~for at least one $i$.
	\end{enumerate}
\end{lemma}
\begin{proof}
	Consider the given matrix $A=[a_{i,j}], i,j=1(1)n$ with its elements given by (\ref{a4})-(\ref{a6}),
	Clearly when $Kh\rightarrow 0$, then
	a00={10}/{3}, a20=-{1}/{6}, a10=-{2}/{3}.
	it is also noted that $a00$ gives the values of all diagonal element of the matrix $A$ and $a10, a20$ give the values of the off-diagonal elements of the matrix $A$ which are clearly non-positive. Thus $a_{i,j}\leq 0, ~~~i\neq j,~~ i,j=1(1)n$.
	
	Now let $S_q$ denotes the $q^{th}$ row sum of the matrix $A$. Then, for $Kh\rightarrow 0$,
	\begin{equation}\label{a10}
	\begin{split}
	&S_1=a00+2a10+a20={11}/{6}>0,\\
	&S_q=a00+3a10+2a20=1>0,~ q=2(1)N-2,\\
	&S_{N-1}=a00+2a10+a20={11}/{6}>0,
	\end{split}
	\end{equation}
	\begin{equation}\label{a11}
	\begin{split}
	S_{(r-1)(N-1)+1}=&a00+3a10+2a20=1>0, s=2(1)N-2,\\
	S_{(r-1)(N-1)+q}=&a00+4a10+4a20=0\geq 0,~q=2(1)N-2, s=2(1)N-2,\\
	S_{(r-1)(N-1)+N-1}=&a00+3a10+2a20=1>0,~ s=2(1)N-2,
	\end{split}
	\end{equation}
	\begin{equation}\label{a12}
	\begin{split}
	&S_{(N-2)(N-1)+1}=a00+2a10+a20={11}/{6}>0,\\
	&S_{(N-2)(N-1)+q}=a00+3a10+2a20=1>0, q=2(1)N-2,\\
	&S_{(N-2)(N-1)+N-1}=a00+2a10+a20={11}/{6}>0.
	\end{split}
	\end{equation}
	From the row sum of the matrix given by (\ref{a10})-(\ref{a12}), it is clear that all the row sum are nonnegative and at least one is positive, which completes the proof.
\end{proof}
Next we will show that the difference scheme has unique solution.
\begin{theorem}\label{th:3}
	Consider the difference scheme (\ref{a1}) as
	\begin{equation}\label{a13}
	\mathcal{L}_{i,j}U_{i,j}=F_{i,j},
	\end{equation}
	where the operator $\mathcal{L}$ and $F_{i,j}$ are given in (\ref{a1}). Let $Kh$ is sufficiently small, where $K$ is the wave number, $h$ is the mesh size. Then, there exist a unique solution of the difference scheme (\ref{a13}).
\end{theorem}
\begin{proof}
	Since from the Lemma \ref{lm:1}, the directed graph $\mathcal{G}(A)$ of the matrix $A$ is strongly connected and hence the matrix $A$ is irreducible \cite{varga2009matrix}. Further, from Lemma \ref{lm:2}, all off-diagonal elements of the matrix $A$ are non-positive and also all row sum of the irreducible matrix $A$ are nonnegative with at least one positive row sum which implies that the matrix $A$ is monotone \cite{henrici1962discrete}. Thus the matrix $A$ is irreducible and monotone and hence inverse of the matrix $A$ exist and $A^{-1}\geq 0$ \cite{meyer2000matrix}. Consequently, the difference scheme (\ref{a13}) has unique solution, which proves the Theorem.
\end{proof}
\subsection{Error estimate of the Scheme}
In this section, the error bound for the proposed difference scheme (\ref{eq32}) is established. For this first we will show that the proposed scheme satisfies the discrete maximum principle \cite{Thomas2013}, then we will establish the discrete regularity result \cite{Thomas2013,strikwerda2004finite}. Both the results will be used to prove our main Theorem for getting error bound. Before we present some results, we need to mention some notations for norms which will be used in the later contexts.

Let $u$ be the exact solution to the problem (\ref{eq15})-(\ref{bc1}) and $U_{i,j}, (i,j)\in\Omega_h$ be the solution to the difference scheme (\ref{eq32}), where $\Omega_h$ is the discrete grid on rectangle domain $\Omega$ with its boundary $\partial\Omega_h$. We let define the sup-norms of the solution values
\begin{equation*}\label{c0}
\begin{split}
||U||_{\infty}=\max_{0\leq i,j\leq N}{|U_{i,j}|},~~
||U||_{\infty,\Omega_h^0}=\max_{1\leq i,j\leq N-1}{|U_{i,j}|},~~
||U||_{\infty,\partial\Omega_h}=\max_{\partial\Omega_h}{|U_{i,j}|}.
\end{split}
\end{equation*}
Now it is claimed that the discrete maximum principle is satisfied by the present difference scheme (\ref{eq32}) in the following result.
\begin{lemma}\label{lm:4}
	Let $ U_{i,j}, (i,j)\in\Omega_h$, be a function defined on the grid $\Omega_h$.
	If
	\begin{equation}\label{c1}
	\mathcal{L}_{i,j} U_{i,j}\leq 0, ~~(x_i,y_j)\in\Omega_h^0
	\end{equation}
	$(\mathcal{L}_{i,j} U_{i,j}\geq 0)$ on $\Omega_h^0$, where $\mathcal{L}$ is given in (\ref{a1}) and $\Omega_h^0$ is an interior of the discrete grid $\Omega_h$. Then, for sufficiently small $Kh$, the maximum (minimum) of the solution value $ U_{i,j}$ on $\Omega_h$ lies on the boundary $\partial\Omega_h$ of the grid $\Omega_h$. Where, $K$ is the wave number and $h$ is the grid size.
\end{lemma}
\begin{proof}
	For the sake of convenience, we will show that the solution value $ U_{i,j}$ cannot have a local maximum in the interior of the domain.
	
	Let $G=\{i-1,i,i+1\}\times\{j-1,j,j+1\}$, $\mathcal{L}_{i,j} U_{i,j}\leq 0$, from (\ref{a2}) we have
	\begin{equation}\label{c3}
	\begin{split}
	a_0^{(1)} U_{i,j}\leq& a_1^{(1)}\mathcal{F}U_{i,j}+a_2^{(1)}\mathcal{D}U_{i,j},
	\end{split}
	\end{equation}
	where, for $Kh\rightarrow 0$, $a_0^{(1)}=a_{00}={10}/{3},~a_1^{(1)}=-a_{10}={2}/{3},~a_2^{(1)}=-a_{20}={1}/{6}$. The operator $\mathcal{F}$ and $\mathcal{D}$ are same as in (\ref{eq17}).
	Now let us assume that the solution value $ U_{i,j}$ is a local maximum, which implies that
	\begin{equation}\label{c4}
	U_{i,j}\geq  U_{l,m}, \forall (l,m)\in G\setminus\{(i,j)\}.
	\end{equation}
	Keeping $ U_{i+1,j}$ fixed, using (\ref{c4}) into (\ref{c3}) we have
	\begin{equation}\label{c5}
	a_0^{(1)} U_{i,j}\leq a_1^{(1)}( U_{i+1,j}+3 U_{i,j})+4a_2^{(1)} U_{i,j}.
	\end{equation}
	Now using $ U_{i,j}\geq  U_{i+1,j}$ into (\ref{c5}), equation (\ref{c5}) becomes
	\begin{equation}\label{c6}
	a_0^{(1)} U_{i,j}\leq a_1^{(1)}( U_{i+1,j}+3 U_{i,j})+4a_2^{(1)} U_{i,j}\leq (4a_1^{(1)}+4a_2^{(1)}) U_{i,j}.
	\end{equation}
	It is also clear that $4a_1^{(1)}+4a_2^{(1)}=a_0^{(1)}$, hence inequality (\ref{c6}) will become
	\begin{equation}\label{c7}
	a_0^{(1)} U_{i,j}= a_1^{(1)} U_{i+1,j}+(3a_1^{(1)}+4a_2^{(1)}) U_{i,j},
	\end{equation}
	which further implies that $ U_{i,j}= U_{i+1,j}$. In this same manner, one can also prove that
	\begin{equation}\label{c8}
	U_{i,j}= U_{i-1,j}= U_{i,j+1}= U_{i,j-1}.
	\end{equation}
	Now keeping $ U_{i+1,j+1}$ fixed, again using assumptions (\ref{c4}), into (\ref{c3}), and proceeding as above, we have
	\begin{equation}\label{c11}
	U_{i,j}= U_{i-1,j}= U_{i,j+1}= U_{i,j-1}= U_{i+1,j-1}= U_{i-1,j+1}= U_{i-1,j-1}.
	\end{equation}
	Hence, $ U_{i,j}$ cannot have a local maximum in the interior of the domain $\Omega$. Therefore, the maximum of the solution value on $\Omega_h$ must lies on the boundary $\partial\Omega$ of the domain. Following the above procedure, One can also prove that the minimum of the solution value on $\Omega_h$ must lies on the boundary $\partial\Omega$ of the domain in the case of $\mathcal{L}_{i,j} U_{i,j}\geq 0$. The proof is complete.
\end{proof}
We next have the following result for the present scheme.
\begin{lemma}\label{lm:5}
	Let $\Omega_h$ be a rectangular grid with its interior $\Omega_h^0$ and boundary $\partial\Omega_h$. $U_{i,j}, (i,j)\in\Omega_h$, be a grid function defined on the grid $\Omega_h$, with $ U_{i,j}=0, (i,j)\in\partial\Omega_h$. Then for sufficiently small $Kh$,
	\begin{equation}\label{c12}
	||\textbf{U}||_{\infty}\leq\frac{1}{8}||\mathcal{L}_{i,j} U_{i,j}||_{\infty,\Omega_h^0},
	\end{equation}
	where $\mathcal{L}$ is given in (\ref{a1}) and $||\cdot||_{\infty,\Omega_h^0}$ be the maximum norm defined on the interior of the grid $\Omega_h$.
\end{lemma}
\begin{proof}
	Consider a grid function $F_{i,j}$ defined on the interior of the discrete grid $\Omega_h$ by
	\begin{equation}\label{c13}
	F_{i,j}=\mathcal{L}_{i,j}U_{i,j}, (i,j)\in\Omega_h^0.
	\end{equation}
	By the definition of maximum norm, it is obvious that
	\begin{equation}\label{c14}
	|F_{i,j}|\leq ||\textbf{F}||_{\infty,\Omega_h^0}, \forall (i,j)\in \Omega_h^0.
	\end{equation}
	Then from (\ref{c13})-(\ref{c14}), we get the inequalities
	\begin{equation}\label{c15}
	\mathcal{L}_{i,j}U_{i,j}+||\textbf{F}||_{\infty,\Omega_h^0} \geq 0, ~~~~ \mathcal{L}_{i,j}U_{i,j}-||\textbf{F}||_{\infty,\Omega_h^0} \leq 0.
	\end{equation}
	We next define
	$v_{i,j}=[\left(x_i-{1}/{2}\right)^2+\left(y_j-{1}/{2}\right)^2]/4$,
	then for $Kh\rightarrow 0$, we have
	\begin{equation}\label{c17}
	\begin{split}
	\mathcal{L}_{i,j}v_{i,j}&=-[1+K^2\{\left(x_i-{1}/{2}\right)^2+\left(y_j-{1}/{2}\right)^2\}/4]=-C,
	\end{split}
	\end{equation}
	where $C=[1+K^2\{\left(x_i-{1}/{2}\right)^2+\left(y_j-{1}/{2}\right)^2\}/4]>1$. Then, we have
	\begin{equation*}\label{c18}
	\mathcal{L}_{i,j}( U_{i,j}-||\textbf{F}||_{\infty,\Omega_h^0}v_{i,j})=\mathcal{L}_{i,j} U_{i,j}-||\textbf{F}||_{\infty,\Omega_h^0}\mathcal{L}_{i,j}v_{i,j}=\mathcal{L}_{i,j} U_{i,j}+C||\textbf{F}||_{\infty,\Omega_h^0}.
	\end{equation*}
	Now since $C>1$, Thus from first inequalities in (\ref{c15}), we have
	\begin{equation}\label{c19}
	\mathcal{L}_{i,j}( U_{i,j}-||\textbf{F}||_{\infty,\Omega_h^0}v_{i,j})\geq\mathcal{L}_{i,j} U_{i,j}+||\textbf{F}||_{\infty,\Omega_h^0}\geq 0.
	\end{equation}
	In the same manner, using second inequalities in (\ref{c15}), we can have
	\begin{equation}\label{c20}
	\mathcal{L}_{i,j}( U_{i,j}+||\textbf{F}||_{\infty,\Omega_h^0}v_{i,j})
	\leq 0.
	\end{equation}
	Thus from Lemma \ref{lm:4}, the maximum of $\textbf{U}+||\textbf{F}||_{\infty,\Omega_h^0}\textbf{v}$ and the minimum of $\textbf{U}-||\textbf{F}||_{\infty,\Omega_h^0}\textbf{v}$ must occur on the boundary $\partial\Omega_h$. Then, using $U_{i,j}=0$ on $\partial\Omega_h$, we get
	\begin{equation}\label{c21}
	\begin{split}
	\max_{\partial\Omega_h}[\textbf{U}+||\textbf{F}||_{\infty,\Omega_h^0}\textbf{v}]=\max_{\partial\Omega_h}\textbf{U}+||\textbf{F}||_{\infty,\Omega_h^0}||\textbf{v}||_{\infty,\partial\Omega_h}=||\textbf{F}||_{\infty,\Omega_h^0}||\textbf{v}||_{\infty,\partial\Omega_h},
	\end{split}
	\end{equation}
	\begin{equation}\label{c22}
	\begin{split}
	\max_{\partial\Omega_h}[\textbf{U}+||\textbf{F}||_{\infty,\Omega_h^0}\textbf{v}]\geq&  U_{i,j}+||\textbf{F}||_{\infty,\Omega_h^0}v_{i,j}
	\geq  U_{i,j}. ~~~~(\because ||\textbf{F}||_{\infty,\Omega_h^0}v_{i,j}\geq 0)
	\end{split}
	\end{equation}
	Hence, from (\ref{c21})-(\ref{c22}),
	\begin{equation}\label{c23}
	||\textbf{F}||_{\infty,\Omega_h^0}||\textbf{v}||_{\infty,\partial\Omega_h}\geq  U_{i,j}, \forall (i,j)\in\Omega_h.
	\end{equation}
	Proceeding as above for the minimum attained on the boundary, one can prove
	\begin{equation}\label{c24}
	-||\textbf{F}||_{\infty,\Omega_h^0}||\textbf{v}||_{\infty,\partial\Omega_h}\leq U_{i,j}, \forall (i,j)\in\Omega_h.
	\end{equation}
	From (\ref{c23})-(\ref{c24}), we have
	\begin{equation}\label{c25}
	-||\textbf{F}||_{\infty,\Omega_h^0}||\textbf{v}||_{\infty,\partial\Omega_h}\leq U_{i,j}\leq ||\textbf{F}||_{\infty,\Omega_h^0}||\textbf{v}||_{\infty,\partial\Omega_h}, \forall (i,j)\in\Omega_h.
	\end{equation}
	Since $||\textbf{v}||_{\infty,\partial\Omega_h}={1}/{8}$, thus from (\ref{c25}), we get the required result
	\begin{equation}\label{c26}
	||\textbf{U}||_{\infty}\leq\frac{1}{8}||\mathcal{L}_{i,j}U_{i,j}||_{\infty,\Omega_h^0},
	\end{equation}
	which proves the Lemma.
\end{proof}
Now we consider error estimate for scheme (\ref{eq32}) in the following theorem.\\
For this, let we denote
\begin{equation*}
||\partial^pu||_{\infty,\Omega_h^0}=\sup\left\{\left\rvert\frac{\partial^pu}{\partial x^{r_1}\partial y^{r_2}}(x,y)\right\rvert: (x,y)\in\Omega_h^0, r_1+r_2=p, r_1,r_2=0(1)p\right\},
\end{equation*}
which will be used in the later context.
\begin{theorem}\label{th:6}
	Let $\overline{\Omega}=\Omega\cup\partial\Omega$, $f\in C^{6}(\overline{\Omega})$ and $u\in C^{8}(\overline{\Omega})$ be a solution to the problem (\ref{eq15})-(\ref{bc1}) and $U_{i,j}$ be the solution  to the present difference scheme (\ref{eq32}). Then for sufficiently small $Kh$, there exist the following error estimate for the present scheme
	\begin{equation}\label{c27}
	||\textbf{u}-\textbf{U}||_{\infty}\leq C_1^{(1)}h^6+C_2^{(1)}h^6||\partial^8u||_{\infty,\Omega_h^0},
	\end{equation}
	for some constants $C_1^{(1)}, C_2^{(1)}$, where $K$ is the wave number, $h$ is the grid size.
\end{theorem}
\begin{proof}
	Consider the difference scheme (\ref{eq32}) as
	\begin{equation}\label{c28}
	\mathcal{L}_{i,j}U_{i,j}=F_{i,j},
	\end{equation}
	with the corresponding local truncation error $T_{i,j}$ given by (\ref{eq34}).
	
	Next, the local truncation error is the amount by which the solution value $u$ to the problem (\ref{eq15})-(\ref{bc1}) does not satisfy the difference scheme (\ref{c28}). Hence substituting the solution value $u$ in the difference scheme (\ref{c28}), we get
	\begin{equation}\label{c30}
	\mathcal{L}_{i,j}u_{i,j}=F_{i,j}+T_{i,j}.
	\end{equation}
	From (\ref{c28}) and (\ref{c30}), we have
	\begin{equation}\label{c31}
	\mathcal{L}_{i,j}(u_{i,j}-U_{i,j})=T_{i,j},
	\end{equation}
	or
	\begin{equation}\label{c32}
	||\mathcal{L}_{i,j}(u_{i,j}-U_{i,j})||_{\infty,\Omega_h^0}\leq Ch^6+Dh^6||\partial^8u||_{\infty,\Omega_h^0}.
	\end{equation}
	Both the constants $C, D$ are explicitly independent of the wave number $K$ and the solution value $u$ and given by
	$C=||\partial^6f||_{\infty,\Omega_h^0}/720, ~D={1}/{3024}$. 
	It is also noted that $u_{i,j}-U_{i,j}=0, (i,j)\in\partial\Omega_h$. Therefore from Lemma \ref{lm:5}, we have that
	\begin{equation}\label{c34}
	||\textbf{u}-\textbf{U}||_{\infty}\leq\frac18||\mathcal{L}_{i,j}(u_{i,j}-U_{i,j})||_{\infty,\Omega_h^0}.
	\end{equation}
	Thus from (\ref{c32}) and (\ref{c34}), we get the required error estimate
	\begin{equation}\label{c35}
	||\textbf{u}-\textbf{U}||_{\infty}\leq C_1^{(1)}h^6+C_2^{(1)}h^6||\partial^8u||_{\infty,\Omega_h^0},
	\end{equation}
	for some constants $C_1^{(1)}=C/8, ~~C_2^{(1)}=D/8$, which completes the proof.
\end{proof}
Next we consider error estimate for scheme (\ref{eq26}) in the following Theorem.
\begin{theorem}\label{th:7}
	Let $\overline{\Omega}=\Omega\cup\partial\Omega$, $f\in C^{6}(\overline{\Omega})$ and $u\in C^{8}(\overline{\Omega})$ be a solution to the problem (\ref{eq15})-(\ref{bc1}) and $ U_{i,j}$ be the solution value to the standard sixth order scheme (\ref{eq26}). Then for sufficiently small $Kh$,
	\begin{equation}\label{c36}
	||\textbf{u}-\textbf{U}||_{\infty}\leq C_3^{(1)}h^6+C_4^{(1)}\left(K^2||\partial^6u||_{\infty,\Omega_h^0}+||\partial^{8}u||_{\infty,\Omega_h^0}\right)h^6,
	\end{equation}
	where $K$ is the wave number and $h$ is the grid size.
\end{theorem}
\begin{proof}
	Consider the scheme (\ref{eq26}) with its corresponding local truncation error (\ref{eq27})
	and then proceeding in a similar way as for Theorem \ref{th:6}, it follows that
	\begin{equation}\label{c38}
	\begin{split}
	&||\mathcal{L}_{i,j}(u_{i,j}-U_{i,j})||_{\infty,\Omega_h^0}\\
	&\leq\left(\beta_1^{(2)}||\partial^6f||_{\infty,\Omega_h^0}+\beta_2^{(2)}K^2||\partial^6u||_{\infty,\Omega_h^0}+\beta_3^{(2)}||\partial^{8}u||_{\infty,\Omega_h^0}\right)h^6,
	\end{split}
	\end{equation}
	where $\beta_1^{(2)}=1/432, \beta_2^{(2)}=1/1080, \beta_3^{(2)}=11/5040$.
	Then, from Lemma \ref{lm:5},
	\begin{equation}\label{c39}
	||\textbf{u}-\textbf{U}||_{\infty}\leq C_3^{(1)}h^6+C_4^{(1)}\left(K^2||\partial^6u||_{\infty,\Omega_h^0}+||\partial^{8}u||_{\infty,\Omega_h^0}\right)h^6,
	\end{equation}
	where $C_3^{(1)}=||\partial^6f||_{\infty,\Omega_h^0}/3456,~ C_4^{(1)}={11}/{40320}$, which completes the proof.
\end{proof} 
Thus, Theorems \ref{th:6}-\ref{th:7} show that the bound of the error norm is explicitly independent of the wave number $K$ for our present scheme (\ref{eq32}). Hence, we expect that the present scheme will provide better results for very large wave number $K$. 

Now we conclude the convergence analysis of the present difference scheme (\ref{eq32}) in the following Theorem.
\begin{theorem}\label{th:8}
	Let $\overline{\Omega}=\Omega\cup\partial\Omega$ be a rectangular domain, with its boundary $\partial\Omega$. Let $f\in C^{6}(\overline{\Omega})$ and $u\in C^{8}(\overline{\Omega})$ be a solution to the problem (\ref{eq15})-(\ref{bc1}) and $ U_{i,j}$ be the solution value to the present scheme (\ref{eq32}). Then for sufficiently small $Kh$, the present scheme provides sixth order accuracy. Further, as mesh size $h\rightarrow 0$ such that $Kh$ is sufficiently small, the present scheme converges to the solution of the corresponding boundary value problem.
\end{theorem}
\section{Results and Discussion}
\label{sec:4}
In this section, we test the accuracy of the present difference schemes (\ref{eq32}) and (\ref{eq3D28}) over the standard sixth-order schemes \cite{nabavi2007new,sutmann2007compact} in two and three dimensions respectively. For the purpose, several model problems governed by the Helmholtz equations in two and three dimensions are presented. All computations were performed with a 3.7 GHz Intel Xeon processor with 64 GB RAM. The program is implemented in MATLAB 2019 using BiCGstab(2) iterative algorithm with some given tolerance. Norm of the residual error is used as the stopping criteria for the iteration of the algorithm.
\subsection{Two-dimensional test cases}
We consider here four test case studies for the two dimensional Helmholtz equation. In order to compare the performance of the scheme (\ref{eq32}) over the scheme \cite{nabavi2007new} given by (\ref{eq26}), we computed the $l_2$- and $l_{\infty}$-norms of the error and the order of convergence $\Theta_{\infty}$. We also plotted numerical errors versus different wave numbers $K$ for smaller and larger grids. If the numerical solution and exact solution values to (\ref{eq15})-(\ref{bc1}) are $U$ and $u$ respectively, $N$=grid points in $x$- and $y$- coordinate directions. Then the error-norms 
are given by
\begin{equation}\label{MAE2D}
\begin{split}
|| \boldsymbol{\epsilon} ||_\infty=\max_{i,j}|u_{i,j}-U_{i,j}|
,~|| \boldsymbol{\epsilon} ||_2=\frac{1}{N}\sqrt{\sum\limits_{i,j=1}^{N}\mid u_{i,j}-U_{i,j} \mid^2},
\end{split}
\end{equation}
and the computational order of convergence
\begin{equation}\label{Order2D}
\Theta_\infty=log_2\frac{|| \boldsymbol{\epsilon} ||_\infty(N)}{|| \boldsymbol{\epsilon} ||_\infty(2N)}.
\end{equation}
We compared the results obtained by the new scheme (\ref{eq32}) to the 6th order scheme \cite{nabavi2007new} for large wave numbers $K$.
\begin{problem}\label{p4}
	Consider the two-dimensional Helmholtz equation \cite{fu2008compact}
	\begin{equation*}\label{eq36}
	\nabla^2u(x,y)+K^2u(x,y)=\pi^2\sin \pi x\sin \pi y,~~~ (x,y)\in\Omega,
	\end{equation*}
	where $\Omega=[0,1]\times[0,0.5]$, with the Dirichlet boundary
	\begin{equation*}
	\begin{split}
	u(0,y)=u(1,y)=u(x,0)=0, u(x,0.5)=\sin\pi x\sin(l\pi/2)+\frac{\sin\pi x}{l^2-1}.
	\end{split}
	\end{equation*} 
	Where $K^2=\pi^2(1+l^2)$, $l$ is odd number. The analytic solution is given by
	\begin{equation*}
	u(x,y)=\sin\pi x\sin l\pi y+\frac{\sin\pi x\sin\pi y}{l^2-1}.
	\end{equation*}
\end{problem}
\begin{table}[]\caption{Error norms to the solution values of Problem \ref{p4} for $l=7, K\approx 22$.} \label{tab:1}
	\centering
	\vspace{0.1cm}
	\begin{tabular}{{c c c c c c c c c c}}
		\hline
		\multirow{2}{*}{$1/h$} &\multicolumn{3}{c}{new scheme (\ref{eq32})}&&\multicolumn{3}{c}{6th order \cite{nabavi2007new}}\\[0.05cm]
		\hhline{~---~---}& $||\epsilon||_{\infty}$  & $||\epsilon||_2$ & $\Theta_{\infty}$& & $||\epsilon||_{\infty}$ & $||\epsilon||_2$ & $\Theta_{\infty}$ \\[0.05cm]
		\hline
		\vspace{0.1cm}
		16 	& 1.32e-04 & 4.87e-05  & -Inf && 2.00e-03 & 7.42e-04  & -Inf\\[0.05cm]
		32 	& 1.04e-06 & 3.65e-07  & 6.99 && 2.43e-05 & 8.58e-06  & 6.36\\[0.05cm]
		64 	& 2.55e-08 & 8.78e-09  & 5.34 && 3.58e-07 & 1.23e-07  & 6.09\\[0.05cm]
		128 & 4.33e-10 & 1.47e-10  & 5.88 && 5.54e-09 & 1.88e-09  & 6.02\\
		\hline    	  
	\end{tabular}
\end{table}
\begin{table}[]\caption{Error norms to the solution values of Problem \ref{p4} for $l=255, K\approx 800$.} \label{tab:2}
	\centering
	\vspace{0.1cm}
	\begin{tabular}{{c c c c c c c c c c}}
		\hline
		\multirow{2}{*}{$1/h$} &\multicolumn{3}{c}{new scheme (\ref{eq32})}&&\multicolumn{3}{c}{6th order \cite{nabavi2007new}}\\[0.05cm]
		\hhline{~---~---}& $||\epsilon||_{\infty}$ & $||\epsilon||_2$ & $\Theta_{\infty}$ && $||\epsilon||_{\infty}$& $||\epsilon||_2$ & $\Theta_{\infty}$ \\[0.05cm]
		\hline
		512	& 3.64e-02& 1.06e-02  & -Inf && 2.23e-01& 6.53e-02 & -Inf\\[0.05cm]
		1024& 1.06e-04& 3.09e-05  & 8.42 && 2.50e-03& 7.26e-04 & 6.48\\[0.05cm]
		2048& 3.63e-07& 1.05e-07  & 8.20 && 3.63e-05& 1.05e-05 & 6.11\\
		\hline    	  
	\end{tabular}
\end{table}
The error norms for the new scheme (\ref{eq32}) and the sixth-order \cite{nabavi2007new} for $K\approx 22,$ and very large $K\approx 800$ are given in Tables \ref{tab:1}-\ref{tab:2}. For this test case, the tolerance for the iteration stopping criteria is taken as $1.0e-11$. In Table \ref{tab:3}, the error ratio of the new scheme (\ref{eq32}) to the sixth-order scheme \cite{nabavi2007new} is obtained for $l=3,  K\approx 10$ and $l=255, K\approx 800$ . In Table \ref{tab:4}, $l_{\infty}$- error for new scheme (\ref{eq32}) and the sixth order scheme \cite{nabavi2007new} for different grid points and wave numbers is also computed.

It is noted that the source function in Problem \ref{p4} is independent of the wave number $K$. From Table \ref{tab:1}-\ref{tab:2}, it is clear that the result obtained by the new scheme is more accurate than the standard sixth order scheme. Further, Table \ref{tab:2} shows that the new scheme is highly accurate for large enough wave number $K$. Table \ref{tab:3} shows that the given error ratio also decreases for large wave number $K$ which shows that the  new scheme(\ref{eq32}) is efficient than the scheme \cite{nabavi2007new} for large wave number $K$ and small step size $h$ so that $Kh$ is sufficiently small.

In Figure \ref{fig:5}, numerical error and order of accuracy of both the schemes for different wave numbers $K$ is plotted. From Figure \ref{fig:5}(a), It is clear that the numerical errors are growing with $K$ but the numerical errors developed by the new scheme is less than that of the standard scheme. From Figure \ref{fig:5}(b), it is also clear that the slope of each line is six in both cases for small wave number $K$ and in case of large $K$, the slope of the line for the new scheme is more than that of the scheme \cite{nabavi2007new} which implies that the scheme (\ref{eq32}) has more accuracy than that of scheme \cite{nabavi2007new} in case of large enough wave number, which also confirms the accuracy of the new scheme. The exact and numerical solutions are also plotted in Figure \ref{fig:2Dsolex1} for this test case.
\begin{figure}[]
	\begin{minipage}{.5\textwidth}
		\centering
		\includegraphics[width=1.0\linewidth]{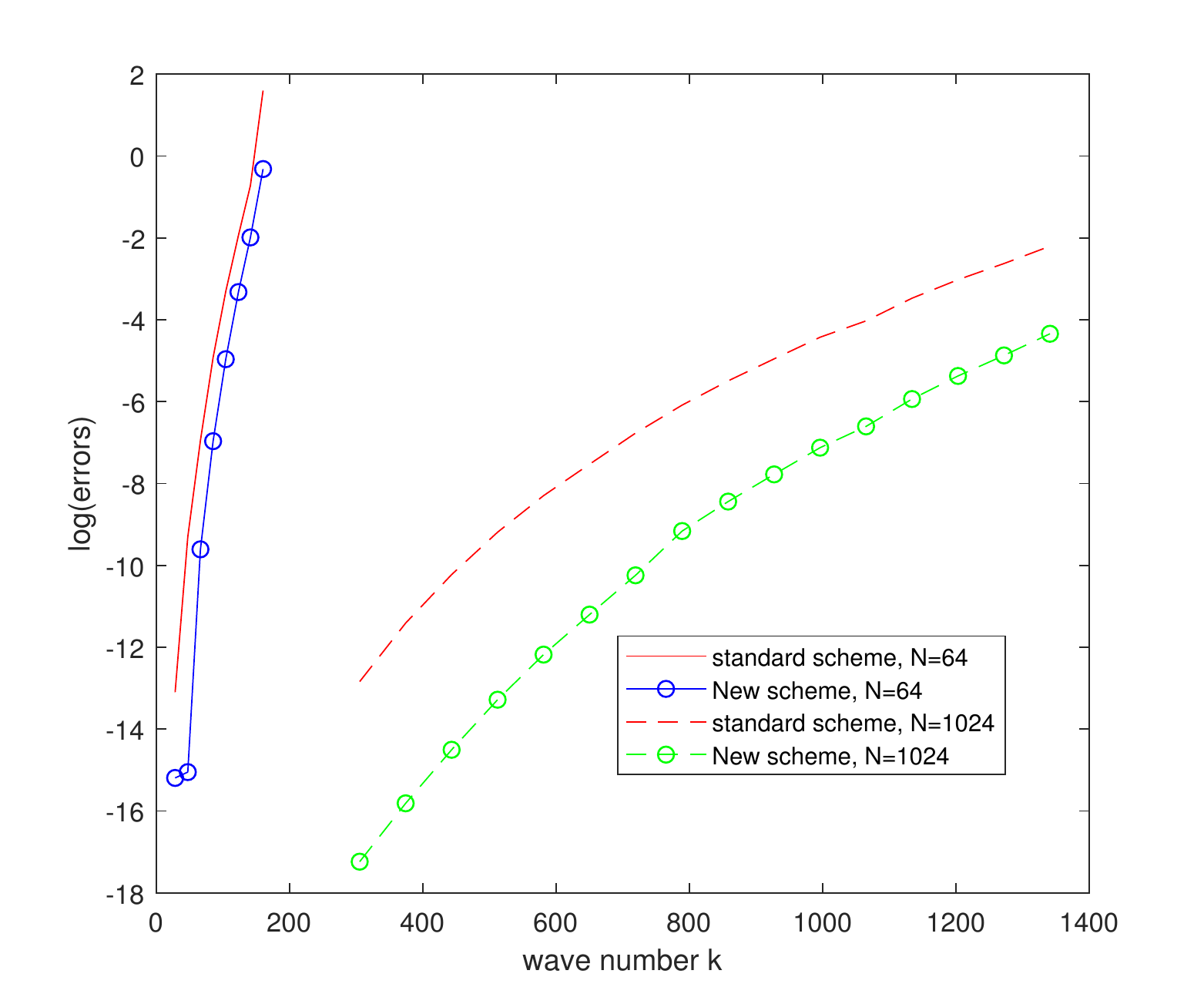}
		\subcaption{}
	\end{minipage}	
	\begin{minipage}{.5\textwidth}
		\centering
		\includegraphics[width=1.0\linewidth]{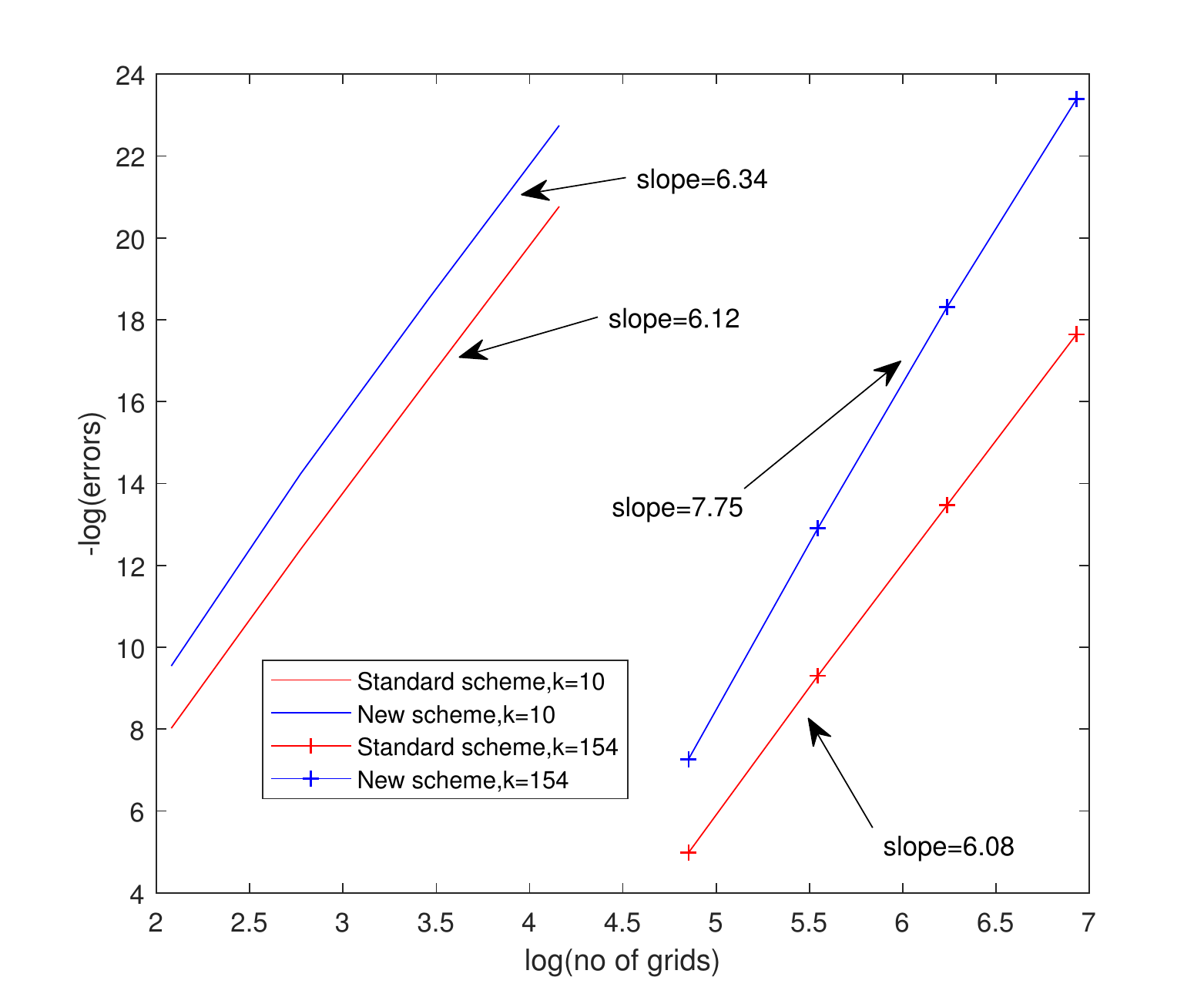}
		\subcaption{}
	\end{minipage}
	\caption{(a) Numerical errors (b) Order of accuracy with different $K$ for Problem \ref{p4}.}
	\label{fig:5}
\end{figure}
\begin{table}[]\caption{$l_2$-Error ratio $||E||_{(\ref{eq32})}/||E||_{\cite{nabavi2007new}}$ for $K\approx 10$ and $K\approx 800$ in Problem \ref{p4}.} \label{tab:3}
	\centering
	\vspace{0.1cm}
	\begin{tabular}{{c c c c c c c c}}
		\hline
		\multirow{2}{*}{$1/h$} &\multicolumn{2}{c}{$||E||_{(\ref{eq32})}/||E||_{\cite{nabavi2007new}}$}\\[0.05cm]
		\hhline{~----}& $K\approx 10$  & $K\approx 800$\\[0.05cm]
		\hline
		512	 & 1.0555  & 0.1623 \\ [0.05cm] 
		1024 & 0.7434  & 0.0425 \\[0.05cm]
		2048 & 1.0047  & 0.0100 \\
		\hline    	  
	\end{tabular}
\end{table}
\begin{table}[]\caption{Number of grid points with the wave numbers for a given accuracy: Problem \ref{p4}.} \label{tab:4}
	\centering
	\vspace{0.1cm}
	\begin{tabular}{{c c c c c c c c}}	
		\hline
		\multirow{2}{*}{$N=1.3626\times K^{7/6}$}  & \multicolumn{3}{c}{$l_{\infty}$-error}\\[0.05cm]
		\hhline{~----}& $K$ & 6th order \cite{nabavi2007new} & new scheme (\ref{eq32})\\[0.05cm]
		\hline
		20  & 10    & 1.06e-06 &  1.59e-07\\[0.05cm]
		80  & 32.98 & 2.93e-06 &  5.08e-08\\[0.05cm]
		140 & 53.49 & 1.84e-06 &  1.18e-07\\[0.05cm]
		200 & 72.32 & 1.88e-06 &  4.63e-08\\[0.05cm]
		260 & 89.99 & 2.25e-06 &  1.96e-08\\[0.05cm]
		320 & 107.9 & 3.67e-06 &  3.24e-08\\
		\hline    	  
	\end{tabular}
\end{table}
\begin{figure}[]
	\begin{minipage}{.5\textwidth}
		\centering
		\includegraphics[width=1.0\linewidth]{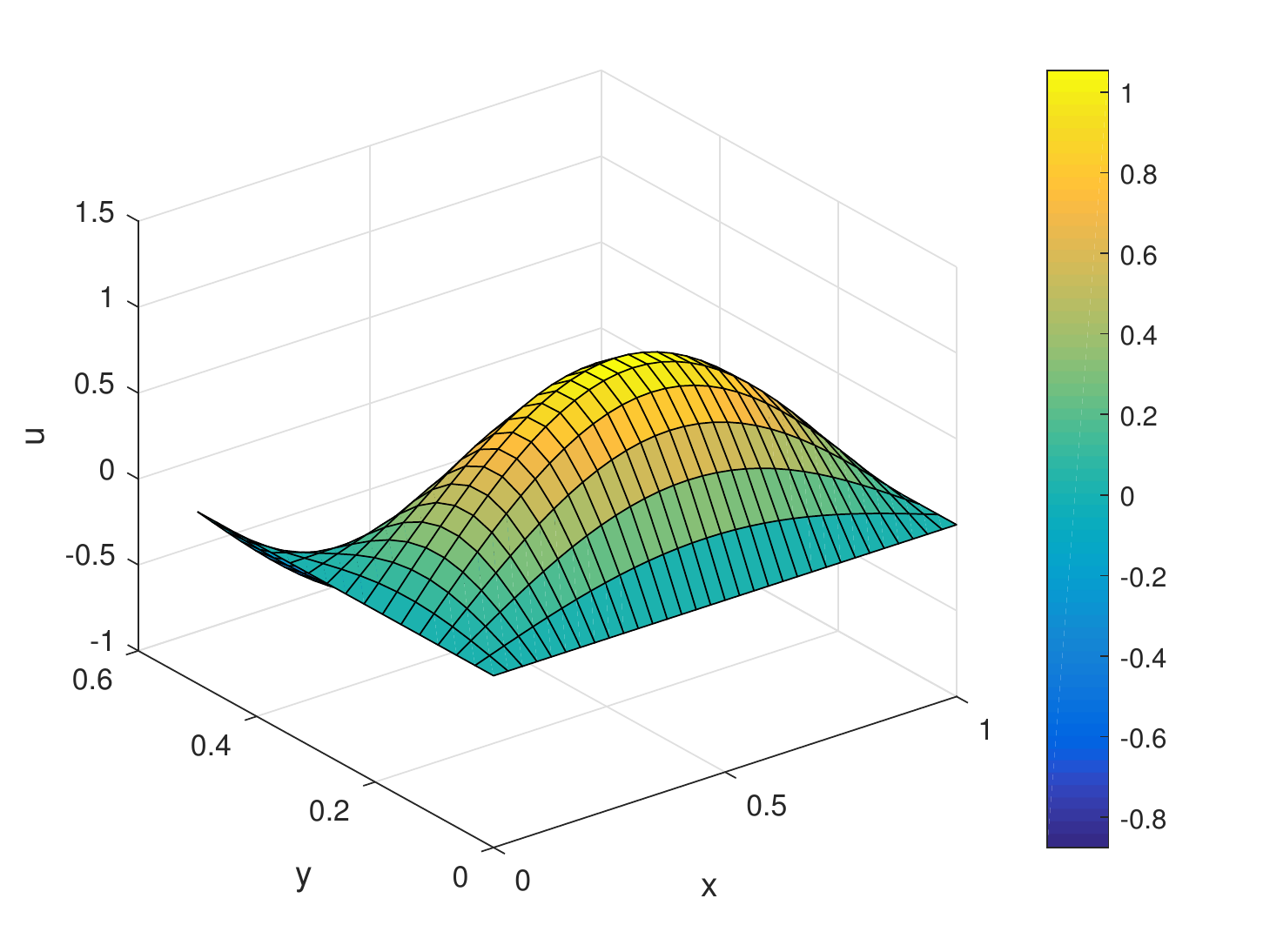}
		\subcaption{}
	\end{minipage}	
	\begin{minipage}{.5\textwidth}
		\centering
		\includegraphics[width=1.0\linewidth]{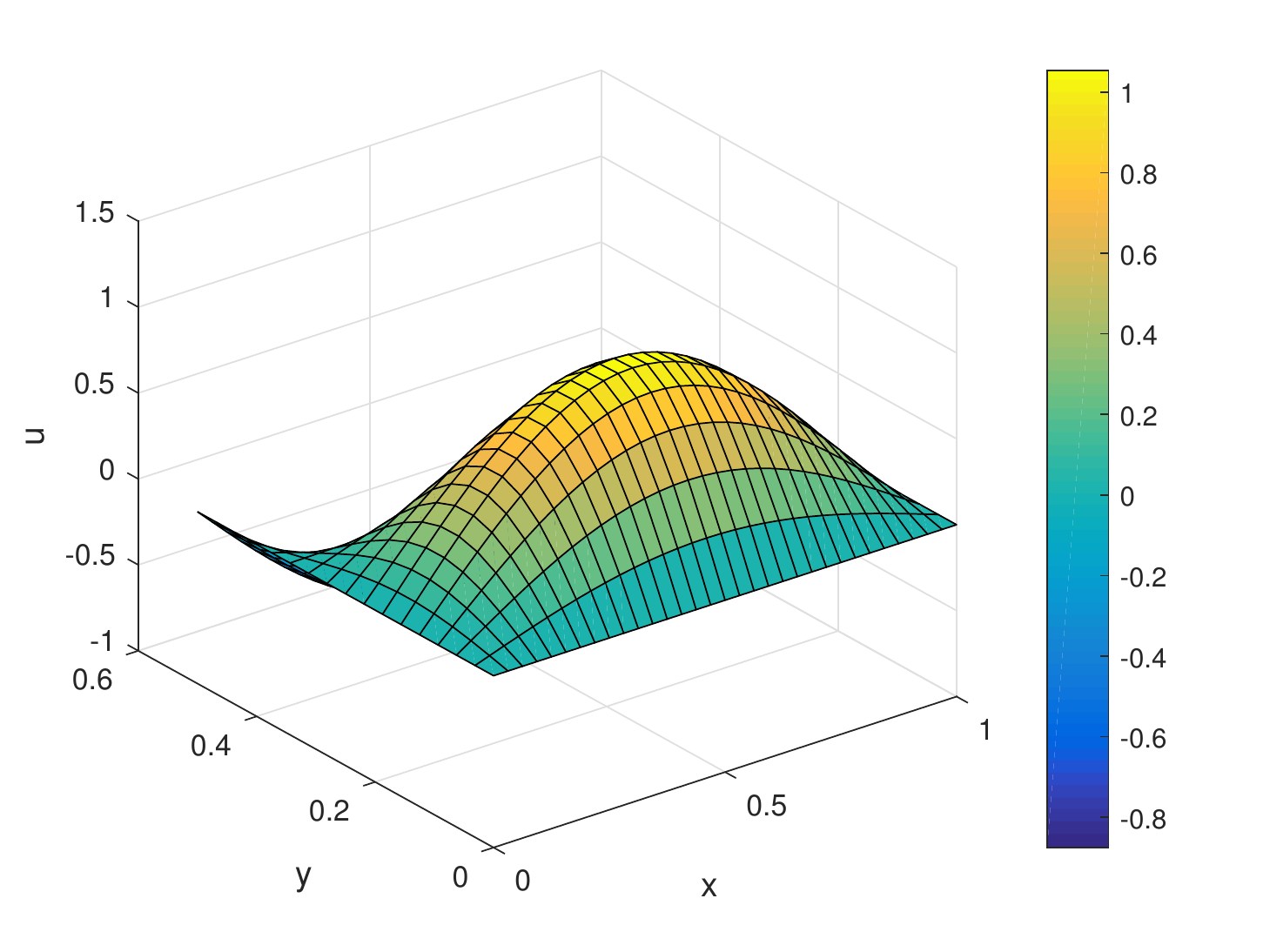}
		\subcaption{}
	\end{minipage}
	\caption{(a) Exact solution and (b) Numerical solution with $N=32, K=10$: Problem \ref{p4}.}
	\label{fig:2Dsolex1}
\end{figure}

In Table \ref{tab:4}, the pollution formula $N=CK^{(p+1)/p}$ is verified for the Problem \ref{p4}, where $p$ is the order of the finite difference scheme (p=6 in this case), N is the no of mesh grids along $x$- direction and $C$ is a constant that depends on the  accuracy of the scheme. Table \ref{tab:4} shows that the number of grid points required for a given accuracy increases with the wave number $K$ \cite{babuska1997pollution,bayliss1985accuracy}. The constant $C=20/(10)^{7/6}$ is calculated by the base value $N=20$ and $K=10$. From the Table \ref{tab:4}, it is clear that the new scheme (\ref{eq32}) provides high accuracy with the validation of the pollution formula. 
\begin{problem}
	We consider the test case in Problem \ref{p4} with the Neumann boundary condition
	\begin{equation*}\label{NM2D}
	u(0,y)=u(1,y)=u(x,0)=0, \partial_yu|_{y=1/2}=0.
	\end{equation*}
	with the same analytic solution as in  Problem \ref{p4}, where $l$ is an odd number such that $K^2=\pi^2(1+l^2)$.
\end{problem}
In order to get the sixth order accuracy, the sixth order approximation (\ref{NBC5}) for the Neumann boundary is used with the difference schemes. The maximum error norms and computational order of convergence for the Neumann problem \ref{NM2D} are computed in Tables \ref{tab:2DNM1}-\ref{tab:2DNM2}. For this test case, the tolerance for the iteration stopping criteria is taken as $1.0e-10$. Table \ref{tab:2DNM1} shows that both the schemes are sixth order accurate and the new scheme is more accurate than (\ref{eq26}). From Table \ref{tab:2DNM2}, it is also clear that the new scheme (\ref{eq32}) is highly accurate in case of large wave numbers with sufficiently small $Kh$. We also plotted the exact and numerical solution values in Figure \ref{fig:2Dsolex2} for the Problem \ref{NM2D}.
\begin{table}[]\caption{Error norms to the solution values of Problem \ref{NM2D} for $l=7, K\approx 22$.} \label{tab:2DNM1}
	\centering
	\vspace{0.1cm}
	\begin{tabular}{{c c c c c c c c c c}}
		\hline
		\multirow{2}{*}{$1/h$} &\multicolumn{3}{c}{new scheme (\ref{eq32})}&&\multicolumn{3}{c}{6th order \cite{nabavi2007new}}\\[0.05cm]
		\hhline{~----~----}& $||\epsilon||_{\infty}$  & $||\epsilon||_2$ & $\Theta_{\infty}$& & $||\epsilon||_{\infty}$ & $||\epsilon||_2$ & $\Theta_{\infty}$ \\[0.05cm]
		\hline
		\vspace{0.1cm}
		16 	&1.42e-03 &6.91e-04 & -Inf&&2.20e-02 & 1.07e-02 & -Inf\\[0.05cm]
		32 	&1.11e-05 &5.67e-06 & 6.99&&2.61e-04 & 1.33e-04 & 6.40\\[0.05cm]
		64 	&2.74e-07 &1.40e-07 & 5.34&&3.85e-06 & 1.96e-06 & 6.09 \\[0.05cm]
		128 &4.65e-09 &2.35e-09 & 5.88&&5.93e-08 & 3.00e-08 & 6.02 \\
		\hline    	  
	\end{tabular}
\end{table}
\begin{table}[]\caption{Error norms to the solution values of Problem \ref{NM2D} for $l=127, K\approx 400$.} \label{tab:2DNM2}
	\centering
	\vspace{0.1cm}
	\begin{tabular}{{c c c c c c c c c c}}
		\hline
		\multirow{2}{*}{$1/h$} &\multicolumn{3}{c}{new scheme (\ref{eq32})}&&\multicolumn{3}{c}{6th order \cite{nabavi2007new}}\\[0.05cm]
		\hhline{~---~---}& $||\epsilon||_{\infty}$ & $||\epsilon||_2$ & $\Theta_{\infty}$ && $||\epsilon||_{\infty}$& $||\epsilon||_2$ & $\Theta_{\infty}$ \\[0.05cm]
		\hline
		512	& 1.93e-02& 9.66e-0  & -Inf && 3.13e-01& 1.57e-01 & -Inf\\[0.05cm]
		1024& 7.28e-05& 3.65e-05 & 8.05 && 6.57e-03& 3.29e-03 & 5.57\\[0.05cm]
		2048& 2.88e-07& 1.44e-07 & 7.98 && 1.01e-04& 5.08e-05 & 6.02\\
		\hline    	  
	\end{tabular}
\end{table}
\begin{figure}[]
	\begin{minipage}{.5\textwidth}
		\centering
		\includegraphics[width=1.0\linewidth]{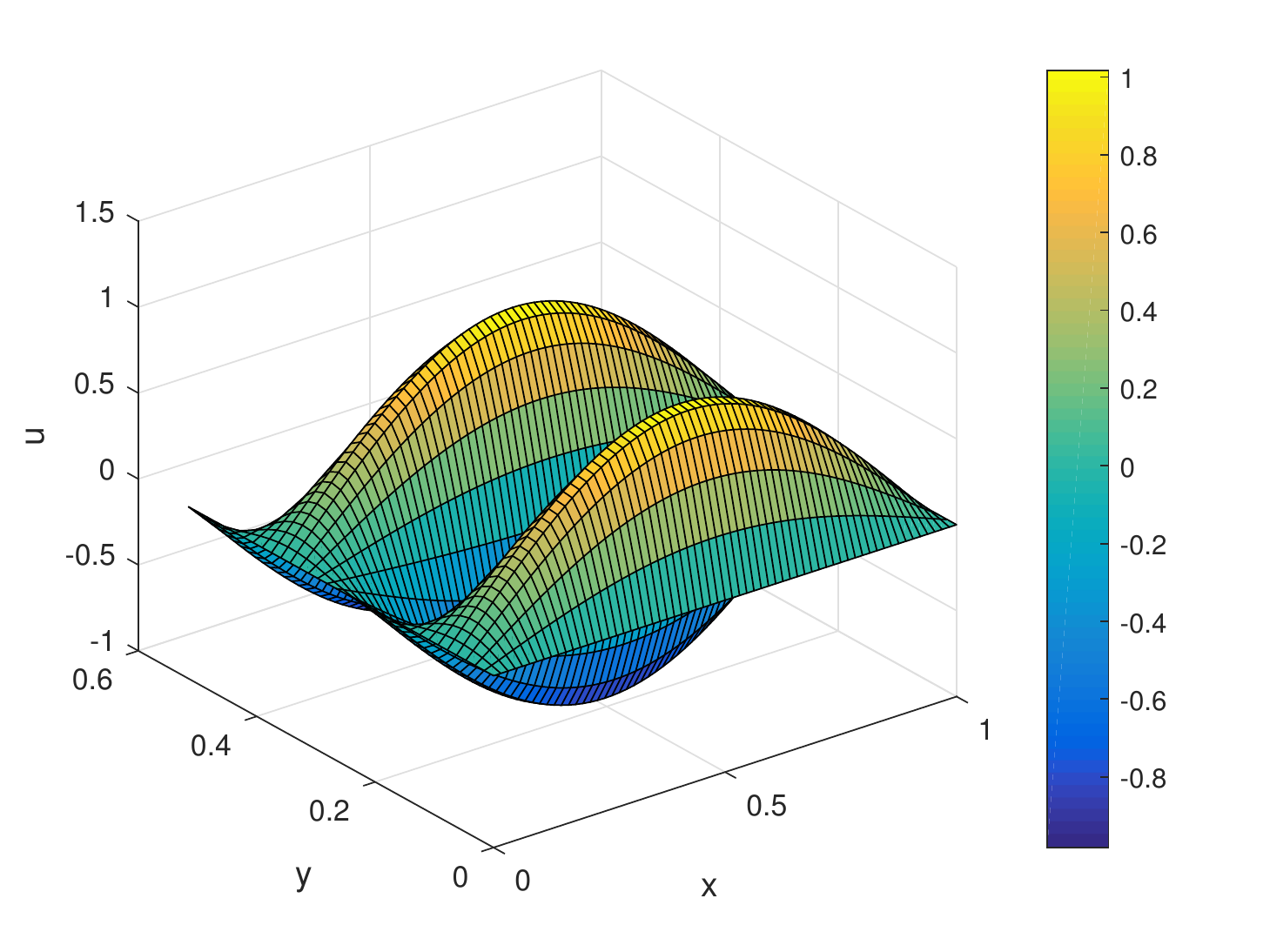}
		\subcaption{}
	\end{minipage}	
	\begin{minipage}{.5\textwidth}
		\centering
		\includegraphics[width=1.0\linewidth]{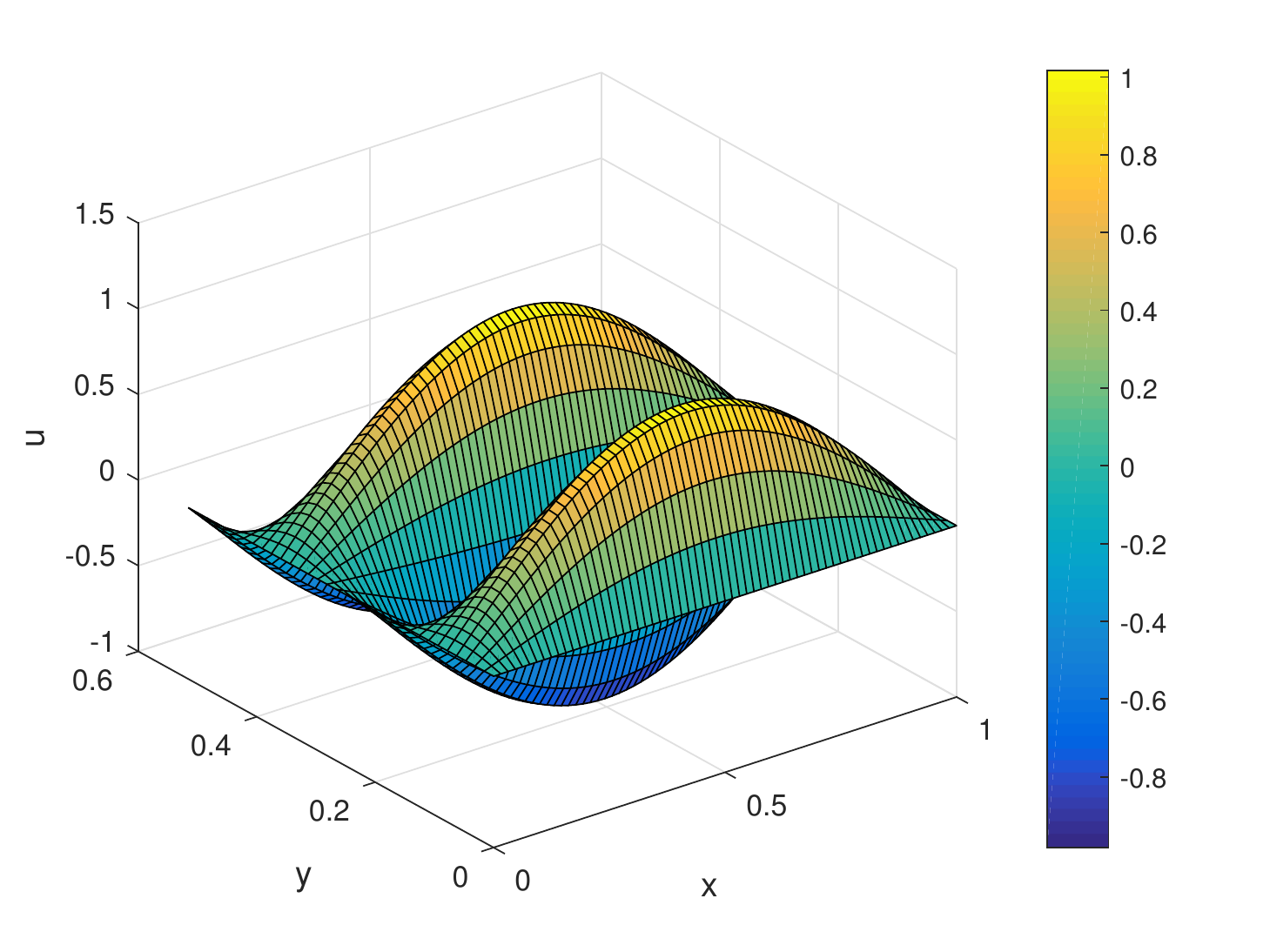}
		\subcaption{}
	\end{minipage}
	\caption{(a) Exact and (b) numerical solution with $N=64, K=22$ for Problem \ref{NM2D}.}
	\label{fig:2Dsolex2}
\end{figure}
\begin{problem}\label{p5}
	Consider the model problem \cite{fu2008compact}
	\begin{equation*}\label{eq38}
	\begin{split}
	&\nabla^2u+K^2u=(K^2-\pi^2-K^2 \pi^2)\sin\pi x\sin K\pi y,~ (x,y)\in\Omega\\
	&u|_{\partial\Omega}=0,
	\end{split}
	\end{equation*}
	where $\Omega=[0,1]\times[0,1]$ is a square domain with boundary $\partial\Omega$. 
	The analytic solution is given by $u(x,y)=\sin\pi x\sin K\pi y$.
\end{problem}
The maximum error norms and computational order of convergence for the scheme (\ref{eq32}) and (\ref{eq26}) are computed in Tables \ref{tab:5}-\ref{tab:6}. Table \ref{tab:5} shows that both the schemes are sixth order accurate and the new scheme is comparatively more accurate than (\ref{eq26}). From Table \ref{tab:6}, it is also clear that the new scheme (\ref{eq32}) maintained its sixth order convergence rate even for large enough wave numbers $K$ such that $Kh$ is sufficiently small. From Table \ref{tab:5}, it is also observed that the computed result obtained by both schemes are much more accurate than those obtained by fourth order difference scheme \cite{fu2008compact}.

We plotted the accuracy and numerical errors for the new scheme and the standard sixth order scheme \cite{nabavi2007new} in Figures \ref{fig:or1}-\ref{fig:6}. From Figure \ref{fig:or1}, it is clear that the slope of each line is six. In Figure \ref{fig:6}, it is shown that the numerical errors are growing with the increasing $K$ for both the scheme and the numerical errors developed by the new scheme is less than that of the scheme \cite{nabavi2007new}, which confirms the accuracy of the new scheme. We also plotted the exact and numerical solution values in Figure \ref{fig:2Dsolex3}.
\begin{table}[]\caption{Max-error to the solution values of Problem \ref{p5} for $K=30$.} \label{tab:5}
	\centering
	\vspace{0.1cm}
	\begin{tabular}{{c c c c c c c c c c}}
		\hline
		\multirow{2}{*}{$1/h$} &\multicolumn{2}{c}{new scheme (\ref{eq32})}&&\multicolumn{2}{c}{6th order \cite{nabavi2007new}}&&\multicolumn{2}{c}{4th order \cite{fu2008compact}}\\[0.05cm]
		\hhline{~--~--~---}& $||\epsilon||_{\infty}$ & $\Theta_{\infty}$ && $||\epsilon||_{\infty}$ & $\Theta_{\infty}$&&$||\epsilon||_{\infty}$ & $\Theta_{\infty}$ \\[0.05cm]
		\hline
		32 & 9.3715e-02  & -Inf && 1.0534e-01 & -Inf && 3.61e-01 & -Inf\\[0.05cm]
		64 & 6.7252e-04  & 7.12 && 7.4942e-04 & 7.14 && 1.28e-02 & 4.82\\[0.05cm]
		128& 8.7385e-06  & 6.27 && 9.7186e-06 & 6.27 && 7.10e-04 & 4.17\\[0.05cm]
		256& 1.2846e-07  & 6.09 && 1.4280e-07 & 6.09 && 4.31e-05 & 4.04\\[0.05cm]
		512& 1.9761e-09  & 6.02 && 2.1963e-09 & 6.02 && 2.68e-06 & 4.01\\
		\hline    	  
	\end{tabular}
\end{table}
\begin{figure}[]
	\begin{minipage}{.5\textwidth}
		\centering
		\includegraphics[width=1.0\linewidth]{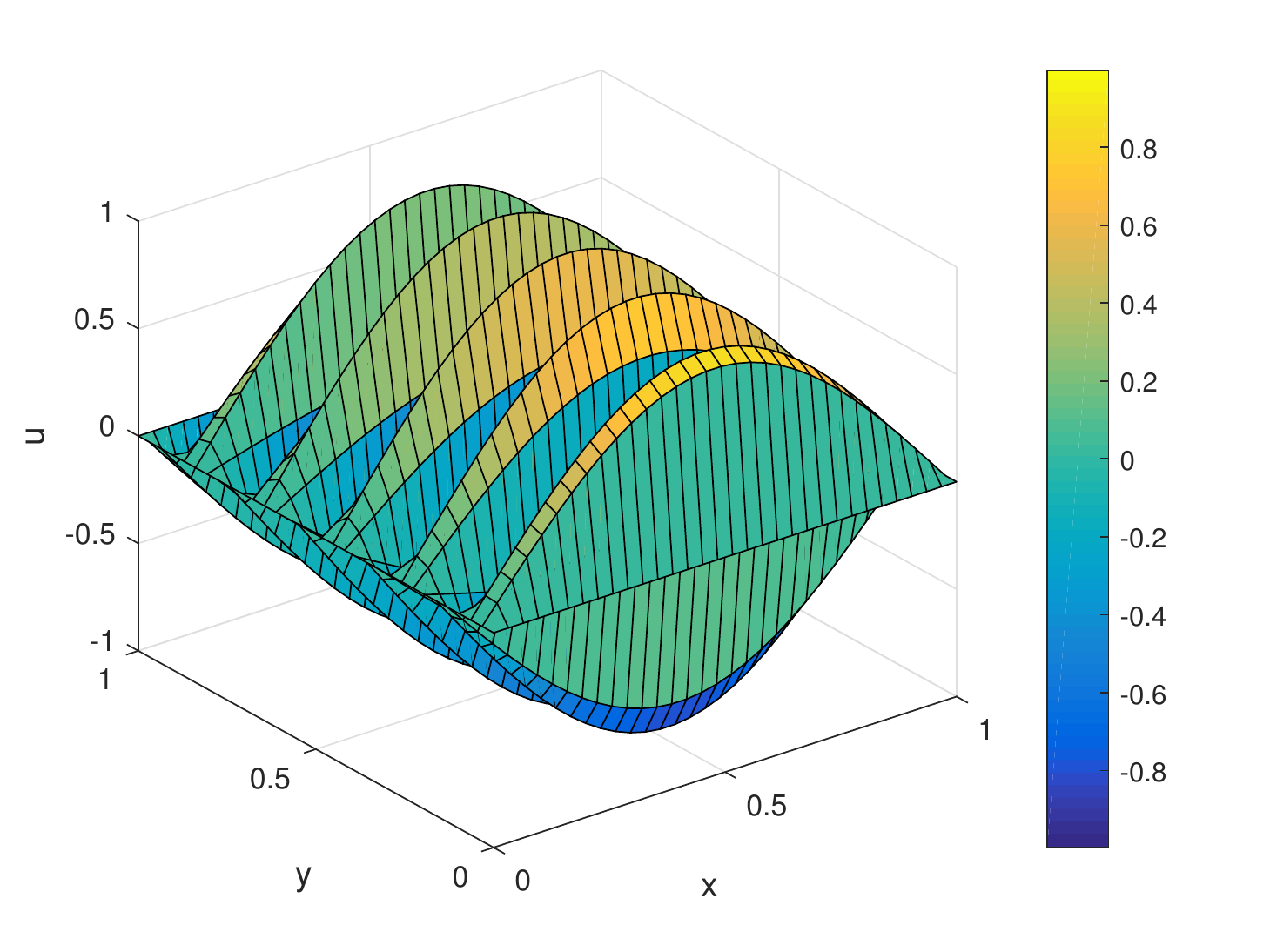}
		\subcaption{}
	\end{minipage}	
	\begin{minipage}{.5\textwidth}
		\centering
		\includegraphics[width=1.0\linewidth]{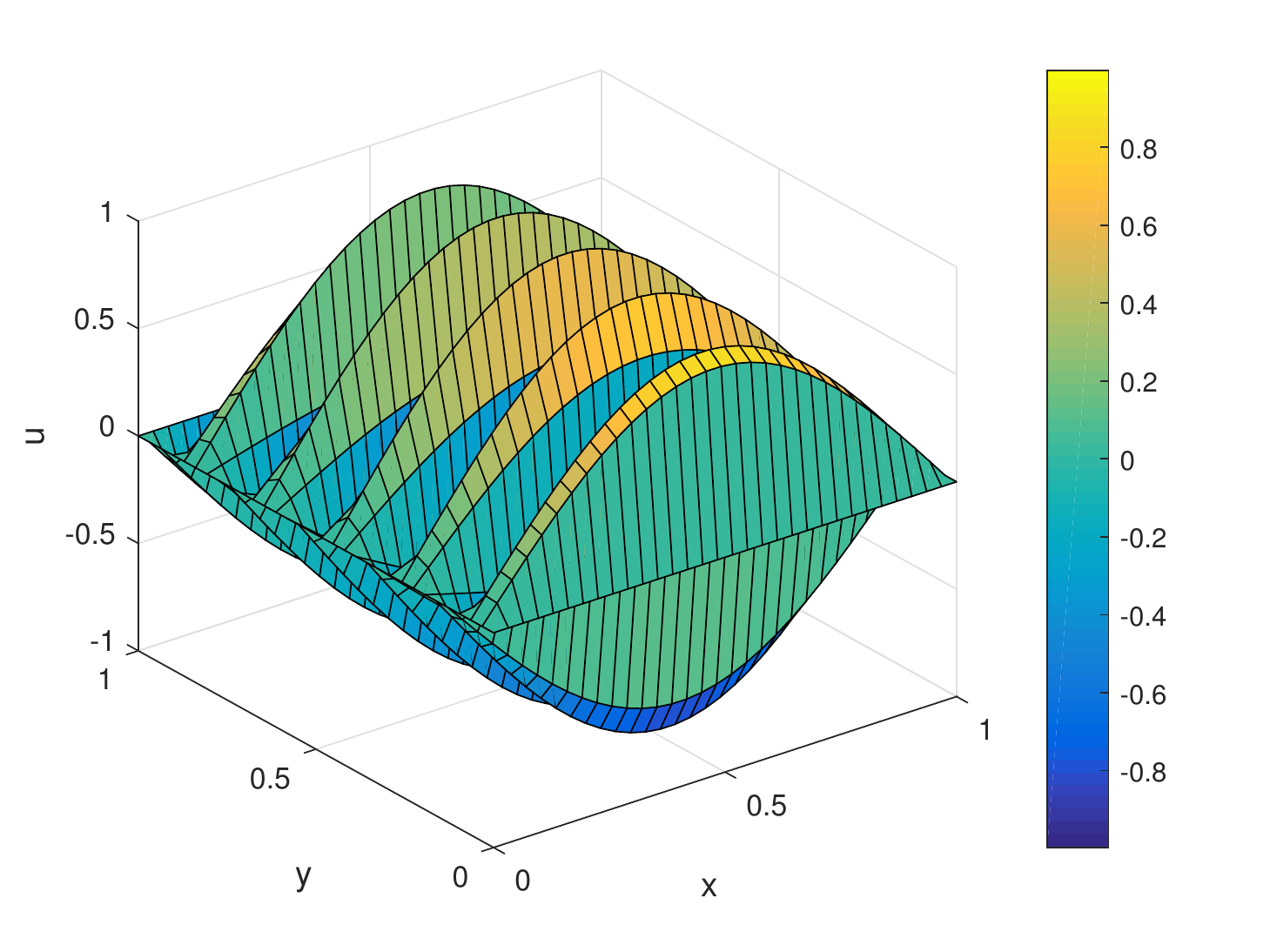}
		\subcaption{}
	\end{minipage}
	\caption{(a) Exact and (b) Numerical solution with $N=32, K=10$ for Problem \ref{p5}.}
	\label{fig:2Dsolex3}
\end{figure}
\begin{problem}\label{p6}
	Consider the model problem \cite{singer2006sixth}
	\begin{equation*}
	\begin{split}
	&\nabla^2u+K^2u=(K^2-l^2-m^2)\sin lx\sin my,~ (x,y)\in[0,\pi]\times[0,\pi]
	\end{split}
	\end{equation*}
	with the boundary conditions $u(0,y)=u(\pi,y)=u(x,0)=u(x,\pi)=0$. Where $l, m$ are positive integers. The analytic solution is given by $u(x,y)=\sin lx\sin my$.
\end{problem}
\begin{figure}
	\centering
	\includegraphics[scale=0.4]{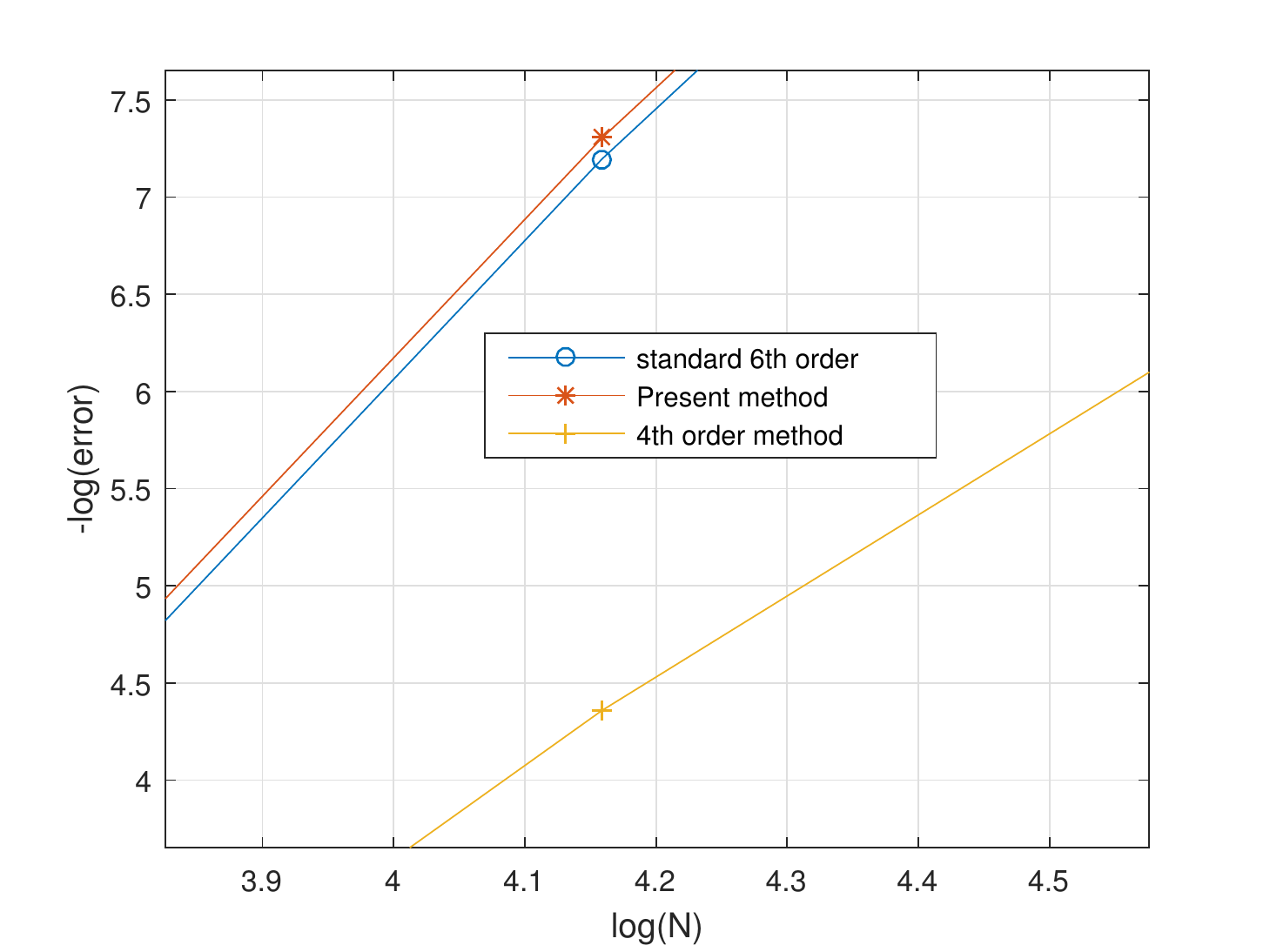}
	\caption{Order of accuracy for $K=30$ for Problem \ref{p5}.}
	\label{fig:or1}	
\end{figure}	
\begin{table}[]\caption{Order of convergence for $K=300$ in Problem \ref{p5}.} \label{tab:6}
	\centering
	\vspace{0.1cm}
	\begin{tabular}{{c c c c c c c c c c}}
		\hline
		\multirow{2}{*}{$1/h$} &\multicolumn{2}{c}{new scheme (\ref{eq32})}&&\multicolumn{2}{c}{6th order \cite{nabavi2007new}}\\[0.05cm]
		\hhline{~--~-----}& $||\epsilon||_{\infty}$& $\Theta_{\infty}$ && $||\epsilon||_{\infty}$& $\Theta_{\infty}$\\[0.05cm]
		\hline
		512 & 2.5966e-03  &      && 2.9009e-03 &     \\[0.05cm]
		1024& 3.2478e-05  & 6.32 && 3.6175e-05 & 6.33\\[0.05cm]
		2048& 4.8063e-07  & 6.08 && 5.3494e-07 & 6.08\\
		\hline    	  
	\end{tabular}
\end{table}	
\begin{figure}[]
	\begin{minipage}{.5\textwidth}
		\centering
		\includegraphics[width=1.0\linewidth]{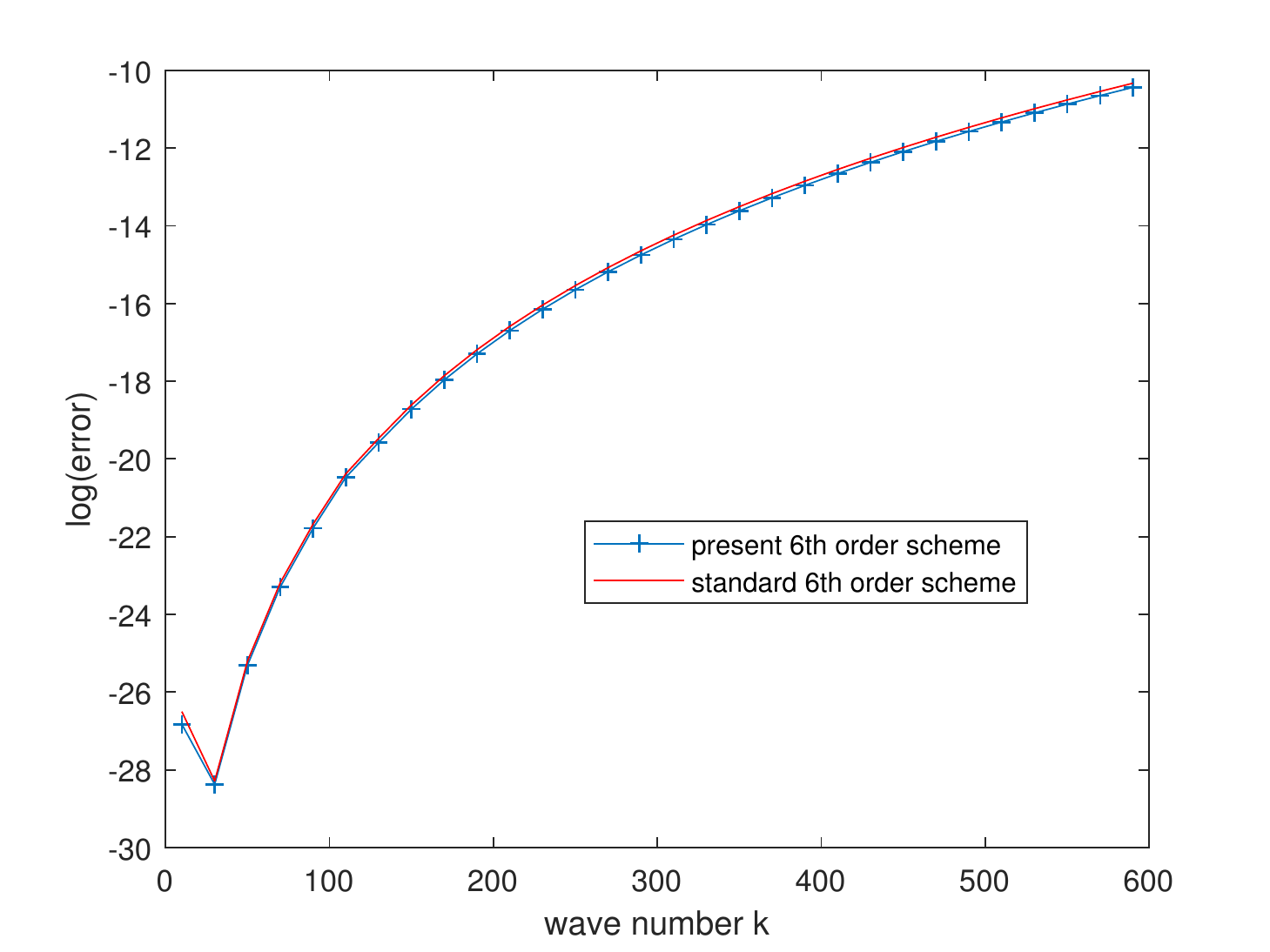}
		\subcaption{}
	\end{minipage}	
	\begin{minipage}{.5\textwidth}
		\centering
		\includegraphics[width=1.0\linewidth]{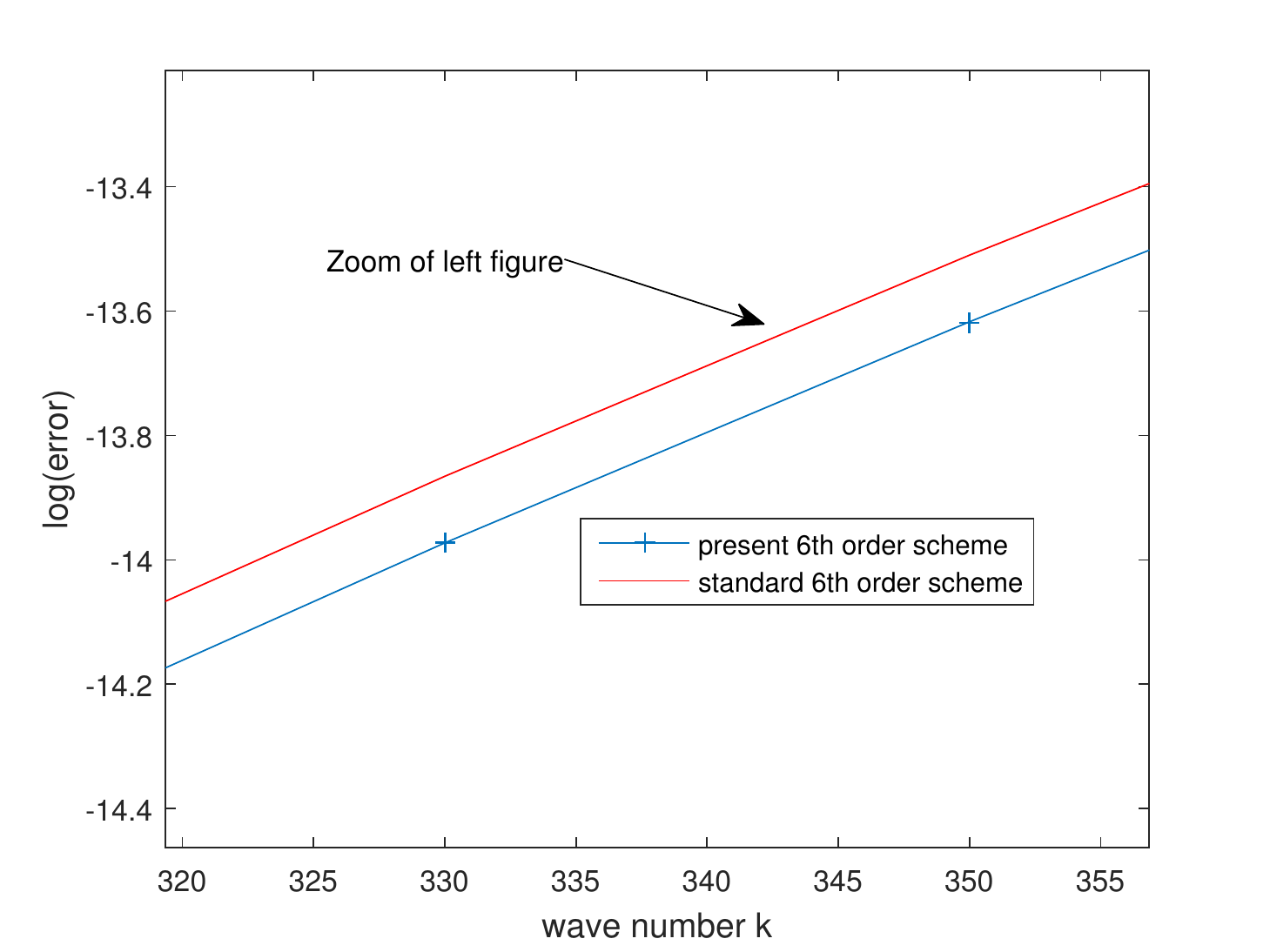}
		\subcaption{}
	\end{minipage}
	\caption{Numerical errors for N=2048 w.r.t $K$ for Problem \ref{p5}.}
	\label{fig:6}
\end{figure}
The error norms and computational order of convergence for the scheme (\ref{eq32}) and scheme \cite{nabavi2007new} are computed in Tables \ref{tab:7}-\ref{tab:8}. Table \ref{tab:7} shows that both the schemes are sixth order accurate and the new scheme is more accurate than (\ref{eq26}). From Table \ref{tab:8}, it is also clear that the new scheme (\ref{eq32}) is highly accurate for large wave numbers $K$ such that $Kh$ is sufficiently small.

The numerical errors and accuracy of the new scheme are plotted in Figures \ref{fig:errorpr6}-\ref{fig:orderpr6}. Figure \ref{fig:errorpr6}, clearly shows that the numerical errors are growing with the increasing $K$ for both the scheme and the numerical errors developed by the new scheme is less than that of the scheme \cite{nabavi2007new}. From Figure \ref{fig:orderpr6}, it is clear that the slope of the line for new scheme is more than that of the scheme \cite{nabavi2007new} for large wave number $K$ which confirms the accuracy of the new scheme for the Problem \ref{p6} in case of large $K$.  The solution values for both the exact and numerical solution values are also  plotted in Figure \ref{fig:2Dsolex4} for the Problem \ref{p6}.
\begin{table}[]\caption{Error norms to the solution values of Problem \ref{p6} with $l=100, m=50, K\approx 112$.} \label{tab:7}
	\centering
	\vspace{0.1cm}
	\begin{tabular}{{c c c c c c c c c c}}
		\hline
		\multirow{2}{*}{$1/h$} &\multicolumn{3}{c}{new scheme (\ref{eq32})}&&\multicolumn{3}{c}{6th order \cite{nabavi2007new}}\\[0.05cm]
		\hhline{~---~----}& $||\epsilon||_{\infty}$  & $||\epsilon||_2$& $\Theta_{\infty}$& & $||\epsilon||_{\infty}$ & $||\epsilon||_2$& $\Theta_{\infty}$ \\[0.05cm]
		\hline
		\vspace{0.1cm}
		256 & 2.6355e-02 & 1.6781e-02 & -Inf && 6.9254e-02 & 4.4095e-02 & -Inf\\[0.05cm]
		512	& 2.7938e-04 & 1.7786e-04 & 6.56 && 8.5743e-04 & 5.4586e-04 & 6.34\\[0.05cm]
		1024& 3.9206e-06 & 2.4959e-06 & 6.16 && 1.2417e-05 & 7.9048e-06 & 6.11\\
		\hline    	  
	\end{tabular}
\end{table}
\begin{table}[]\caption{Error norms to the solution values of Problem \ref{p6} with $l=500, m=100, K\approx 510$.} \label{tab:8}
	\centering
	\vspace{0.1cm}
	\begin{tabular}{{c c c c c c c c c c}}
		\hline
		\multirow{2}{*}{$1/h$} &\multicolumn{3}{c}{new scheme (\ref{eq32})}&&\multicolumn{3}{c}{6th order \cite{nabavi2007new}}\\[0.05cm]
		\hhline{~---~----}& $||\epsilon||_{\infty}$  & $||\epsilon||_2$& $\Theta_{\infty}$& & $||\epsilon||_{\infty}$ & $||\epsilon||_2$& $\Theta_{\infty}$ \\[0.05cm]
		\hline
		\vspace{0.1cm}
		1024& 3.0193e-01 & 1.9222e-01 & -Inf && 7.1572e-01 & 4.5564e-01 & -Inf\\[0.05cm]
		2048& 1.4838e-03 & 9.4459e-04 & 7.67 && 2.7449e-02 & 1.7474e-02 & 4.70\\[0.05cm]
		4096& 8.5049e-06 & 5.4144e-06 & 7.45 && 4.0781e-04 & 2.5962e-04 & 6.07\\
		\hline    	  
	\end{tabular}
\end{table}
\begin{figure}[]
	\begin{minipage}{.5\textwidth}
		\centering
		\includegraphics[width=1.0\linewidth]{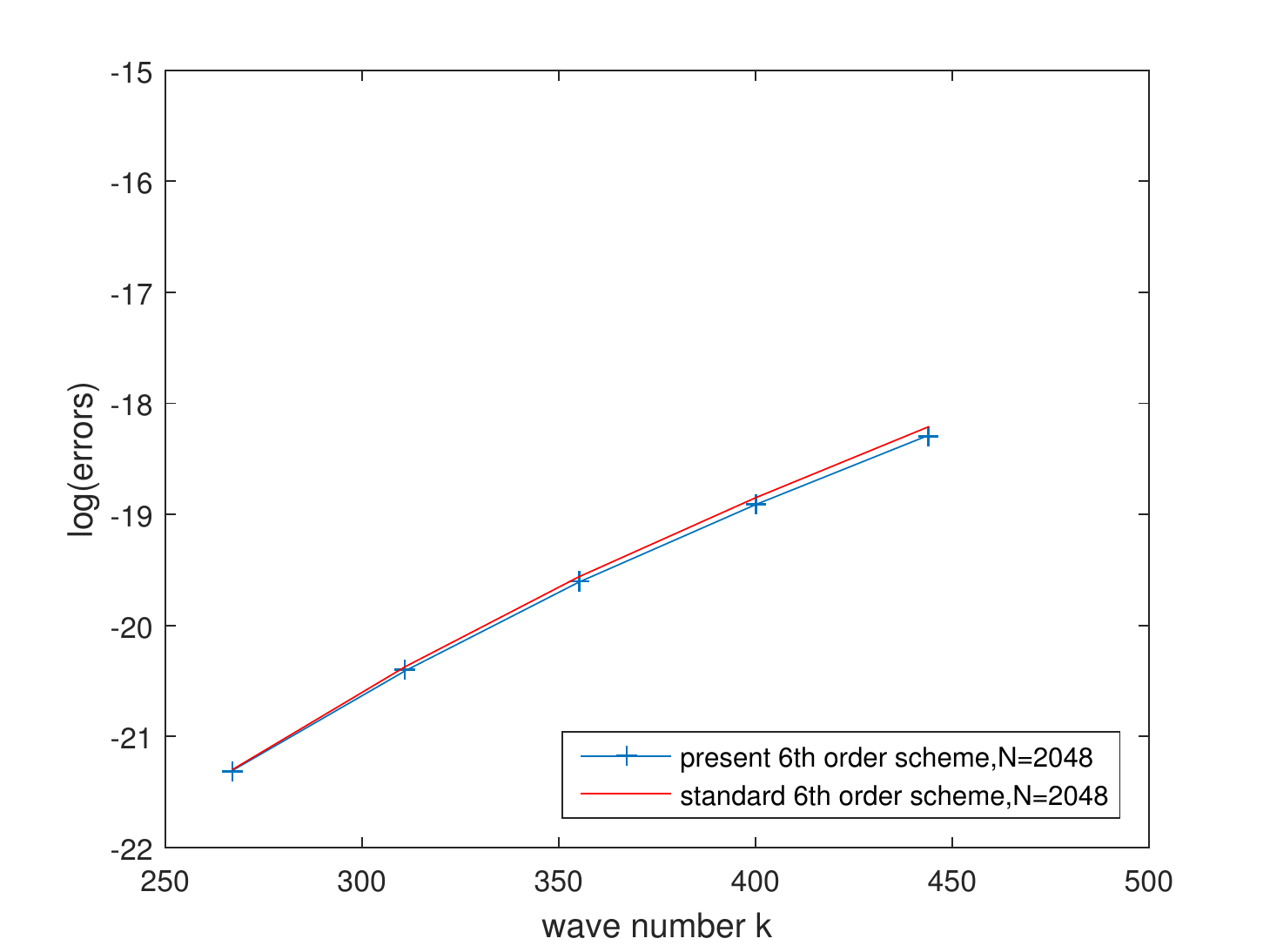}
		\subcaption{}
	\end{minipage}	
	\begin{minipage}{.5\textwidth}
		\centering
		\includegraphics[width=1.0\linewidth]{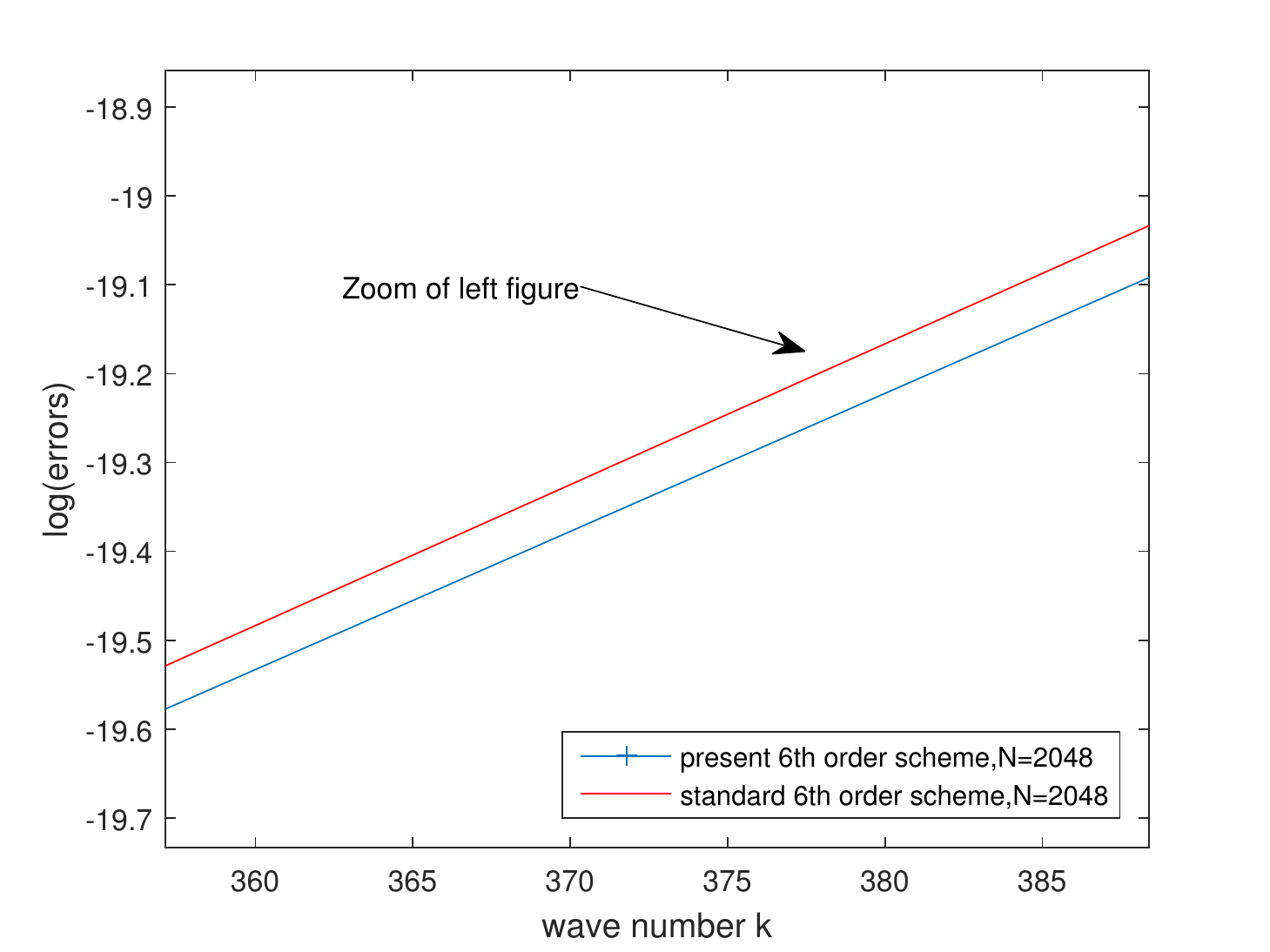}
		\subcaption{}
	\end{minipage}
	\caption{Numerical errors for N=2048 with different $K$ for Problem \ref{p6}.}
	\label{fig:errorpr6}
\end{figure}
\begin{figure}
	\centering
	\includegraphics[scale=0.4]{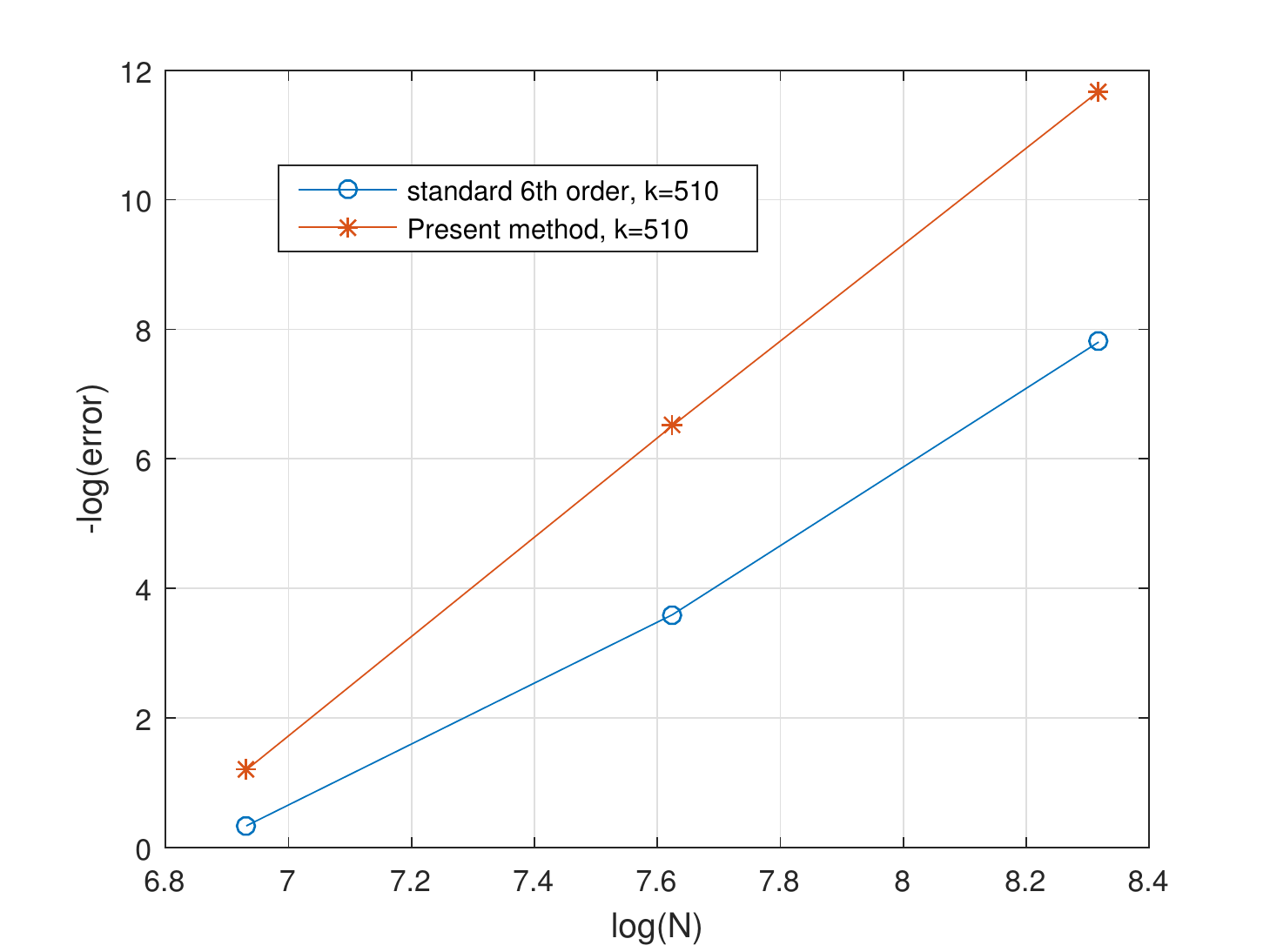}
	\caption{Accuracy of the new scheme with large $K$ for Problem \ref{p6}.}
	\label{fig:orderpr6}	
\end{figure}
\begin{figure}[]
	\begin{minipage}{.5\textwidth}
		\centering
		\includegraphics[width=1.0\linewidth]{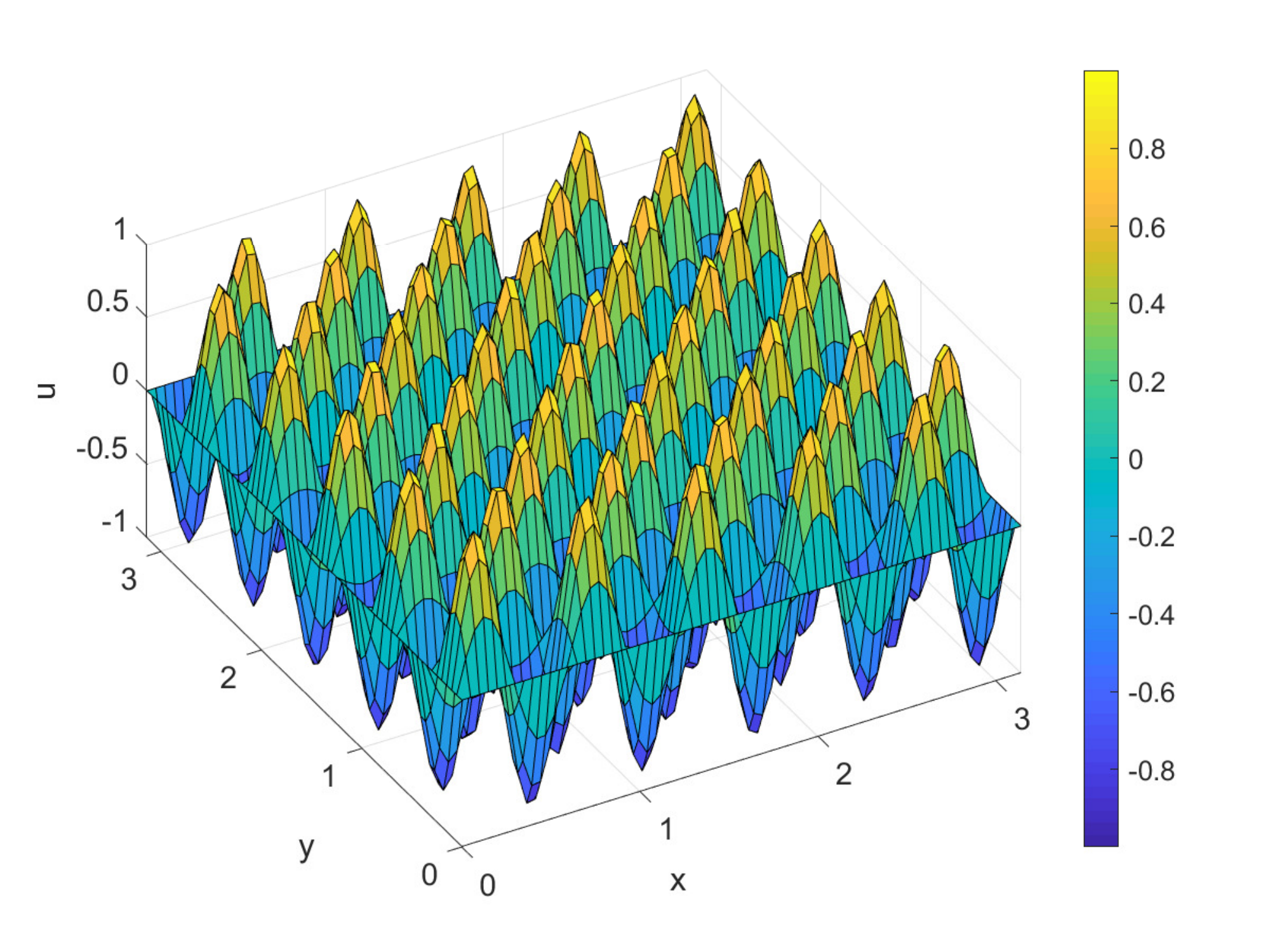}
		\subcaption{}
	\end{minipage}	
	\begin{minipage}{.5\textwidth}
		\centering
		\includegraphics[width=1.0\linewidth]{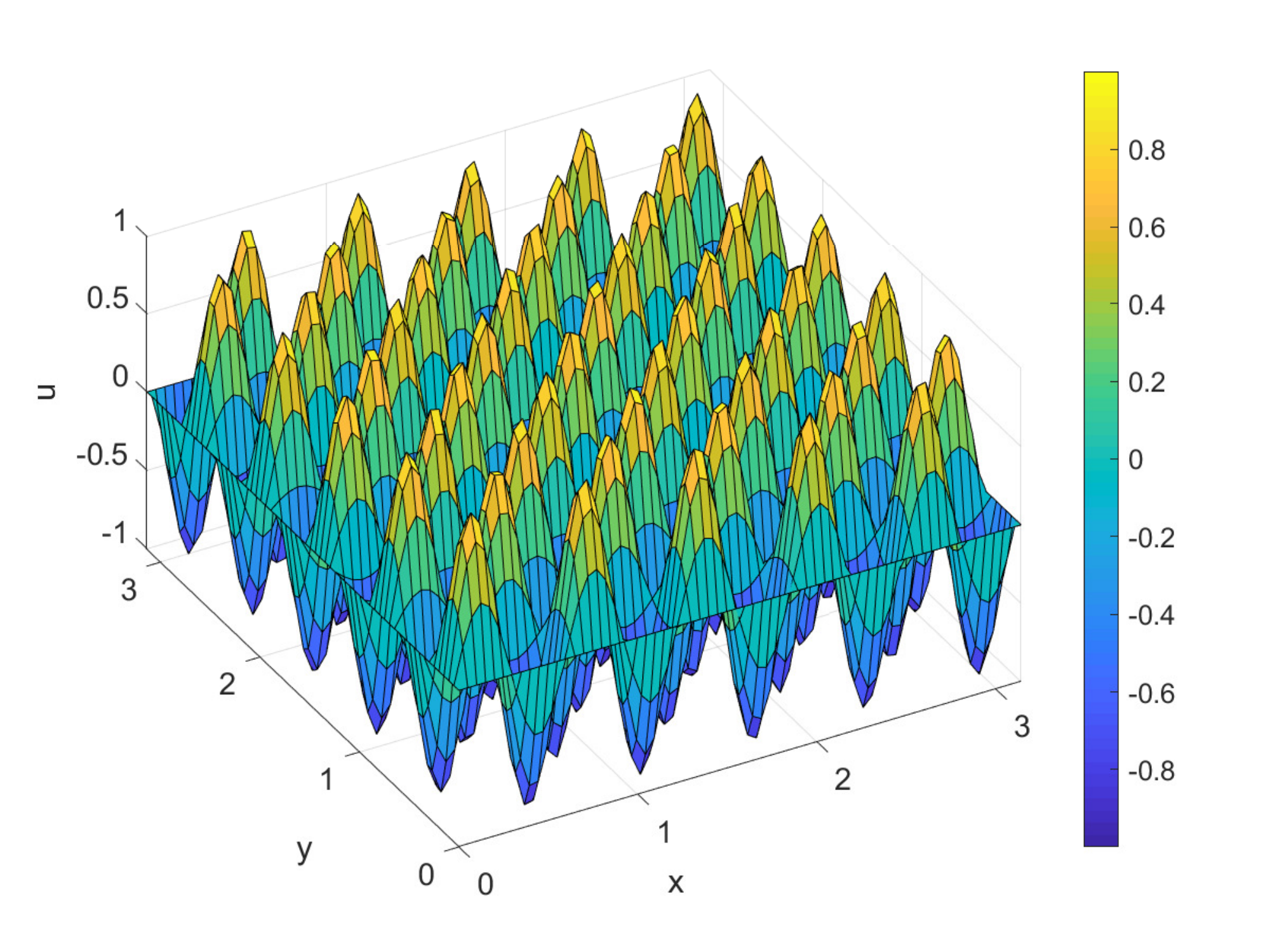}
		\subcaption{}
	\end{minipage}
	\caption{(a) Exact, (b) Numerical solution with $N=64, l=m=10, K=14$: Problem \ref{p6}.}
	\label{fig:2Dsolex4}
\end{figure}
\subsection{Three-dimensional test cases}
For the three dimensional case, three test cases are presented for the Helmholtz equation in three dimensions. Hybrid BiCG method is used to solve system of equations. Results obtained from the new scheme (\ref{eq3D28}) are compared with the sixth order scheme \cite{sutmann2007compact} given by (\ref{eq3D17}) using some error norms and order of convergence. The formulae (\ref{MAE2D})-(\ref{Order2D}) can be extended in three dimensions to calculate the error norms and order of convergence for the difference schemes. Graphically, it is also shown that the numerical errors are developed with increasing the wave numbers $K$. The exact and numerical solutions for the corresponding difference schemes are also plotted at the given grids.
\begin{problem}\label{p1:3D}
	We extended the test problem from Fu\cite{fu2008compact} into three dimensions with the analytic solution as
	\begin{equation*}
	u(x,y,z)=\sin\pi x\sin\pi y\sin l\pi z+\frac{\sin\pi x\sin\pi y\sin\pi z}{l^2-1}, (x,y,z)\in\Omega,
	\end{equation*}
	and the source function as
	\begin{equation*}
	f(x,y,z)=\pi^2\sin\pi x\sin\pi y\sin\pi z, (x,y,z)\in\Omega,
	\end{equation*}
	where $\Omega=[0,1]\times[0,1]\times[0,1/2]$, $K^2=\pi^2(2+l^2)$, $l$ is odd number. The Dirichlet boundary conditions can be obtained from the analytic solution.
\end{problem}
$l_2$-error and $l_{\infty}$-error norms for the new scheme (\ref{eq3D28}) and the scheme (\ref{eq3D17}) are given in Tables \ref{tab1:3D}-\ref{tab5:3D} for wave numbers $K\approx 29, 204$. In Figure \ref{fig:3D1}, numerical errors are also plotted for both the scheme with small and very large wave numbers. For this test case, iteration stops when the residual error norm is reduced by $1.0e-11$.

From  tables and figures, it is noted that the numerical errors are developed for both the schemes however new scheme (\ref{eq3D28}) shows more accuracy than that of the scheme (\ref{eq3D17}). Further, it is also noted that the source function, in this test case, does not depend upon the wave number $K$ and therefore from Table \ref{tab5:3D}, it is observed that the new scheme is highly accurate for very large wave number. The solution values for both the exact and numerical solution values are plotted in Figure \ref{fig:3Dsolex1} for the Problem \ref{p1:3D}.
\begin{table}[]\caption{Error norms to the solution values of Problem \ref{p1:3D} for $l=9, K\approx 29$.}\label{tab1:3D}
	\centering
	\vspace{0.1cm}
	\begin{tabular}{{c c c c c c c c c c}}
		\hline
		\multirow{2}{*}{$1/h$} &\multicolumn{3}{c}{new scheme (\ref{eq3D28})}&&\multicolumn{3}{c}{6th order scheme (\ref{eq3D17})}\\[0.05cm]
		\hhline{~---~---}& $||\epsilon||_{\infty}$  & $||\epsilon||_2$ & $\Theta_{\infty}$& & $||\epsilon||_{\infty}$ & $||\epsilon||_2$ & $\Theta_{\infty}$ \\[0.05cm]
		\hline
		\vspace{0.1cm}
		16 	& 5.20e-03 & 1.42e-03  & -Inf && 1.88e-02 & 5.13e-03  & -Inf\\[0.05cm]
		32 	& 2.69e-05 & 6.89e-06  & 7.59 && 1.87e-04 & 4.78e-05  & 6.65\\[0.05cm]
		64 	& 2.95e-07 & 7.06e-08  & 6.51 && 2.73e-06 & 6.55e-07  & 6.09\\[0.05cm]
		128 & 4.15e-09 & 9.70e-10  & 6.15 && 4.20e-08 & 9.83e-09  & 6.02\\
		\hline    	  
	\end{tabular}
\end{table}
\begin{table}[]\caption{Error norms to the solution values of Problem \ref{p1:3D} for $l=65, K\approx 204$.}\label{tab5:3D}
	\centering
	\vspace{0.1cm}
	\begin{tabular}{{c c c c c c c c c c}}
		\hline
		\multirow{2}{*}{$1/h$} &\multicolumn{3}{c}{new scheme (\ref{eq3D28})}&&\multicolumn{3}{c}{6th order scheme (\ref{eq3D17})}\\[0.05cm]
		\hhline{~---~---}& $||\epsilon||_{\infty}$  & $||\epsilon||_2$ & $\Theta_{\infty}$& & $||\epsilon||_{\infty}$ & $||\epsilon||_2$ & $\Theta_{\infty}$ \\[0.05cm]
		\hline
		\vspace{0.1cm}
		64 & 9.06e-01 & 3.27e-01  & -Inf && 9.39e-01 & 3.39e-01 & -Inf\\[0.05cm]
		128& 1.04e-02 & 2.21e-03  & 6.44 && 6.32e-02 & 1.33e-02 & 3.89\\[0.05cm]
		256& 2.59e-05 & 5.43e-06  & 8.66 && 7.07e-04 & 1.48e-04 & 6.48\\[0.05cm]
		512& 2.49e-08 & 5.20e-09  & 10.02&& 1.02e-05 & 2.14e-06 & 6.11\\
		\hline    	  
	\end{tabular}
\end{table}

\begin{figure}[]
	\begin{minipage}{.5\textwidth}
		\centering
		\includegraphics[width=1.0\linewidth]{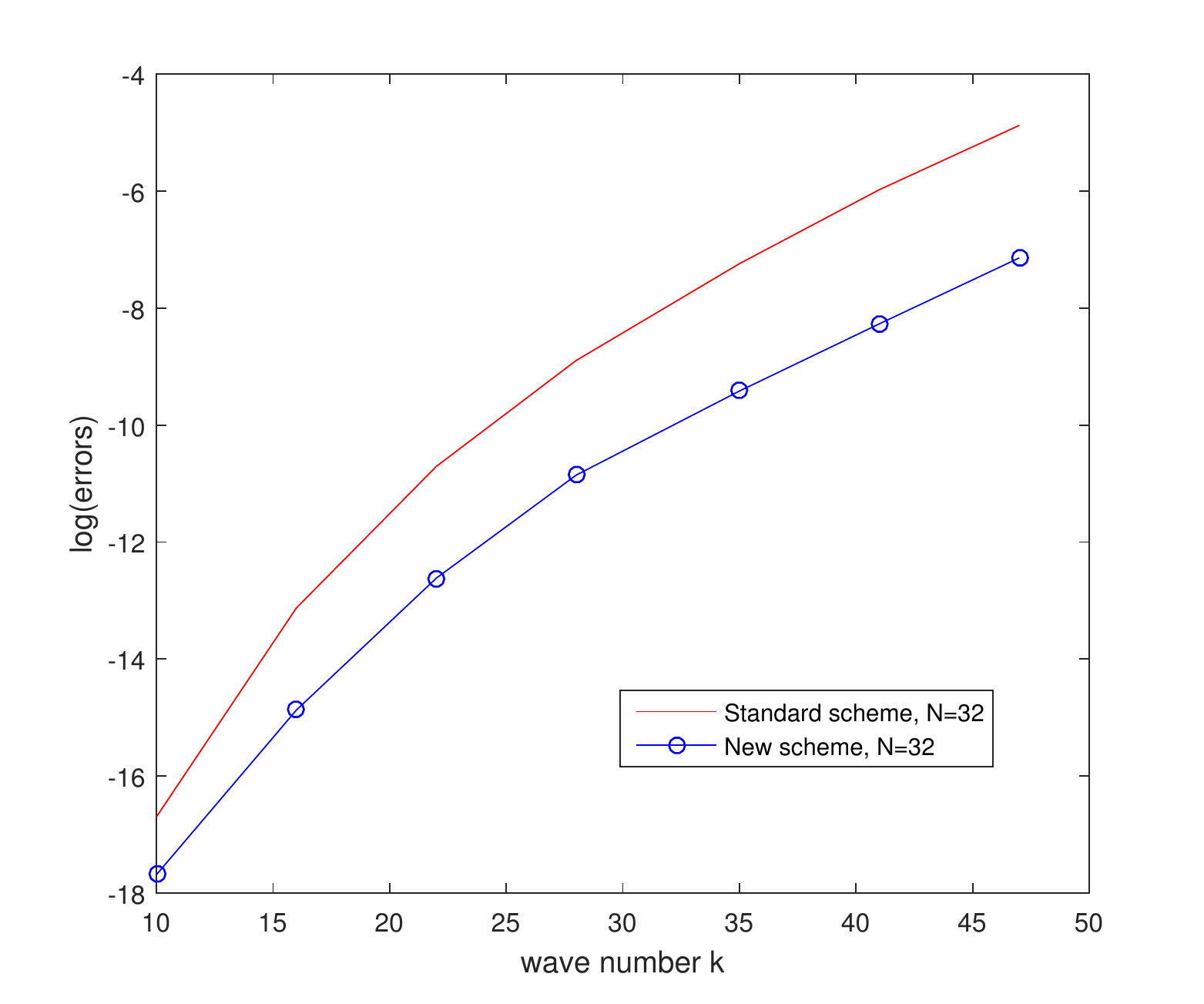}
		\subcaption{}
	\end{minipage}	
	\begin{minipage}{.5\textwidth}
		\centering
		\includegraphics[width=1.0\linewidth]{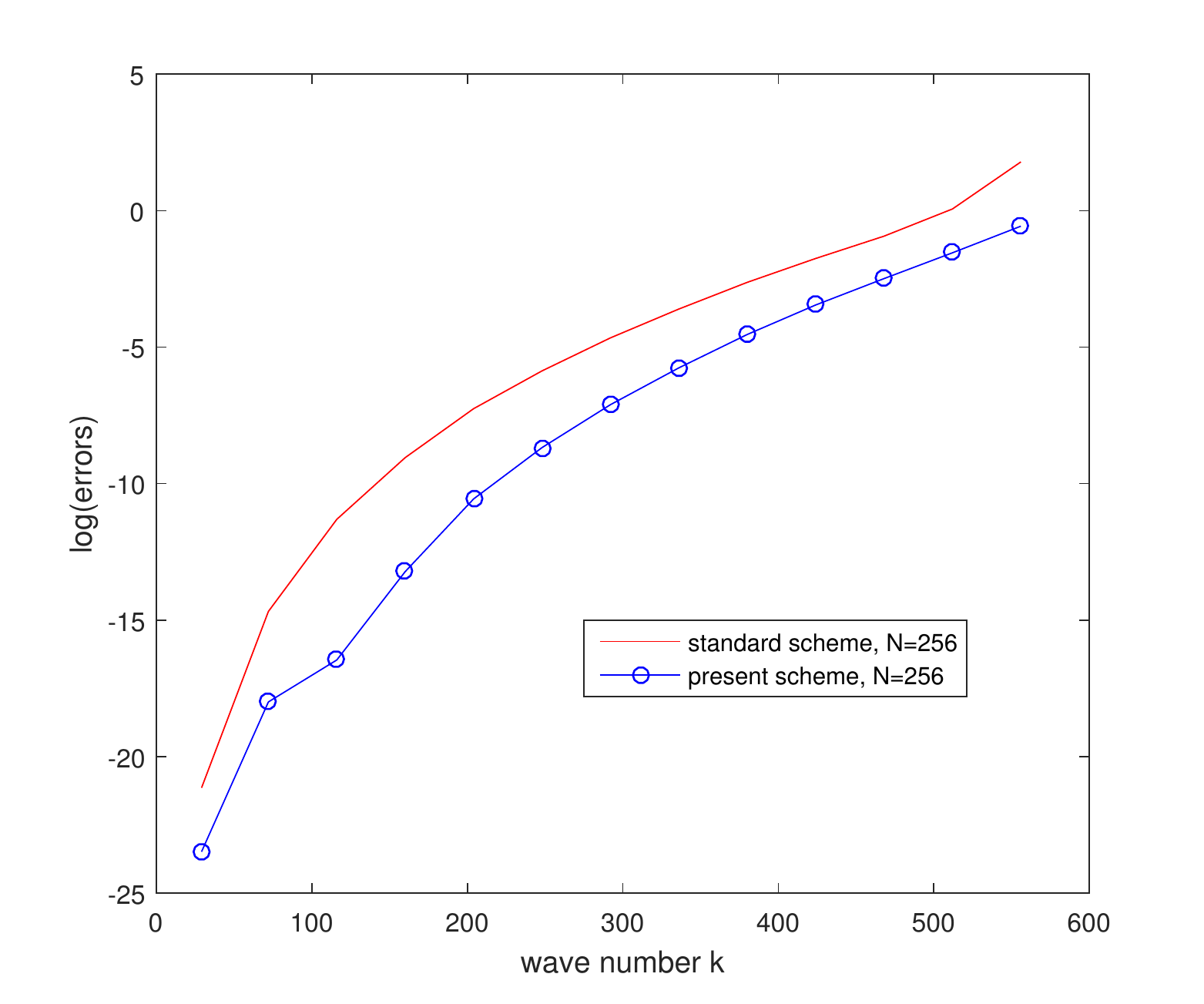}
		\subcaption{}
	\end{minipage}
	\caption{Numerical errors for (a) N=32, (b) N=256 with different $K$ for Problem \ref{p1:3D}.}
	\label{fig:3D1}
\end{figure}
\begin{figure}[]
	\begin{minipage}{.5\textwidth}
		\centering
		\includegraphics[width=1.0\linewidth]{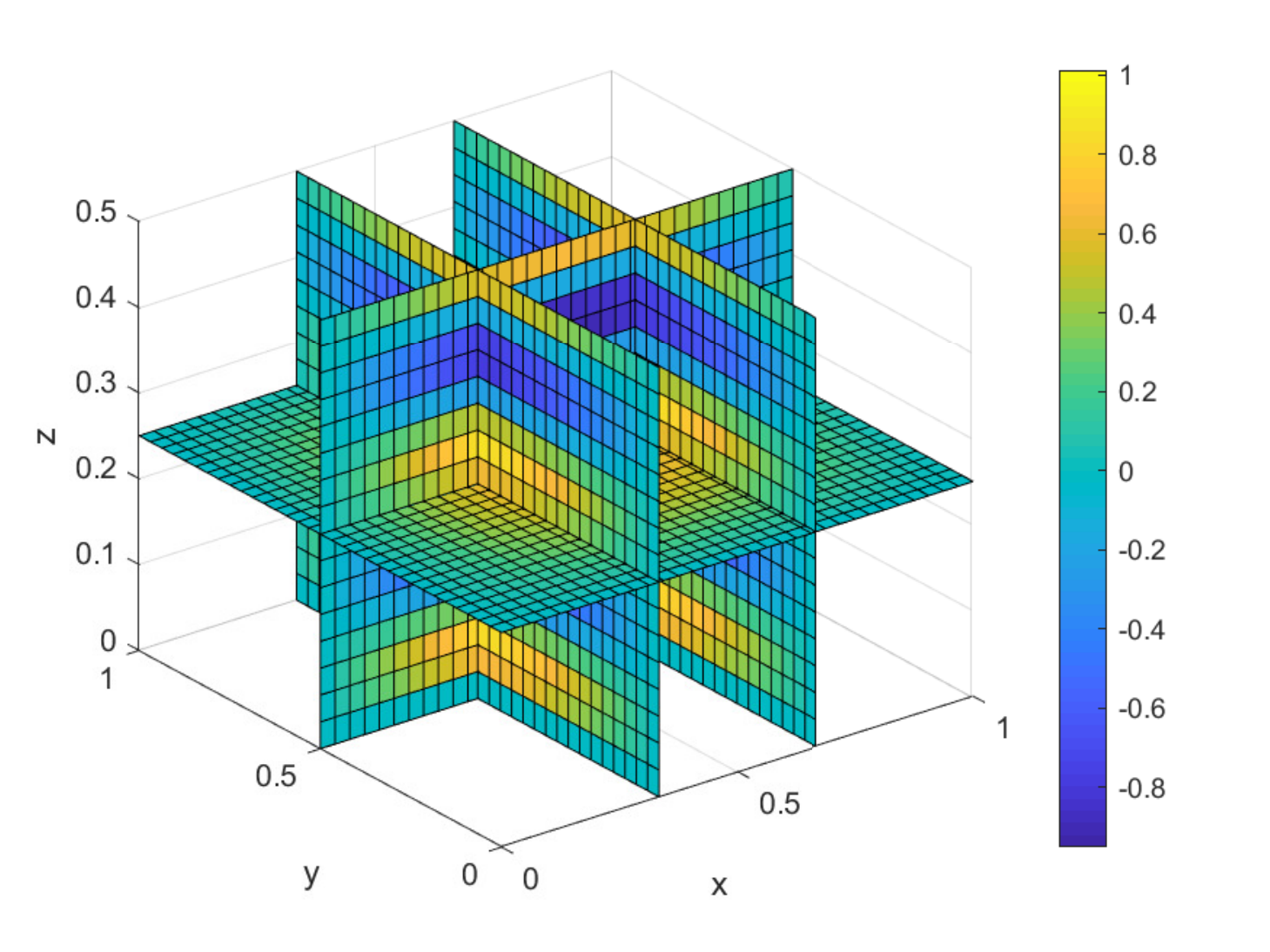}
		\subcaption{}
	\end{minipage}	
	\begin{minipage}{.5\textwidth}
		\centering
		\includegraphics[width=1.0\linewidth]{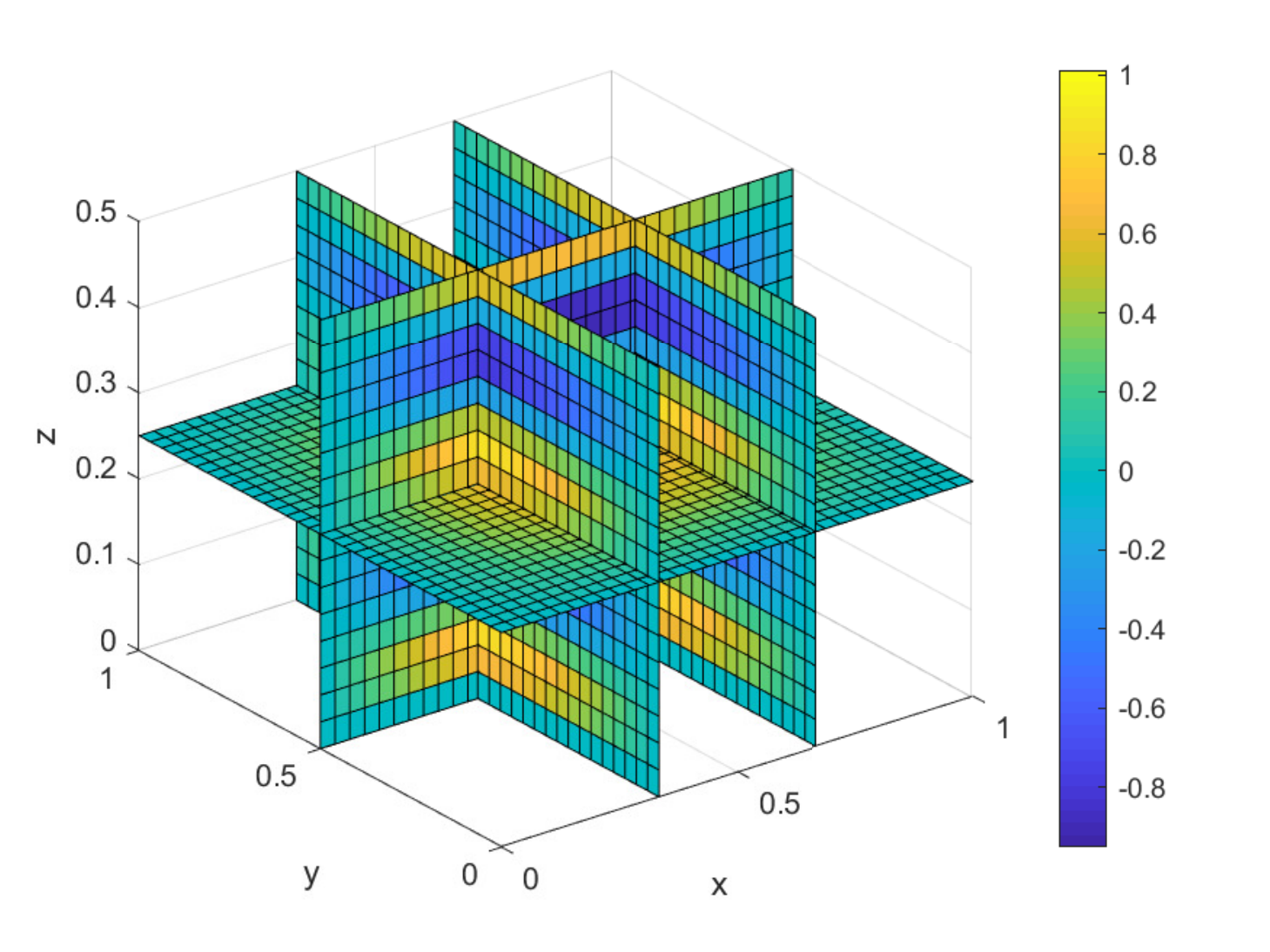}
		\subcaption{}
	\end{minipage}
	\caption{(a) Exact and (b) Numerical solution for $N=32, l=9, K=29$: Problem \ref{p1:3D}.}
	\label{fig:3Dsolex1}
\end{figure}
\begin{problem}\label{p2:3D}
	Consider the Neumann problem from \cite{nabavi2007new}, which is extended into three dimensions as follows
	\begin{equation*}
	(\partial_x^2+\partial_y^2+\partial_z^2)u+K^2u(x,y,z)=(K^2-3\pi^2)\cos(\pi x)\sin(\pi y)\sin(\pi z), (x,y,z)\in \Omega,
	\end{equation*}
	with the Neumann boundary conditions
	\begin{equation*}
	\begin{split}
	\partial_xu|_{x=0}=0, u(1,y,z)=-\sin(\pi y)\sin(\pi z),0\leq y,z\leq 1,\\
	u(x,0,z)=u(x,1,z)=0, 0\leq x,z\leq 1,
	u(x,y,0)=u(x,y,1)=0, 0\leq x,y\leq 1,
	\end{split}
	\end{equation*}
	where $\Omega=[0,1]^3$ is a cubic domain. The analytic solution for the problem is given by
	\begin{equation*}
	u(x,y,z)=\cos(\pi x)\sin(\pi y)\sin(\pi z).
	\end{equation*}	
\end{problem}
In order to get the sixth order accuracy, it is used sixth order approximation (\ref{NBC9}) for the Neumann boundary. Maximum error norms for both the schemes (\ref{eq3D28}) and (\ref{eq3D17}) are given in Table \ref{tab2:3D} for different grids. For this test case, the tolerance for the iteration stopping criteria is taken $1.0e-13$. Table shows that the scheme (\ref{eq3D28}) maintains its sixth order accuracy and it is also more accurate than (\ref{eq3D17}). The exact and numerical solution values for the mesh grids $N=16$ and wave number $K=10$ are plotted in Figure \ref{fig:3Dsolex2} for the Problem \ref{p2:3D}.
\begin{table}[]\caption{Maximum error norms to the solution values of Problem \ref{p2:3D} for $K= 50$.} \label{tab2:3D}
	\centering
	\vspace{0.1cm}
	\begin{tabular}{{c c c c c c c c c c}}
		\hline
		\multirow{2}{*}{$1/h$} &\multicolumn{2}{c}{new scheme (\ref{eq3D28})}&&\multicolumn{2}{c}{6th order scheme (\ref{eq3D17})}\\[0.05cm]
		\hhline{~--~---}& $||\epsilon||_{\infty}$ & $\Theta_{\infty}$& & $||\epsilon||_{\infty}$ & $\Theta_{\infty}$ \\[0.05cm]
		\hline
		\vspace{0.1cm}
		16 	& 4.42e-07  & -Inf && 1.02e-06  &  -Inf\\[0.05cm]
		32 	& 4.60e-09  & 6.58 && 9.92e-09  & 6.68\\[0.05cm]
		64 	& 5.70e-11  & 6.34 && 1.20e-10  & 6.37\\[0.05cm]
		128	& 8.31e-13  & 6.10 && 2.49e-12  & 5.59\\
		\hline    	  
	\end{tabular}
\end{table}
\begin{figure}[]
	\begin{minipage}{.5\textwidth}
		\centering
		\includegraphics[width=1.0\linewidth]{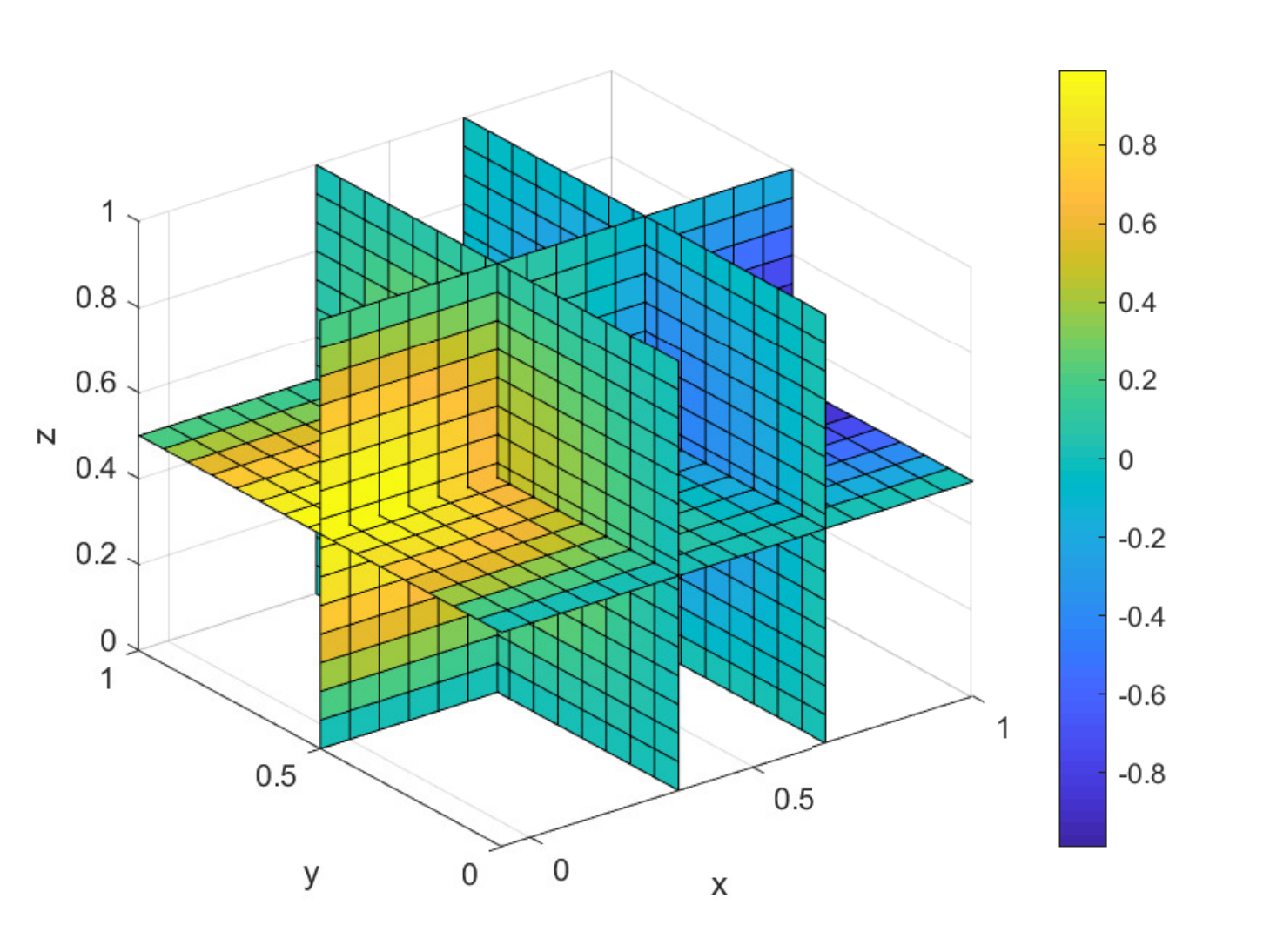}
		\subcaption{}
	\end{minipage}	
	\begin{minipage}{.5\textwidth}
		\centering
		\includegraphics[width=1.0\linewidth]{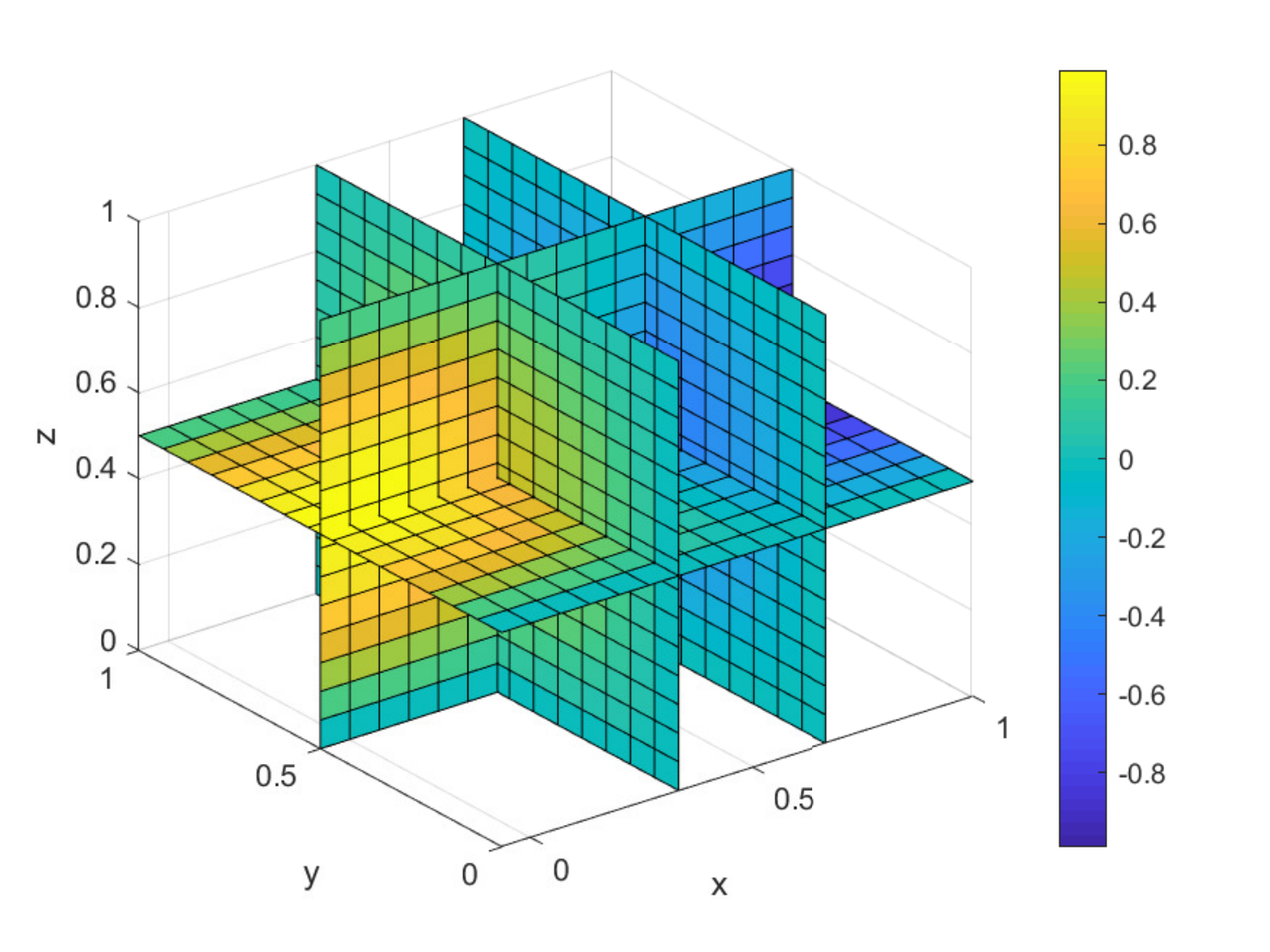}
		\subcaption{}
	\end{minipage}
	\caption{(a) Exact and (b) Numerical solution with $N=16, K=10$ for Problem \ref{p2:3D}.}
	\label{fig:3Dsolex2}
\end{figure}
\begin{problem}\label{p3:3D}
	We consider third 3D test problem with the following analytical solution for the acoustical field \cite{marin2005meshless}
	\begin{equation*}
	u(x,y,z)=\cos(\zeta_1x+\zeta_2y+\zeta_3z),
	\end{equation*}
	for which the source function is given as
	\begin{equation*}
	f(x,y,z)=(\partial_x^2+\partial_y^2+\partial_z^2)u+K^2u=
	(K^2-\zeta_1^2-\zeta_2^2-\zeta_3^2)\cos(\zeta_1x+\zeta_2y+\zeta_3z).
	\end{equation*}
	The Dirichlet Boundary conditions on all sides of the cubic domain $\Omega=\left[-\dfrac12,\dfrac12\right]^3$ can be calculated from the analytic solution.
\end{problem}
Table \ref{tab4:3D} contains $l_2$-error and maximum error norms for the schemes (\ref{eq3D28}) and (\ref{eq3D17}) for the Problem \ref{p3:3D} with the wave number $K= 12$, and parameter values $\zeta1=6, \zeta2=8, \zeta3=7$. Table \ref{tab4:3D} justifies the accuracy of the new scheme in comparison to the scheme (\ref{eq3D17}). The tolerance for the residual error norm is taken $1.0e-12$ for this test case. The exact and numerical solution values are also plotted in Figure \ref{fig:3Dsolex3} for the Problem \ref{p3:3D}.
\begin{table}[]\caption{Error norms to the solution values of Problem \ref{p3:3D} for $K= 12, \zeta1=6, \zeta2=8, \zeta3=7$.} \label{tab4:3D}
	\centering
	\vspace{0.1cm}
	\begin{tabular}{{c c c c c c c c c c}}
		\hline
		\multirow{2}{*}{$1/h$} &\multicolumn{3}{c}{new scheme (\ref{eq3D28})}&&\multicolumn{3}{c}{6th order scheme (\ref{eq3D17})}\\[0.05cm]
		\hhline{~---~---}& $||\epsilon||_{\infty}$& $||\epsilon||_2$ & $\Theta_{\infty}$& & $||\epsilon||_{\infty}$& $||\epsilon||_2$& $\Theta_{\infty}$ \\[0.05cm]
		\hline
		\vspace{0.1cm}
		8 	& 7.44e-03 & 2.63e-03  & -Inf && 3.26e-02 & 1.16e-02  & -Inf\\[0.05cm]
		16 	& 4.83e-05 & 1.18e-05  & 7.27 && 2.98e-04 & 7.26e-05  & 6.78\\[0.05cm]
		32 	& 5.30e-07 & 1.21e-07  & 6.51 && 3.62e-06 & 8.23e-07  & 6.36\\[0.05cm]
		64 	& 7.24e-09 & 1.60e-09  & 6.19 && 5.07e-08 & 1.12e-08  & 6.16\\[0.05cm]
		128	& 1.08e-10 & 2.33e-11  & 6.06 && 7.54e-10 & 1.64e-10  & 6.07\\
		\hline    	  
	\end{tabular}
\end{table}
\begin{figure}[]
	\begin{minipage}{.5\textwidth}
		\centering
		\includegraphics[width=1.0\linewidth]{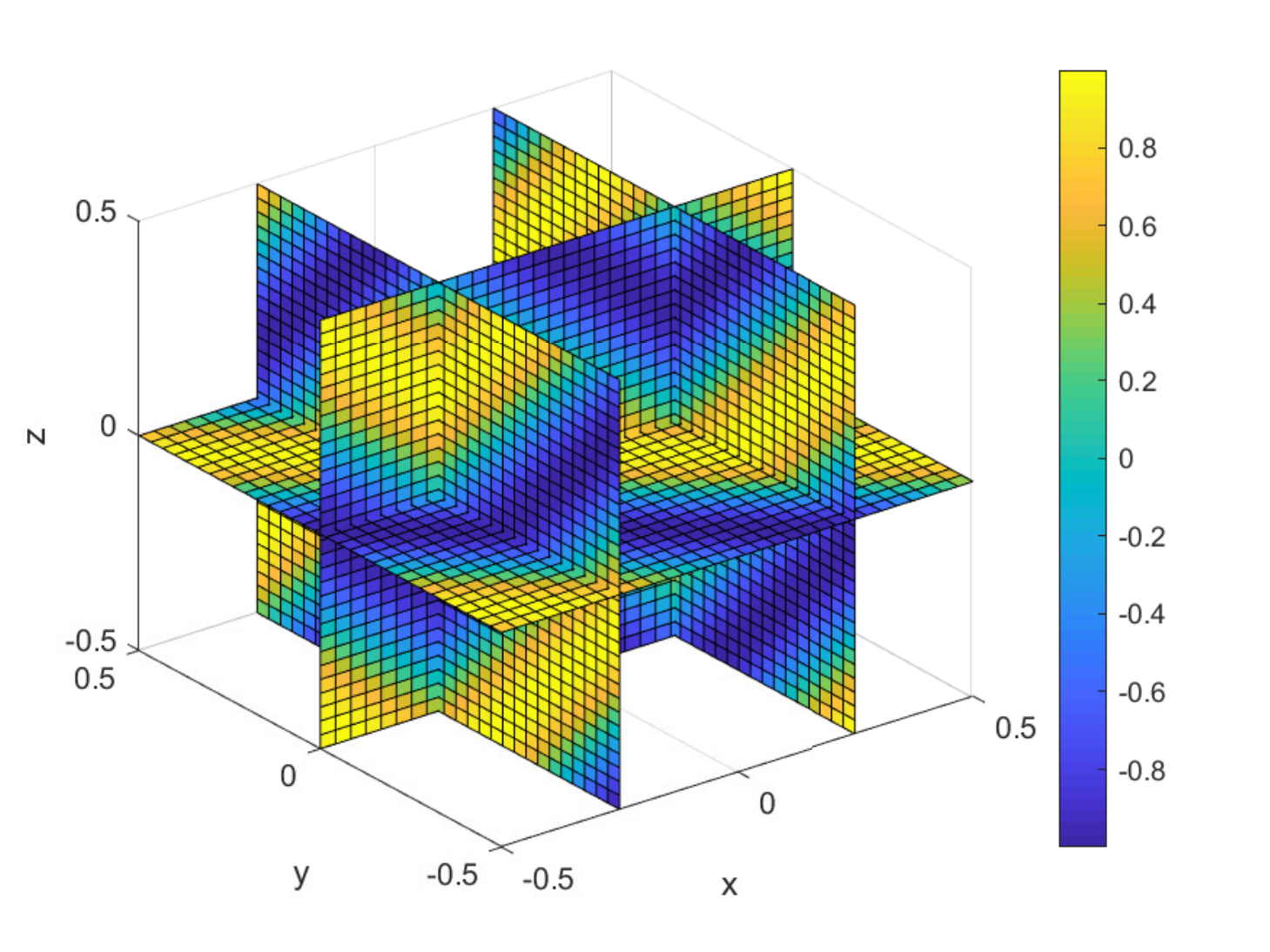}
		\subcaption{}
	\end{minipage}	
	\begin{minipage}{.5\textwidth}
		\centering
		\includegraphics[width=1.0\linewidth]{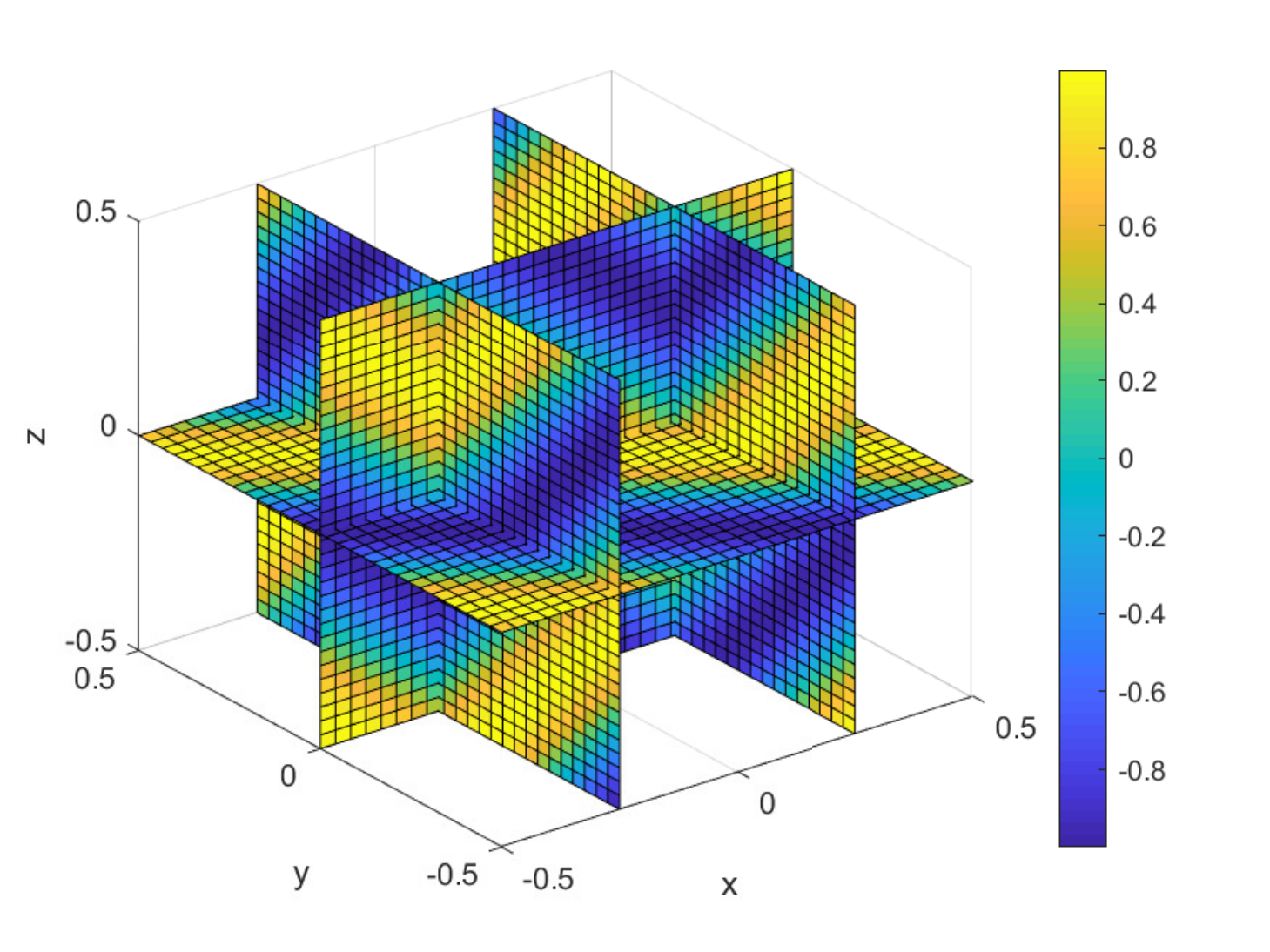}
		\subcaption{}
	\end{minipage}
	\caption{(a)Exact (b)Numerical solution: $N=32, K= 12, \zeta1=6, \zeta2=8, \zeta3=7$:Prob \ref{p3:3D}.}
	\label{fig:3Dsolex3}
\end{figure}
\section{Conclusions}
\label{sec:5}
A new compact sixth-order accurate finite difference scheme for the two and three-dimensional Helmholtz equation is presented. The leading truncation error term of the new scheme does not explicitly depend on the wave number. Thus the new scheme also works for problems with very large wave number $K$. It is also shown that the new scheme is uniquely solvable for sufficiently small $Kh$. Theoretically, it is proved that the bound of the error norm is explicitly independent of the wave number $K$ for the new scheme. Required symbolic derivation is performed using MAPLE with the help of mtaylor command for multi-dimensional Taylor series expansion. The resulting system of equations obtained from difference scheme are solved using BiCGstab(2) iterative method. The new scheme is tested to some model problems governed by the two and three-dimensional Helmholtz equations.  Comparison of the new scheme is done with the standard sixth-order schemes \cite{nabavi2007new,sutmann2007compact}. From the results it is shown that the new scheme is highly accurate for very high wave number. This approach can be extended to derive high-order difference schemes for the two and three dimensional problems with variable wave numbers.
\section*{Acknowledgements}
The research work to the first author is supported by SRM Institute of Science and Technology Kattankulathur, Tamil Nadu India. The authors also acknowledge the SERB, New Delhi, India for the financial support towards computational facility through project file No. MTR/2017/000187. 

\bibliographystyle{unsrt} 
\bibliography{mybibfile}

\end{document}